\documentclass[a4paper,10pt]{article}
\usepackage{a4wide,bm,epsfig,booktabs}

\usepackage{amsmath}
\usepackage{amsthm}
\usepackage{tikz}
\usepackage{setspace}
\usepackage{relsize}  
\usepackage{graphicx}  
\usepackage{cancel} 
\usepackage{slashed}
\usepackage{verbatim} 
\usepackage{enumerate} 
\usepackage{stmaryrd} 
\usepackage{cite}
\usepackage[flushmargin,hang]{footmisc}
\usepackage{fancyhdr} 
\usepackage[arrow, matrix,curve]{xy}
\usepackage{hyperref}
\usepackage{amssymb}
\usepackage{eucal}
\usepackage{scalerel}

\renewcommand{\theenumi}{\arabic{enumi})} 
\renewcommand{\labelenumi}{\theenumi}

\let\origenumerate\enumerate
\def\enumerate{\origenumerate\itemsep0pt}
\let\origitemize\itemize
\def\itemize{\origitemize\itemsep0pt}

\newtheorem*{beisp}{Example}

\newtheorem{theorem}{Theorem}
\newtheorem{proposition}{Proposition}
\newtheorem{lemma}{Lemma}
\newtheorem{corollary}{Corollary}
\newtheorem{remark}{Remark}
\newtheorem{claim}{Claim}

\newtheorem{definition}{Definition}

\newtheorem{convention}{Convention}

\newtheorem{innercustomthm}{Theorem}
\newenvironment{customthm}[1]
  {\renewcommand\theinnercustomthm{#1}\innercustomthm}
  {\endinnercustomthm}

\newcommand{\Sym} {\mathrm{S}}

\newcommand{\pmm} {{+\slash -}}
\newcommand{\mmp} {{-\slash +}}

\newcommand{\conj}  {\mathrm{Conj}}

\newcommand{\len}[1]   {\ci(#1)}
\newcommand{\lenhhh}[1]   {\cn(#1)}
\newcommand{\lent}[1]   {\mathfrak{L}_{#1}}
\newcommand{\homeo}   {\boldsymbol{\upsilon}}
\newcommand{\overlap}   {\mathcal{G}}

\newcommand{\sign}   {\mathrm{sign}}
\newcommand{\defff} {{\it iff }} %Used for definitions
\newcommand{\e} {\mathrm{e}}
\newcommand{\I} {\mathrm{i}}
\newcommand{\ID} {\mathcal{D}}

\newcommand{\UE} {\mathrm{U(1)}}
\newcommand{\DOM} {\mathcal{B}}
\newcommand{\id} {\mathrm{id}}

\newcommand{\MFree} {\mathcal{M}}
\newcommand{\Free} {\mathcal{F}}
\newcommand{\Seg} {\mathcal{S}}

\newcommand{\he} {\hspace{1pt}}

\newcommand{\innt} {\mathrm{int}}
\newcommand{\innnt} {\mathrm{Int}}
\newcommand{\clos} {\mathrm{cls}}

\newcommand{\GGG} {\mathcal{H}}
\newcommand{\GES} {\GGG(S)}
\newcommand{\GESS} {\GGG(S')}
\newcommand{\CM} {C}
\newcommand{\im} {\mathrm{im}}
\newcommand{\dom} {\mathrm{dom}}
\newcommand{\dd} {\mathrm{d}}
\newcommand{\wt}[1] {\widetilde{#1}} 
\newcommand{\ovl}[1]{\overline{#1}} 

\newcommand{\A} {\mathcal{A}}
\newcommand{\KK} {\mathcal{K}}
\newcommand{\LL} {\mathcal{L}}
\newcommand{\OO} {\mathcal{O}}
\newcommand{\UU} {\mathcal{U}}
\newcommand{\RR} {\mathbb{R}}
\newcommand{\CC} {\mathbb{C}}
\newcommand{\ZZ} {\mathbb{Z}}
\newcommand{\CN} {\mathfrak{N}}
\newcommand{\cn} {\mathfrak{n}}
\newcommand{\hcn} {\hat{\mathfrak{n}}}
\newcommand{\ci} {\mathrm{m}}
\newcommand{\CI} {\mathrm{M}_\Sigma}
\newcommand{\ocn} {\ovl{\mathfrak{n}}}
\newcommand{\NN} {\mathbb{N}}

\newcommand{\cnN} {\cn_0}
\newcommand{\cnK} {\cn_{\mathsmaller{<}}}

\newcommand{\hcnN} {\hat{\cn}_0}
\newcommand{\hcnK} {\hat{\cn}_{\mathsmaller{<}}}

\newcommand{\setm} {\backslash}
\newcommand{\cp} {\circ}
\newcommand{\wm} {\varphi}
\newcommand{\g} {{\vec{g}}}
\newcommand{\mg} {\mathfrak{g}}

\newcommand{\SD} {\mathcal{S}}

\newcommand{\SO} {\mathrm{SO}}
\newcommand{\cpsim} {\sim_\cp}

\newcommand{\EE} {\mathrm{E}}
\newcommand{\MK}{{\mathfrak{D}}}
\newcommand{\ML}{{\mathfrak{C}}}

\begin{document}
\title{Decompositions of Analytic 1-Manifolds}
\author{Maximilian Hanusch\thanks{e-mail:
    {\tt mhanusch@math.upb.de}}\\   
  \\
  {\normalsize\em Department of Physics}\\[-0.15ex]
  {\normalsize\em Florida Atlantic University}\\[-0.15ex]
  {\normalsize\em 777 Glades Road}\\[-0.15ex]
  {\normalsize\em FL 33431 Boca Raton}\\[-0.15ex]
  {\normalsize\em USA}}   
%\date{\today}
\date{October 17, 2022}
\maketitle

\begin{abstract} 
In \cite{Deco} analytic 1-submanifolds had been classified w.r.t.\ their symmetry under a given regular and separately analytic Lie group action on an analytic manifold. It was shown that such an analytic 1-submanifold is either free or   (via the exponential map) analytically diffeomorphic to the unit circle or an interval. 
In this paper, we show that each free analytic 1-submanifold is discretely generated by the symmetry group, i.e., naturally decomposes into countably many symmetry free segments that are mutually and uniquely related by the Lie group action. 
This is shown under the same assumptions that were used in \cite{Deco} to prove analogous  decomposition results for  analytic immersive curves. 
Together with the results obtained in \cite{Deco}, this completely classifies 1-dimensional analytic objects (analytic curves and analytic 1-submanifolds) w.r.t.\ their symmetry under a given regular  and separately analytic Lie group action.
\end{abstract}  
\tableofcontents

\section{Introduction}
\label{intro} 
Curves and 1-submanifolds are the key mathematical objects of both  loop quantum gravity \cite{BackLA, Thiemann} and string theory \cite{Polch}, two of the major current approaches to quantum gravity.   
Therefore, in the context of a symmetry reduction of these two theories,  a detailed analysis of the symmetry properties of such 1-dimensional objects should be done. 
For instance, in loop quantum gravity (where the analytic category is used) symmetries of embedded analytic curves  
have turned out to be essential for the investigation of quantum-reduced configuration spaces \cite{InvConLQG, MAX}. The reason is that \cite{Deco} under comparatively mild assumptions on the symmetry action (separately analytic and regular), a given embedded analytic curve is either  
\begingroup
\setlength{\leftmargini}{10pt}
\begin{itemize} 
\item
a free segment (symmetry free),
\item
free, and then 
completely determined by its holonomy along a symmetry free building block,
\item
exponential, and then explicitly given in terms of the exponential map, see Lemma 5.6.1 in \cite{MAX}.
\end{itemize}
\endgroup 
\noindent
Simply put, the statement in the second point is due to a decomposition result for free analytic immersive curves (that holds even under milder assumptions on the Lie group action than regularity) that essentially states that each such curve is build up countably many translates of a symmetry free building block. 
In this paper, we prove an analogous decomposition result for analytic 1-submanifolds that, together with the classification result obtained in\footnote{In the following, we refer to the  arXiv-version of \cite{Deco}; which is identical in content to the published version, but more detailed with improved presentation. Specifically, Theorem 4/Theorem 5 in the arXiv-version corresponds to Corollary 18/Theorem 4  in the published version. Moreover,  in the arXiv-version (for convenience of the reader), an explicit proof of the fact had been added that the two definitions for a free 1-submanifold used in \cite{Deco} are indeed equivalent (see Lemma I in Sect.\ 5.3 in \cite{Deco}).} \cite{Deco} (Theorem 5 in \cite{Deco}),  completely classifies analytic 1-submanifolds w.r.t.\ their symmetry properties under a given regular and separately analytic Lie group action. 

More specifically, let $M$ be a fixed analytic manifold (finite-dimensional, without boundary, and not necessarily second countable). 
An analytic 1-submanifold 
$(S,\iota)$ of $M$ is a   
 connected,  Hausdorff, second countable $1$-dimensional analytic manifold with boundary $S$, together  
with an injective analytic immersion $\iota\colon S\rightarrow M$.  A left action $\wm\colon G\times M\rightarrow M$  of a Lie group $G$ on  $M$ is said to be 
\begingroup
\setlength{\leftmargini}{10pt}
\begin{itemize}
\item[-] 
analytic in $G$\hspace{2.6pt}\qquad\quad \textbf{\defff}\:$\wm_x\colon G\ni g\mapsto \wm(g,x)\in M$ is analytic for each $x\in M$. 
\item[-]
analytic in $M$\hspace{1.9pt}\qquad\:\: \textbf{\defff}\:$\wm_g\colon M\ni x\mapsto \wm(g,x)\in M$ is analytic for each $g\in G$.
\item[-]
separately analytic\:\: \textbf{\defff}\:$\wm$ is analytic in $G$ and $M$. 
\item[-]
sated\hspace{2pt}\qquad\qquad\qquad\: \textbf{\defff}\:to each  $C\subseteq M$ with $|C|\geq 2$ and each $x\in M\setm C$, there exists a neighbourhood $U$
\vspace{-14.5pt}

\hspace{107.5pt}of $x$ with $g\cdot C\subsetneq U$ for all $g\in G$.
\item[-]
non-contractive\hspace{2.6pt}\qquad \textbf{\defff}\:$\wm$ is sated as well as analytic in $M$.
\item[-]
regular\hspace{4pt}\qquad\qquad\quad\: \textbf{\defff}\:$\wm$ is sated  as well as stable.\footnote{We omit the rather technical definition of stability at this point, because we do not explicitly use it in this paper -- It can be found in Definition 1 in \cite{Deco}.} 
\end{itemize}
\endgroup
\noindent
We always assume that $\wm$ is continuous in $G$, i.e., that $\wm_x\colon G\ni g\mapsto \wm(g,x)$ is continuous for each $x\in M$. 
Notably, $\wm$ is sated if one of the following conditions are fulfilled (see Remark 1 in \cite{Deco}):
\begingroup
\setlength{\leftmargini}{10pt}
\begin{itemize} 
\item
There exists a $G$-invariant continuous metric on $M$.
\item
The map $\wm_x\colon G\ni g\mapsto \wm(g,x)\in M$ is proper for each $x\in M$.
\item
$M$ is a topological group, with 
$\wm(g,x)=\phi(g)\cdot x$ for all $g\in G$ and $x\in M$, 
for a (necessarily continuous) group homomorphism $\phi\colon G\rightarrow M$.\footnote{Actually, it suffices to require that $(M,*)$ is a group such that $*\colon M\times M\rightarrow M$ is continuous in the first argument (i.e., $M\ni x\mapsto x*y\in M$ is continuous for each fixed $y\in M$).}
\end{itemize}
\endgroup
\noindent
It was shown in \cite{Deco} that if 
$\wm$ is separately analytic and regular, then an analytic curve in $M$ (Theorem 2) or an analytic 1-submanifold of $M$ (Theorem 5) is either exponential or free:
\begingroup
\setlength{\leftmargini}{10pt}
\begin{itemize} 
\item
An analytic curve $\gamma\colon D\rightarrow M$ is said to be 
\vspace{-4pt}
\begingroup
\setlength{\leftmarginii}{10pt}
\begin{itemize}
\item[$\cp$]
exponential \hspace{1pt}\:\textbf{\defff} $\gamma= \exp(\rho\cdot \g)\cdot x$ holds, for some $x\in M$, $\g\in \mg$, and an analytic map $\rho\colon D\rightarrow \RR$. 
\item[$\cp$]
free\hspace{37.7pt}\:\hspace{1pt}\textbf{\defff} it admits a free segment -- i.e., an immersive  subcurve $\delta\equiv \gamma|_{D'}$ such that 
\begin{align*}
	\textstyle\qquad\qquad\qquad\quad\: g\cdot \delta(J) =\delta(J')\qquad\Longleftrightarrow\qquad g\cdot \delta=\delta \qquad\big(\stackrel{\text{analyticity}}{\Longleftrightarrow}\quad g\in G_\gamma:=\bigcap_{t\in D}G_{\gamma(t)}\big) 
\end{align*}
\vspace{-16pt}
\par
\begingroup
\leftskip=2.68cm 
\noindent
with $g\in G$ as well as $J,J'\subseteq D'$ non-empty open intervals  
on which $\delta$ is an embedding.
\par
\endgroup
\end{itemize}
\endgroup
{\it Notably, if $\gamma$ is non-constant exponential, then (under the given assumptions on $\wm$) $\g\in \mg\backslash  \mg_x$ is unique up to rescaling by $\lambda\in \RR_{\neq 0}$ and addition of elements in the Lie algebra $\mg_\gamma$ of the stabilizer $G_\gamma$ of $\gamma$.}
\item
An analytic 1-submanifold $(S,\iota)$ of $M$ is said to be 
\vspace{-4pt}
\begingroup
\setlength{\leftmarginii}{10pt}
\begin{itemize}
\item[$\cp$]
exponential \hspace{1pt}\:\textbf{\defff} it is either 
analytically diffeomorphic to $U(1)$ or to a nondegenerate interval $\ID\subseteq \RR$ via 
\begin{align}
\label{kjkfdkldf}
	\qquad\qquad\qquad\UE\ni \e^{\I \phi}\mapsto \iota^{-1}(\exp(\phi\cdot \g)\cdot x)\in S\qquad\text{or}\qquad \ID\ni t\mapsto \iota^{-1}(\exp(t\cdot \g)\cdot x)\in S,\tag{$\natural$}\\[-19.5pt]\nonumber
\end{align}
\par
\begingroup
\leftskip=2.68cm 
\noindent
respectively, for some $\g\in \mg\backslash \mg_x$. \hspace*{\fill}(See Proposition 8.2 in \cite{Deco}.) 
\par
\endgroup
\item[$\cp$]
free\hspace{37.7pt}\:\hspace{1pt}\textbf{\defff} it admits a free segment -- i.e., a non-singleton connected subset $\emptyset\neq\Sigma\subseteq S$ such that 
\vspace{-3pt}
\begin{align*}
	\textstyle\qquad\qquad\qquad\qquad\hspace{5.5pt} g\cdot \iota(\OO) = \iota(\OO')\qquad\Longleftrightarrow\qquad g\cdot \iota|_\Sigma=\iota|_\Sigma\qquad\big(\stackrel{\text{analyticity}}{\Longleftrightarrow}\quad g\in G_S:=\bigcap_{z\in S}G_{\iota(z)}\big)\quad
\end{align*}
\vspace{-16pt}
\par
\begingroup
\leftskip=2.68cm 
\noindent
with $g\in G$ as well as $\OO,\OO'\subseteq \Sigma$ non-empty connected open subsets of $S$   
on which $\iota$ is an embedding. \hspace*{\fill}(See Lemma I in \cite{Deco}.) 
\par
\endgroup
\end{itemize}
\endgroup
{\it Notably, if $(S,\iota)$ is exponential, then (under the given assumptions on $\wm$, and for $\ID$ fixed in the second case in \eqref{kjkfdkldf}) $\g\in \mg\backslash  \mg_x$ is unique up to addition of elements in  the Lie algebra $\mg_S$ of the stabilizer $G_S$ of $S$.}
\end{itemize}
\endgroup
\noindent
It was furthermore shown in \cite{Deco} that if $\wm$ is non-contractive, then each free analytic immersive curve naturally decomposes into symmetry free subcurves that are mutually and uniquely related to each other by the Lie group  action 
(see Theorem 3 and Theorem 4, Definition 1, and Convention 3 in \cite{Deco}).
In this paper, we prove an analogous result for  analytic 1-submanifolds. Specifically, we show the following statement (see the more comprehensive Theorem \ref{dfsofgofg}):
\begin{customthm}{A}\label{kkdskldslkdskllkdskds}
	Assume that $\wm$ is non-contractive; and let $(S,\iota)$ be free, but not a free segment by itself. Then, $S$ either admits a unique $z$-decomposition or a compact maximal free segment that is contained in the interior $\innt[S]$ of $S$. The first case cannot occur if $S$ is compact without boundary; and, in the second case, $S$ is either positive or negative, and admits a unique $\Sigma$-decomposition for each (necessarily compact) maximal segment $\Sigma\subset \innt[S]$.  
\end{customthm} 
\noindent
To explain these statements in more detail, let us first recall \cite{Gale} that $S$ is either homeomorphic to $U(1)$ (we write $S \cong \UE$) or to some nondegenerate interval $\ID\subseteq \RR$ (we write $S\cong\ID$). Hence, $S$ is compact without boundary \defff $S\cong U(1)$ holds. We need the following definitions:
\begingroup
\setlength{\leftmargini}{10pt}
\begin{itemize} 
\item
Let $\Seg(S)$ denote the set of all segments, i.e., all non-empty, non-singleton, and connected subsets $\Sigma\subseteq S$. 
\item
Write $g\cdot \iota|_\Sigma\cpsim \iota|_{\Sigma'}$ for $g\in G$ and $\Sigma,\Sigma'\in \Seg(S)$ \defff $g\cdot \iota(\OO)=\iota(\OO')$ holds for open segments $\OO\subseteq \Sigma$, $\OO'\subseteq\Sigma'$  such that $\iota|_\OO,\iota|_{\OO'}$ are embeddings. 
\item
Let $\Free(S)\subseteq \Seg(S)$ denote the set of all 
free segments, i.e., all 
$\Sigma\in \Seg(S)$ such that 
\begin{align*}
	\textstyle g\cdot \iota|_\Sigma\cpsim \iota|_\Sigma
	\quad\text{for}\quad 
	g\in G\qquad\quad\Longleftrightarrow\qquad\quad 
	g\in G_S.
\end{align*}
\vspace{-15pt}
\item
Let  
 $\MFree(S)\subseteq \Free(S)$ denote the set of all 
maximal free segments, i.e., all 
$\Sigma\in \Free(S)$ such that $\Sigma\subseteq \Sigma'\in \Free(S)$ implies $\Sigma=\Sigma'$. 
\item
A boundary segment of a segment $\Sigma\in \Seg(S)$ is a segment $\Seg(S)\ni \Sigma'\subseteq \Sigma$ such that $\Sigma\setm\Sigma'$ is connected.
\end{itemize}
\endgroup
\noindent
Assume now first that $\Free(S)\not\ni S\cong \ID$ holds ($S$ is homeomorphic to an interval and not a free segment by itself) and let $z\in \innt[S]$ be fixed. A $z$-decomposition of $S$ is a class $[e]\neq  [g]\in \GES:= G\slash G_S$ with $g\in G_{\iota(z)}$, such that the following two conditions are fulfilled:
\begingroup
\setlength{\leftmargini}{10pt}
\begin{itemize} 
\item[$\circ$]
There exist compact segments $\KK_\pm\subseteq S$ with $g\cdot \iota(\KK_-)=\iota(\KK_+)$ and $\KK_-\cap\KK_+=\{z\}$. 
\item[$\circ$]
$\Sigma_\pm\in \Free(S)$ holds for the unique boundary segments $\Sigma_\pm\in \Seg(S)$  of $S$ with
\begin{align*}
 S=\Sigma_-\cup \Sigma_+,\qquad \Sigma_-\cap\Sigma_+=\{z\}\qquad\text{and}\qquad
  \KK_\pm\subseteq \Sigma_\pm.
\end{align*}
\vspace{-20pt}
\end{itemize}
\endgroup
\noindent
Then, we have $\MFree(S)=\{\Sigma_-,\Sigma_+\}$ (see Lemma \ref{gfhhgh}); and, exactly one of the following two situations holds (see Lemma \ref{dfdffdfddfdf}):
\begin{align*}
(\iota|_{\Sigma_+})^{-1}\cp(g\cdot \iota|_{C_-})\colon \:\:\:\Sigma_-\supset C_-&\rightarrow \Sigma_+\:\:\:\:\hspace{27.2pt}\text{is an analytic diffeomorphism};\:\:\:\text{hence},
\:\:\: g\cdot \iota(C_-)=\iota(\Sigma_+)\phantom{,}
\\
(\iota|_{C_+})^{-1}\cp(g\cdot \iota|_{\Sigma_-})\colon\hspace{33.9pt} \Sigma_-&\rightarrow C_+\subseteq \Sigma_+\:\:\:\:\text{is an analytic diffeomorphism};\:\:\:\text{hence},
\:\:\: g\cdot \iota(\Sigma_-)=\iota(C_+)
\end{align*}
for segments $C_\pm\in \Free(S)$ with $\KK_\pm\subseteq C_\pm\subseteq \Sigma_\pm$. In addition to that, we have the implication 
\begin{align}
\label{sjkdajkdiudsaisdahds}
g'\cdot \iota|_{\Sigma_-}\cpsim\iota\quad\text{for}\quad g'\in G\qquad\quad\Longrightarrow\qquad\quad [g']\in \{[e],[g]\},\tag{$\flat$}
\end{align}
so that a proper translate of $\iota(\Sigma_-)$ can overlap $\iota(S)$ in exactly one way. 
 The corresponding uniqueness statement in Theorem \ref{kkdskldslkdskllkdskds} then means that both the point $z\in S$ and the class $[g]\in \GES$ are uniquely determined. 
\begin{beisp}
Let $\wm\colon \SO(2)\times \RR^2\rightarrow \RR^2$ be the canonical action of $G\equiv \SO(2)$ on $M\equiv\RR^2$. Let furthermore $S\equiv I$ be an open interval containing $0$, and let $\iota\colon I\ni t \mapsto (t,t^3)\in \RR^2$.
Then, $S$ admits the $0$-decomposition $[R(\pi)]$ with  $R(\pi)$ the rotation by the angle $\pi$, whereby $\Sigma_-=I\cap (-\infty,0]$ and $\Sigma_+=I\cap [0,\infty)$ holds. 
Here, we can choose $\KK_-:=[-\varepsilon,0]$ and $\KK_+:=[0,\varepsilon]$ for each $\varepsilon>0$ with $\pm\varepsilon \in I$.
\end{beisp}
\noindent
Next, assume that $S\notin\Free(S)$ admits a compact maximal segment $\MFree(S)\ni\Sigma\subset\innt[S]$: 
\begingroup
\setlength{\leftmargini}{15pt}
{
\renewcommand{\theenumi}{{\bf\arabic{enumi})}} 
\renewcommand{\labelenumi}{\theenumi}
\begin{enumerate}
\item
\label{asasaasfdj1}
If $S\cong\UE$ holds ($S$ is compact without boundary), then by a $\Sigma$-decomposition of $S$ (see Definition \ref{gfgfgf}) we understand a (finite) collection of compact maximal segments $\Sigma_0,\dots,\Sigma_n\in \MFree(S)$ as well as classes $[g_0],\dots,[g_n]\in \GES$  with $n\geq 1$, such that the following conditions are fulfilled:
\begingroup
\setlength{\leftmarginii}{10pt}
\begin{itemize}
\item
$\Sigma_0=\Sigma$ and $S=\bigcup_{k=0}^n\Sigma_k$, as well as $g_k\cdot \iota(\Sigma_0)=\iota(\Sigma_k)$ for $k=0,\dots,n$.
\vspace{1pt}
\item
If $n=1$, then $\Sigma_0\cap\Sigma_1=\{z_+,z_-\}$ consists of the two boundary points of $\Sigma_0$.
\vspace{1pt}
\item
If $n\geq 2$, then $\Sigma_p\cap \Sigma_q$ is singleton for $|p-q|\in \{1,n\}$, and empty for $2\leq |p-q|\leq n-1$.
\end{itemize}
\endgroup
\noindent 
It follows that the classes $[g_0],\dots,[g_n]$ are mutually different, with $[g_0]=[e]$ as well as $[g_1]=[g_1^{-1}]$ if $n=1$. 
Moreover, in analogy\footnote{Notably,  since $S\cong \UE$ holds, Lemma \ref{compnobound} shows that $g\cdot \iota|_{\Sigma_0}\cpsim \iota$ for $\Sigma_0\in \Seg(S)$, already implies $g\cdot \iota(S)=\iota(S)$ as well as $g\cdot \iota(\Sigma_0)=\iota(\Sigma')$ for $\Seg(S)\ni \Sigma' :=\iota^{-1}(g\cdot \iota(\Sigma_0))$.}  to  \eqref{sjkdajkdiudsaisdahds}, for $\Sigma'\in \Seg(S)$ and $g\in G$, we have 
the implication (Lemma \ref{asdhlkdsajkd}):
\begin{align*}
 g\cdot \iota(\Sigma_0)=\iota(\Sigma')\qquad\Longrightarrow\qquad [g]=[g_k]\:\:\text{ as well as }\:\:\Sigma'=\Sigma_k\:\:\text{ holds for }\:\: k\in \{0,\dots,n\}\:\:\text{ unique}.
\end{align*} 
In particular, for $n=1$, there exists no other $\Sigma$-decomposition of $S$; and, for $n\geq 2$, the only other $\Sigma$-decomposition of $S$ is given by 
\begin{align}
\label{nmdsnmdkjdskjdskjkjdsiuds98ds98ds98ds98ds1cxccx}
	\ovl{\Sigma}_{k}:=\Sigma_{\zeta(k)}\qquad\text{and}\qquad [\ovl{g}_k]:=[g_{\zeta(k)}]\qquad\quad\forall\: 0\leq k\leq n\tag{$\star$}
\end{align}
with $\Sym_{n+1}\ni \zeta\colon \{0,\dots,n\}\rightarrow \{0,\dots,n\}$
defined by $\zeta(0):=0$ as well as $\zeta(k):=n-(k-1)$ for $k=1,\dots,n$. 

Finally, each $\Sigma'\in \MFree(\Sigma)$ is compact (and properly contained in $S$). Moreover, the number of segments that occur in a $\Sigma'$-decomposition of $S$ is the same (thus, equal to $n$) for each $\Sigma'\in \MFree(\Sigma)$ (see Lemma \ref{ghdgddgf}). 
\item
\label{asasaasfdj2}
If $S$ is homeomorphic to an interval, then by a $\Sigma$-decomposition of $S$ (see Definition \ref{deco}), we understand  
 a pair $(\{\Sigma_n\}_{n\in \cn},\{[g_n]\}_{n\in \cn})$, with 
\begin{align*}
\cn=\{n\in \ZZ\: | \: \cn_{-} \leq n \leq \cn_+\}\qquad\text{for certain}\qquad \cn_-,\cn_+ \in \ZZ_{\neq 0}\cup \{-\infty,\infty\}, 
\end{align*}
that consists of classes 
$\{[g_n]\}_{n\in \cn}\subseteq \GES$ as well as segments $\{\Sigma_n\}_{n\in \cn}\subseteq \Seg(S)$ on which $\iota$ is an embedding such that the following holds:
\begingroup
\setlength{\leftmarginii}{10pt}
\begin{itemize}
\item
$\Sigma_0=\Sigma$, and $\Sigma_{\cn_\pm}$ is a boundary segment of $S$ if $\cn_{\pm}\neq \pm\infty$ holds.
\item
$\Sigma_p\cap\Sigma_q$ is singleton for $|p-q|=1$, and empty for $|p-q|\geq 2$.
\item
$g_n\cdot \iota(\Sigma_0)=\iota(\Sigma_n)$ holds for all $\cn_-<n<\cn_+$, and we have
\begin{align*} 
	g_{\cn_{-}}\cdot \iota(\Sigma_{-})=\iota(\Sigma_{\cn_-})\quad\text{if}\quad \cn_-\neq -\infty\qquad\quad\text{as well as}\qquad\quad g_{\cn_{+}}\cdot \iota(\Sigma_{+})=\iota(\Sigma_{\cn_+})\quad\text{if}\quad \cn_+\neq \infty
\end{align*}
for certain boundary segments $\Sigma_\pm$ of $\Sigma_0$. 
\end{itemize}
\endgroup
\noindent
Then, (see Remark \ref{dfxhbghg}) $\Sigma_n$ is free for each $n\in \cn$, as well as compact maximal if 
$g_n\cdot \iota(\Sigma_0)=\iota(\Sigma_n)$ holds (e.g.\ if $\cn_-<n<\cn_+$). Additionally:
\begingroup
\setlength{\leftmarginii}{10pt}
\begin{itemize}
\item
$S=\bigcup_{n\in \cn}\Sigma_n$ holds, as well as $[g_0]=[e]$.
\vspace{2pt}
\item
$[g_m]=[g_n]$ for $m\neq n$, implies $-\infty<\cn_-<\cn_+<\infty$ and $m,n\in \{\cn_-,\cn_+\}$.
\vspace{2pt}
\item 
Let $g\in G$, and let  $\Seg(S)\ni\OO,\OO'\subseteq S$ be open such that $\iota|_\OO,\iota|_{\OO'}$ are embeddings, with $\OO\subseteq \Sigma_0$. Then, the following implication holds (Lemma \ref{dfggfgf}):
\begin{align*}
 (g\cdot \iota|_{\Sigma_0})(\OO)=\iota(\OO')\qquad\Longrightarrow\qquad [g]=[g_n]\:\:\text{ and }\:\:\OO'\subseteq\Sigma_n\:\:\text{ holds for }\:\: n\in \cn\:\:\text{ unique}.
\end{align*} 
\end{itemize}
\endgroup
\noindent
In particular, a $\Sigma$-decomposition is unique up to a reordering of the form (see Lemma \ref{dfggfgf}) 
\begin{align}
\label{nmdsnmdkjdskjdskjkjdsiuds98ds98ds98ds98ds1cxccxasassa0909}
(\{\ovl{\Sigma}_n\}_{n\in \ocn},\{[\ovl{g}_n]\}_{n\in \ocn})\quad\text{with}\quad \ovl{\Sigma}_n=\Sigma_{-n}\quad\text{and}\quad [\ovl{g}_n]=[g_{-n}]
\quad\text{for all}\quad n\in \ocn:=\{-n \: | \: n\in \cn\}.\tag{$\star\star$}
\end{align}
\end{enumerate}}
\endgroup
\noindent
It remains to explain the notions ``positive'' and ``negative'' in Theorem \ref{kkdskldslkdskllkdskds}. For this, let $\MFree(S)\ni \Sigma\subset S$ be compact maximal with boundary points $\{z_-,z_+\}$; and let  
$[e]\neq[g]\in \GES$ as well as $z\in \{z_-,z_+\}\cap \innt[S]$ be fixed.  
Assume that $g\cdot \iota(\Sigma_b)=\iota(\Sigma_z)$ holds, with $g\cdot \iota
(b)=\iota(z)$ for
\begingroup
\setlength{\leftmargini}{12pt}
\begin{itemize}
\item[$-$] 
 a boundary segment $\Sigma_b$ of $\Sigma$ with $b\in \Sigma_b\cap \{z_-,z_+\}$,\footnote{Hence, $b$ is a common boundary point of $\Sigma$ and $\Sigma_b\subseteq \Sigma$.}
\item[$-$] 
a segment 
$\Sigma_z$ with $\Sigma\cap\Sigma_z=\{z\}$.
\end{itemize}
\endgroup 
\noindent
Then, $\Sigma$ is said to be 
\begingroup
\setlength{\leftmargini}{10pt}
\begin{itemize}
\item
\hspace{2.4pt}positive if $g\notin G_{\iota(z)}$ holds; hence, if $g$ shifts $\iota(\Sigma_{b})$ such that $\iota(b)\neq \iota(z)$ is mapped to $\iota(z)$.  
\item
negative if $g\in G_{\iota(z)}$ holds; hence, if $g$ flips $\iota(\Sigma_{b})$ at $\iota(b)=\iota(z)$.
\end{itemize}
\endgroup
\noindent
It follows that $\Sigma$ is either positive or negative (Remark \ref{dsdsdsdskjdsjkdskjdskjdsds09w0909ew09we09ewewewewewewewcx}); and furthermore that (Lemma  \ref{kjdfdfjkdfjkldfjkldfjkldfjkldfjklfk}): 
\begingroup
\setlength{\leftmargini}{10pt}
\begin{itemize}
\item
If $S$ admits a positive segment $\Sigma\subset \innt[S]$, then  
each compact $\Sigma'\in \MFree(S)$ is positive:

$S$ is said to be positive in this case. 
\item
If $S$ admits a negative segment $\Sigma\subset \innt[S]$, then  
each compact $\Sigma'\in \MFree(S)$ is negative:

$S$ is said to be negative in this case.
\end{itemize}
\endgroup
\noindent
Let now $\MFree(S)\ni \Sigma\subset\innt[S]$ be compact maximal; so that we are in the situation of Point \ref{asasaasfdj1} ($S\cong\UE$) or in the situation of Point \ref{asasaasfdj2} ($S\cong\ID$):  
\begingroup
\setlength{\leftmargini}{13pt}
\begin{itemize}
\item[$\blacktriangleright$]
If $\Sigma$ is positive, then we have (see Proposition \ref{sadfpoifdsjk}) 
\begin{align*}
(S\cong\UE)\qquad\textbf{\ref{asasaasfdj1}}\colon\quad\:[g_k]&=[h^k]\quad\text{for}\quad k=0,\dots,n\qquad\text{for each fixed}\qquad h\in[g_1].\hspace{69.5pt}\\[3pt]
(S\cong \ID)\hspace{12.2pt}\qquad\textbf{\ref{asasaasfdj2}}\colon\quad\hspace{1.5pt}[g_n]&=[h^n]\quad\hspace{1pt}\forall\: n\in \cn\qquad\qquad\quad\hspace{11.3pt}\text{for each fixed}\qquad h\in[g_1].\hspace{55pt}
\end{align*}
The class $[g_1]$ is up to inversion\footnote{More specifically, this means that each such class either equals $[g_1]$ or $[g_1^{-1}]$, depending on which of the two possible $\Sigma'$-decompositions is considered (see \eqref{nmdsnmdkjdskjdskjkjdsiuds98ds98ds98ds98ds1cxccx} and \eqref{nmdsnmdkjdskjdskjkjdsiuds98ds98ds98ds98ds1cxccxasassa0909}). The details can be found in Sect.\  \ref{osdfuisfdauoifufsda}.} the same for each  $\Sigma'$-decomposition that corresponds to some compact maximal $\MFree(S)\ni \Sigma'\subset\innt[S]$. 
Moreover, the compact maximal (positive) segments are continuously distributed in $S$ (Proposition \ref{lkfdlkfdlkfdlkfdlk}); so that, e.g., each $z\in\innt[S]$ is contained in the interior of some positive segment $\MFree(S)\ni\tilde{\Sigma}\subset \innt[S]$ (Corollary \ref{fdshjfdf}). 
For instance:
\begingroup
\setlength{\leftmarginii}{15pt}
{
\renewcommand{\theenumi}{{\bf\arabic{enumi})}} 
\renewcommand{\labelenumi}{\theenumi}
\begin{enumerate}
\item
Let $M= S\equiv \UE$ (hence, $\iota=\id_{\UE}$), and let  $G\subseteq \UE$ be the discrete subgroup that is generated by the element $h\equiv \e^{\I 2\pi \slash n}$, just acting  via multiplication from the left on $M$. Then, $\Sigma:=\e^{\I [t,t+2\pi \slash n]}$ is positive for each fixed $t\in \RR$, 
and admits the obvious $\Sigma$-decomposition of $S$ with $[g_1]=[h]$.
\item
Let $G\equiv \RR$ act via $\wm(t,(x,y)):=(t+x,y)$ on $M\equiv \RR^2$; and let $S:=\RR$ as well as $\iota\colon  t\mapsto (t,\sin(t))$. Then, $S$ admits the positive segments $[t,t+2\pi]$ with $t\in \RR$, each of them   giving rise to a decomposition of $S$ with $[g_1]=[2\pi]$.  
\end{enumerate}}
\endgroup
\item[$\blacktriangleright$]
If $\Sigma$ is negative, then $\MFree(S)$ consists exactly of the segments  that occur in the (up to inversion unique) $\Sigma$-decomposition of $S$ (Lemma \ref{fhjfdsfdd}). Moreover, for $\sigma\colon \ZZ_{\neq 0}\rightarrow \{-1,1\}$ defined by 
\begin{align*}
\sigma(n):=
\begin{cases} 
	(-1)^{n-1} &\mbox{if}\quad n > 0 \\ 
	(-1)^n & \mbox{if}\quad n < 0,
\end{cases} 
\end{align*}
\vspace{-15pt}

we have (see \eqref{sdsdffghhh2} and \eqref{sdsdffghhh1}):
\begin{align*}
(S\cong\UE)\qquad\textbf{\ref{asasaasfdj1}}\colon\quad\:&[g_k]=[g_{\sigma(1)}\cdot {\dots}\cdot g_{\sigma(k)}]\qquad\qquad\hspace{18pt}\forall\: 1\leq k\leq n\quad\:\:\text{for}\quad\:\: g_{-1}:=g_n\\[3pt]
(S\cong\ID)\hspace{12.2pt}\qquad\textbf{\ref{asasaasfdj2}}\colon\quad\:&[g_n]=[g_{\sigma(\sign(n))}\cdot {\dots}\cdot g_{\sigma(n)}]\qquad\quad\hspace{6pt}\forall\: n\in \cn, 
\end{align*}
e.g.\ $[g_2]=[g_1\cdot g_{-1}]$ and $[g_3]=[g_1\cdot g_{-1}\cdot g_1]$. It follows that in the situation of \ref{asasaasfdj1} ($S\cong \UE$), the integer $n$ is necessarily odd (Corollary \ref{odsosdopds}). 
For instance:
\begingroup
\setlength{\leftmarginii}{15pt}
{
\renewcommand{\theenumi}{{\bf\arabic{enumi})}} 
\renewcommand{\labelenumi}{\theenumi}
\begin{enumerate}
\item
Let $S\equiv U(1)\subseteq M\equiv \RR^2\cong\CC$: 
\begingroup
\setlength{\leftmarginiii}{12pt}
\begin{itemize}
\item[$-$] 
Let  $G\subseteq \mathrm{O}(2)$ be the discrete group that is generated by the reflection at the $x_2$-axis. Then, $\Sigma_0=\e^{\I K_0}$ and $\Sigma_1=\e^{\I K_1}$ are negative, for $K_0=[-\pi/2,\pi/2]$ and $K_1=[\pi/2, 3\pi/2]$.  
\item[$-$]
Let $G\subseteq \mathrm{O}(2)$ be the discrete group that is generated by the reflections at the $x_1$- and the $x_2$-axis. Then, $\Sigma_p=\e^{\I K_p}$ is negative, for $K_p=[p\cdot \pi/2, (p+1)\cdot \pi/2]$ for $p=0,\dots,3$. 
\end{itemize}
\endgroup
In both cases, the above formula for the classes $[g_p]$ is easily verified.
\item
Let the euclidean group $G\equiv \RR^2 \rtimes \SO(2)$ act canonically on $M\equiv\RR^2$; and let $\iota\colon S\equiv\RR\ni t\mapsto (t,\sin(t))$. Then, $\Sigma=[0,\pi]$ is negative, with $[g_{-1}]$ and $[g_{1}]$ the classes of the rotations by the angle $\pi$ around $(0,0)$ and $(\pi,0)$, respectively.
\end{enumerate}}
\endgroup
\end{itemize}
\endgroup
\noindent
This paper is organized as follows. In Sect.\ \ref{gffggf}, we fix the notations and collect the basic facts and definitions that we need for our argumentation in Sect.\ \ref{oisdffiodusfdiu}. In Sect.\ \ref{nmnbvccnnvcncvnmcncv}, 
  we prove some elementary statements concerning compact maximal  segments that we need in Sect.\ \ref{kjdskjsdkjsdoidsoidsds98s98ds98ds9898ds09ds09ds09dssdds}  to show that each  compact $\MFree(S)\ni \Sigma\subset\innt[S]$ admits a $\Sigma$-decomposition of $S$ ($\wm$ non-contractive). In Sect.\ \ref{lkjkjlfdjfdkjfkjfdlkfdj} and Sect.\ \ref{jhdfdfdfkfjdjkdf},  we show that $S$ either admits a unique $z$-decomposition or a compact maximal segment, provided that 
  $S$ is not a free segment by itself ($\wm$ non-contractive).  
  Sect.\ \ref{asasaghhgfg} contains a detailed analysis of the positive and the negative case.

\section{Preliminaries}
\label{gffggf}
In this section, we fix the notations and provide some basic facts and definitions concerning connected analytic 1-manifolds with boundary as well as free segments.

\subsection{Conventions}
\label{kdskjsjkdcxccx}
Intervals  are always assumed to be nondegenerate (i.e.\ an interval is a non-empty, non-singleton, connected subsets $D\subseteq \RR$).  If we write $I,J$ or $K,L$ instead of $D$,  we always mean that $I,J$ are open, and that $K,L$ are compact.  
Domains of charts are assumed to be non-empty, open, and connected. 
Manifolds are assumed to be   
finite-dimensional,  
Hausdorff, and analytic. 
If $f\colon M\rightarrow N$ is a differentiable map between the (analytic) manifolds $M$ and $N$, then $\dd f\colon TM\rightarrow TN$ denotes the corresponding differential map between their tangent manifolds. A differentiable map $f\colon M\rightarrow N$ is said to be immersive \defff  the restriction $\dd_xf:=\dd f|_{T_xM}\colon T_xM\rightarrow T_{f(x)}N$ is injective for each $x\in M$. Moreover, $f$ is said to be an embedding \defff $f$ is injective, immersive, as well as a homeomorphism to $f(M)$ when equipped with the subspace topology. 
Unless explicitly stated otherwise, manifolds are assumed to be manifolds without boundary, and we do not assume second countability. Specifically, in what follows: 
\begingroup
\setlength{\leftmargini}{10pt}
\begin{itemize} 
\item
$M$ denotes a finite-dimensional, Hausdorff analytic manifold  without boundary (second countability is not assumed). The same conventions also hold for the manifold structures of Lie groups.
\item
$S,S'$ denote  connected,  Hausdorff, second countable $1$-dimensional analytic manifolds with boundary.
\end{itemize}
\endgroup
\noindent
If $Y\subseteq X$ is a subset of a topological space $X$, then  an accumulation point of $Y$ in $X$ is an element $x\in X$ such that there exists a net $\{y_\alpha\}_{\alpha\in I}\subseteq Y\backslash\{x\}$ with $\lim_\alpha y_\alpha=x$. 
By an accumulation point of a topological space $X$, we understand an accumulation point of $X$ in $X$.
\vspace{6pt}

\noindent
In the whole paper, $\wm\colon G\times M\rightarrow M$   denotes a left action of a Lie group $G$ on a manifold $M$ (hence, $G,M$ are analytic manifolds without boundary, and not necessarily second countable) that is continuous in $G$, i.e., $\wm_x\colon G\ni g\mapsto \wm(g,x)$ is continuous for each fixed $x\in M$. We say that $\wm$ is  
\begingroup
\setlength{\leftmargini}{10pt}
\begin{itemize} 
\item
{\bf analytic in $\boldsymbol{M}$} \defff for each $g\in G$, the map
$$\wm_g\colon M\rightarrow M,\qquad  x\mapsto \wm(g,x)$$ 
is analytic. Then, each $\wm_g$ is an analytic diffeomorphism with inverse $\wm_{g^{-1}}$. 
\item
\textbf{non-contractive} \defff $\wm$ is analytic in $M$ such that the following condition holds:
\begingroup
\setlength{\leftmarginii}{17pt}
{
\renewcommand{\theenumii}{{\Roman{enumii}})} 
\renewcommand{\labelenumii}{\theenumii}
\begin{enumerate}
\item[$(*)$]
\label{cond2}
To each $C\subseteq M$ with $|C|\geq 2$ and each $x\notin C$, there exists a neighbourhood $U$ of $x$ with $g\cdot C\nsubseteq U$ for all $g\in G$. 
\end{enumerate}}
\endgroup
\noindent 
\end{itemize}
\endgroup
\noindent
To simplify the notations, we usually write 
\begingroup
\setlength{\leftmargini}{10pt}
\begin{itemize} 
\item
$g\cdot x$ instead of $\wm(g,x)$, for $x\in M$ and $g\in G$.
\item
$g\cdot \iota$ instead of $\wm_g\cp \iota$, for  $g\in G$ and $\iota\colon Z\rightarrow M$ a map. 
\end{itemize}
\endgroup
\noindent
We denote the stabilizer of $x\in M$ by 
$G_x=\{g\in G\: | \: g\cdot x=x\}$, and set $G_z:=G_{\iota(z)}$ for each $z\in Z$.  Moreover, we define
\begin{align*}
\textstyle G_C:=\bigcap_{x\in C} G_x\quad\:\text{for}\quad\: \emptyset\neq C\subseteq M\qquad\quad\text{as well as}\qquad\quad G_Y:= G_{\iota(Y)}\quad\:\text{for}\quad\: \emptyset\neq Y\subseteq Z.
\end{align*}
The corresponding Lie algebras are denoted by $\mg_x,\mg_C,\mg_Y$, respectively.\footnote{Observe that $G_x,G_C,G_Y$ are Lie subgroups of $G$, as closed in $G$ by continuity of $\wm$ in $G$.} We set $\GES:=G\slash G_S$, and let $[g]$ denote the class of $g\in G$ in $\GES$.  
We furthermore denote $\conj_g\colon G\ni q\mapsto g\cdot q\cdot g^{-1}\in G$.

\subsection{The Inverse Function Theorem}
We now briefly recall some standard facts concerning real analytic mappings.
\begin{theorem}
\label{ofdpofdpfdpof}
Assume that $F\colon \RR^n\supseteq U\rightarrow \RR^n$ is analytic ($U$ open) such that $\dd_p F$ is an isomorphism for some $p\in U$. Then,  there exists an open neighbourhood $V$ of $p$ with $V\subseteq U$, and an open neighbourhood $W\subseteq \RR^n$ of $F(p)$ such that $(F|_V)|^W\colon V\rightarrow W$ is an analytic diffeomorphism.
\end{theorem}
\begin{proof}
	Confer, e.g., Theorem 2.5.1 in \cite{PRIM}.
\end{proof}
\begin{lemma}
\label{dfdfdfdfdf}
Let $m\geq 1$ and $\ell\geq 0$ be given; and assume that $f\colon  \RR^m\supseteq U \rightarrow \RR^{m+\ell}$ is analytic ($U$ open), with $\dd_x f$ injective for some $x\in U$. Then, there exist $V\subseteq U$ open with $x\in V$, $W\subseteq \RR^{m+\ell}$ open with $f(V)\subseteq  W$, and an analytic diffeomorphism $\alpha\colon W\rightarrow W'$ to an open subset $W'\subseteq \RR^{m+\ell}$, such that
\begin{align*}
	\qquad\qquad\qquad(\alpha\cp f|_V)(v)=(v,0)\in \underbrace{\RR^{m}\times \RR^\ell}_{\displaystyle\cong \RR^{m+\ell}}\qquad\quad \forall\: v\in V 
\end{align*}
\vspace{-17pt}

\noindent
holds (with $\RR^0:=\{0\}$ for the case $\ell=0$). 
\end{lemma}
\begin{proof}
This is a straightforward consequence of Theorem \ref{ofdpofdpfdpof} (see also Lemma 2 in \cite{Deco}).
\end{proof}
\begin{corollary}
\label{dfdsasasasassa}
Let $m\geq 1$, $\ell\geq 0$, $U\subseteq \RR^m$ open, $M$ an analytic manifold with $\dim[M]=m+\ell$, and $\iota\colon \RR^m\supseteq U \rightarrow M$ analytic with $\dd_x\iota$ injective for some fixed $x\in U$. Then, there exists an open neighbourhood $V\subseteq U$ of $x$ and an analytic chart $(O,\psi)$ of $M$ with $\iota(V)\subseteq O$, such that
\begin{align*}
	(\psi\cp \iota|_V)(v)=(v,0) \qquad\quad \forall\: v\in V 
\end{align*}  
holds; hence, $\iota|_V$ is an embedding.
\end{corollary}
\begin{proof}
This is straightfoward from Lemma \ref{dfdfdfdfdf} (see also  Corollary 1 in \cite{Deco}).
\end{proof}

\subsection{Analytic Curves}
A curve is a continuous map $\gamma\colon D\rightarrow X$ from an interval  $D$ to a topological space $X$. 
An extension of $\gamma$ is a curve $\wt{\gamma}\colon I\rightarrow X$ that is defined on an open interval $I$ with $D\subseteq I$ such that $\wt{\gamma}|_D=\gamma$ holds. 
If $M$ is an analytic manifold, then a curve $\gamma\colon D\rightarrow M$ is said to be
\begingroup
\setlength{\leftmargini}{12pt}
\begin{itemize}
\item
analytic \defff it admits an analytic extension.
\item
(analytic) immersive \defff it admits an (analytic) immersive extension.
\item
an analytic embedding \defff it admits an analytic immersive extension that is a homeomorphism onto its image equipped with the subspace topology inherited from $M$.
\end{itemize}
\endgroup
\noindent
Analogously, a continuous map (curve) $\wt{\rho}\colon I\rightarrow \RR$ is said to be an extension of the continuous map (curve) $\rho\colon  D\rightarrow \RR$ \defff $I\subseteq \RR$ is an open interval with $D\subseteq I$ and $\rho=\wt{\rho}|_D$. Then, $\rho\colon  D\rightarrow D'\subseteq \RR$ is said to be
\begingroup
\setlength{\leftmargini}{12pt}
\begin{itemize}
\item
analytic (immersive) \defff it admits an analytic (immersive) extension.
\item
an (analytic) diffeomorphism \defff it admits an extension which is an (analytic) diffeomorphism.
\end{itemize}
\endgroup
\noindent
We recall the following statements:
\begin{lemma}
     \label{dfdfdfdfdfd}
	 Let $\gamma\colon I\rightarrow M$, $\gamma'\colon I'\rightarrow M$ be analytic embeddings; and let $x$ be an accumulation point of $\im[\gamma]\cap \im[\gamma']$.\footnote{Hence, $x\in X:= \im[\gamma]\cap \im[\gamma']$ holds; and there exists a net $X\setminus\{x\}\supseteq \{x_\alpha\}_{\alpha\in I}\rightarrow x$ that converges w.r.t.\ the subspace topology on $X$  inherited from $M$.} Then, $\gamma(J)=\gamma'(J')$ holds for some open intervals $J\subseteq I$ and $J'\subseteq I'$ with $x\in  \gamma(J),\gamma'(J')$. 
\end{lemma}
\begin{proof}
	See, e.g., Lemma 5 in \cite{Deco}.
\end{proof} 
\begin{lemma}
	\label{yccxycyycxy}
   	Let $\gamma\colon D\rightarrow M$, $\gamma'\colon D'\rightarrow M$ be analytic embeddings with $\gamma(D)=\gamma'(D')$. Then, $\gamma=\gamma' \cp \rho$ holds for a  (necessarily unique) analytic diffeomorphism $\rho\colon D\rightarrow D'$; specifically,
$$
\rho=(\gamma'|^{\im[\gamma']})^{-1}\cp \gamma\colon D \rightarrow D'.
$$   	
\end{lemma}

\begin{proof}
	See, e.g., Lemma 6 in \cite{Deco}.
\end{proof}
\begin{corollary}
    \label{sfdsfdsfdfdhhju}
	 Let $\gamma\colon I\rightarrow M$, $\gamma'\colon I'\rightarrow M$ be analytic embeddings such that there exist sequences $I\setm \{t\}\supseteq \{t_n\}_{n\in \NN}\rightarrow t\in I$ and $I'\setm\{t'\}\supseteq \{t'_n\}_{n\in \NN}\rightarrow t'\in I'$ with $\gamma(t_n)=\gamma'(t_n')$ for all $n\in \NN$. Then, there exist open intervals $J\subseteq I$ and $J'\subseteq I'$ with $t\in J$ and $t'\in J'$ as well as a (necessarily) unique  analytic diffeomorphism $\rho\colon J'\rightarrow J$ with $\gamma'|_{J'}=\gamma\cp\rho$ and $\rho(t')=t$. In particular, $\gamma'(J')=\gamma(J)$ holds.
\end{corollary}
\begin{proof}
	This is clear from Lemma \ref{dfdfdfdfdfd} and Lemma \ref{yccxycyycxy}.
\end{proof}

\subsection{Analytic 1-Manifolds}
\label{ncxcxnmxnmnmccxxc}
By an \textbf{analytic 1-manifold}, we understand a   connected, Hausdorff, second countable 1-dimensional analytic manifold  with boundary  $S$. We will make the following assumptions on a given chart $(U,\psi)$ of $S$: 
\begingroup
\setlength{\leftmargini}{12pt}
\begin{itemize}
\item
$U$ contains at most one boundary point of $S$, 
\item
If 
 $z\in U$ is a boundary point of $S$, then  $\psi(U)=[0,i)\subseteq \mathbb{H}^1\equiv [0,\infty)$ holds, with $i>0$ and $\psi(z)=0$.
\end{itemize}
\endgroup
\noindent
We recall \cite{Gale} that $S$ is either homeomorphic to $\DOM\equiv \UE$ (we write $S\cong\UE$) or to an interval $\DOM\equiv \ID\subseteq \RR$ (we write $S\cong \ID$) via some fixed homeomorphism that we denote by
\begin{align}
\label{hjdshjdssdsddsds}
	\homeo\colon \DOM\rightarrow S. 
\end{align}
A continuous map $\varrho \colon S\rightarrow S'$ between the analytic 1-manifolds $S,S'$ is said to be
\begingroup
\setlength{\leftmargini}{12pt}
\begin{itemize}
\item
 analytic (immersive) \defff for each chart $(U,\psi)$ of $S$, and each chart $(U',\psi')$ of $S'$ with $\varrho(U)\subseteq U'$, the map 
\begin{align*}
	\rho\equiv \psi'\cp \varrho \cp \psi^{-1}\colon D\equiv \psi(U) \rightarrow \psi'(\varrho(U))\equiv  D' 
\end{align*}
is analytic (immersive), i.e., extends to an analytic (immersive) map $\wt{\rho}\colon D\subseteq I\rightarrow \RR$ ($I$ an open interval). 
\item
an analytic diffeomorphism \defff $\varrho$ is a homeomorphism, and both $\varrho,\varrho^{-1}$ are analytic (necessarily) immersive.
\end{itemize}
\endgroup
\noindent
Let $M$ be an analytic manifold (without boundary). 
An \textbf{analytic 1-submanifold} $(S,\iota)$  of $M$ is an analytic  1-manifold $S$, together with an injective analytic immersion $\iota\colon S\rightarrow M$. Specifically, this means that for each chart 
$(U,\psi)$ of $S$, the curve 
\vspace{-8pt}
\begin{align}
\label{nbcnbcnbvccvcc}
	\gamma_\psi\colon D_\psi\equiv\psi(U)\rightarrow M,\qquad t\mapsto \iota\cp \psi^{-1}(t)
\end{align} 
is analytic immersive, i.e., extends to an analytic immersive curve $\wt{\gamma}_\psi\colon D_\psi\subseteq I_\psi\rightarrow M$. 
\begin{remark}
\label{remrrrnmcxnmcx}
Let $(S,\iota)$ be an analytic 1-submanifold; and let 
 $\wm\colon G\times M\rightarrow M$ be analytic in $M$. Then, the following assertions hold: 
\begingroup
\setlength{\leftmargini}{15pt}
\begin{enumerate}
\item
\label{remrrrnmcxnmcx1} 
 	 $(S, g\cdot \iota)$ is an analytic 1-submanifold of $M$ for each $g\in G$. 
\item
\label{remrrrnmcxnmcx2} 
	If $\gamma\colon D\rightarrow M$ is analytic (immersive), then $g\cdot \gamma$ is analytic (immersive) for each $g\in G$. Moreover, if $\gamma$ is an embedding, then $g\cdot \gamma$ is an embedding for each $g\in G$.\footnote{Both statements will be applied, e.g., to the curves defined in \eqref{nbcnbcnbvccvcc} and their analytic immersive extensions.}
\end{enumerate}
\endgroup
\end{remark}
\noindent
Finally, we observe the following:
\begin{lemma}
\label{jdslkdslk}
Let $(\iota,S)$, $(\iota',S')$ be analytic 1-submanifolds of $M$, with $\iota(S)=\iota'(S')$. If $\varrho\equiv \iota'^{-1}\cp \iota \colon S \rightarrow S'$ is a homeomorphism, then    
$\varrho$ is an analytic diffeomorphism. The assumption holds, e.g., if $\iota,\iota'$ are embeddings.
\end{lemma}
\begin{proof}
	Let $x\in S$ be fixed, and let $(U,\psi)$, $(U',\psi')$ be charts of $S,S'$ with $x\in U$ and $\varrho(U)\subseteq U'$. By Corollary \ref{dfdsasasasassa} (applied to analytic immersive extensions $\wt{\gamma}_\psi, \wt{\gamma}_{\psi'}$ of $\gamma_\psi,\gamma_{\psi'}$),  we can shrink $U,U'$ around $x,\varrho(x)$ such that $\gamma_\psi,\gamma_{\psi'}$ are analytic embeddings. Then, $\iota'\cp\varrho =\iota$ yields 
\begin{align*}
	\gamma_{\psi'}\cp\underbrace{(\psi'\cp\varrho \cp\psi^{-1})}_{=: \: \rho}  = \gamma_\psi\qquad\quad\Longrightarrow\qquad\quad \gamma_{\psi'}(D')=\gamma_\psi(D)\quad\text{ for }\quad D:= \psi(U),\: D':=\rho(D). 	
\end{align*}
\vspace{-8pt}

\noindent
Now, $D'$ is an interval, because $D$ is an interval, and because $\rho$ is continuous and injective. Hence, Lemma \ref{yccxycyycxy} shows that $\rho$ is an analytic diffeomorphism; in particular, $\rho$ extends to an analytic immersion $\wt{\rho}\colon D\subseteq I\rightarrow \RR$. This shows that $\varrho$ is an analytic immersion. Then, interchanging the roles of $(S,\iota)$ and $(S',\iota')$, the same arguments applied to $\varrho^{-1}$ show that $\varrho^{-1}$ is analytic immersion; hence, that $\varrho$ is an analytic diffeomorphism.
\end{proof}

\subsection{Segments}
Let $(S,\iota)$ be a fixed analytic 1-submanifold of $M$. 
A \textbf{segment} is a non-empty, non-singleton, connected subset $\Sigma\subseteq S$. The set of all segments in $S$ is denoted by $\Seg(S)$. We observe the following:
\begingroup
\setlength{\leftmargini}{12pt}
\begin{itemize}
\item
	 $\Sigma\cap\Sigma'\neq \emptyset$ for $\Sigma,\Sigma'\in \Seg(S)$ implies $\Sigma\cup\Sigma'\in \Seg(S)$.
\item
 	Let $S\cong\ID$. Then, $C\subseteq \Sigma\cap \Sigma'$ for $C,\Sigma,\Sigma'\in \Seg(S)$ implies $\Sigma\cap\Sigma'\in \Seg(S)$.
\item
 	Let $S\cong\UE$ and $\Seg(S)\ni \LL\subset S$.   
  	 Then, $C\subseteq \Sigma\cap \Sigma'$ for $\Seg(S)\ni C,\Sigma,\Sigma'\subseteq \LL$, implies $\Sigma\cap\Sigma'\in \Seg(S)$. 
\end{itemize}
\endgroup
\noindent
Let $A\subseteq Z\in \{\RR,\UE\}$ be   a connected subset:
\begingroup
\setlength{\leftmargini}{12pt}
\begin{itemize}
\item
$\innnt[A]\subseteq Z$ denotes the interior of $A$ in the topological sense.  
\item
The elements of $\partial_A:=A\setm \innnt[A]$ are called the boundary points of $A$. 
\end{itemize}
\endgroup
\noindent
Let $\Sigma\in \Seg(S)$ in $S$ be given, and let $\homeo$ be as in \eqref{hjdshjdssdsddsds}:
\begingroup
\setlength{\leftmargini}{12pt}
\begin{itemize}
\item 
The {\bf interior of $\boldsymbol{\Sigma}$} is defined by 
	$$\innt[\Sigma]:=\homeo(\innnt[\homeo^{-1}(\Sigma)]).$$
\item
$\clos[\Sigma]$ denotes the  
closure of $\Sigma$ in $S$ in the topological sense. 
Evidently,
\begin{align*}
	\Sigma\in \Seg(S)\qquad\Longrightarrow\qquad \clos[\Sigma]\in \Seg(S).
\end{align*}
\vspace{-15pt}
\item
The elements of 
$\partial_\Sigma:=\Sigma\setm \innt[\Sigma]$ are called the \textbf{boundary points of $\boldsymbol{\Sigma}$}. Then,
$$
\partial_\Sigma=\partial[\Sigma]\cap \Sigma\qquad\text{holds for}\qquad
\partial[\Sigma]:=\clos[\Sigma]\setm\innt[\Sigma].
$$  
\end{itemize}
\endgroup
\vspace{-12pt}

\noindent
We observe the following:
\begingroup
\setlength{\leftmargini}{12pt}
\begin{itemize}
\item
	The above definitions are independent of the explicit choice of the homeomorphism $\homeo\colon \DOM\rightarrow S$.
\item
	$\innt[\Sigma]\in \Seg(S)$ holds for each $\Sigma\in \Seg(S)$; and,   
	$\innt[\Sigma]$ coincides with the topological interior of $\Sigma$ in $S$ if one of the following two conditions are fulfilled:
\begingroup
\setlength{\leftmarginii}{12pt}
\begin{itemize}
\item	
	 $S\cong \UE$.
\item
     $S\cong \ID$ for $\ID$ an open interval.
\end{itemize}
\endgroup
\vspace{-1pt}
\item
	$\innt[S]$ is the interior of $S$ in the sense of manifolds with boundary. 
\item
	$\Sigma'\subseteq \Sigma$ for $\Sigma',\Sigma\in \Seg[S]$ implies $\innt[\Sigma']\subseteq \innt[\Sigma]$.
\item
$\partial[\Sigma]=\partial_\Sigma$ holds if $\Sigma\in \Seg(S)$ is closed in $S$; e.g., if $\Sigma$ is compact or if $\Sigma=S$ holds. 
\item
$\partial_{\homeo^{-1}(\Sigma)}=\homeo^{-1}(\partial_\Sigma)$ holds for each $\Sigma\in  \Seg(S)$.
\end{itemize}
\endgroup
\noindent
A segment $\Seg(S)\ni \Sigma'\subseteq \Sigma$ is called \textbf{boundary segment} of $\Sigma$ \defff $\Sigma\setm \Sigma'$ is connected (possibly empty).\footnote{Notably, if $S\cong \UE$, then each $\Sigma'\in \Seg(S)$ is a boundary segment of $S$.} 
\begin{remark}
\label{cnmnmcviufeiureiure}
Applying the homeomorphism \eqref{hjdshjdssdsddsds}, we see the following:
\begingroup
\setlength{\leftmargini}{15pt}
\begin{enumerate}
\item
\label{bdfoudnnddhhdhdhdseg34930943094309433434434343}
Let $\Seg(S)\ni \Sigma'\subseteq \Sigma\in \Seg(S)$ with $\partial_{\Sigma}\cap\partial_{\Sigma'}\neq \emptyset$. Then,  $\Sigma\setminus\Sigma'$ is connected (\he$\Sigma'$ is a boundary segment of $\Sigma$).
\item
 \label{bdfoudnnddhhdhdhdseg34930943094309433434434343dsdsds}
Let $\Seg(S)\ni \Sigma \subset S$ be compact.  
Then,   
$\partial[\Sigma]=\partial_\Sigma$ consists of two elements, and $\partial_\Sigma\cap \innt[S]\neq \emptyset$ holds.
\item
\label{bdfoudnnddhhdhdhdseg3493094309430943343443434398ds98ds98ds98dsds}
Let $\Sigma,\Sigma'\in \Seg(S)$ be compact with $\Sigma'\subseteq \Sigma$. Then, $\Sigma'= \Sigma$ is equivalent to $\partial_{\Sigma'}=\partial_\Sigma$.
\item
\label{cnmnmcviufeiureiure1}
Let $\OO,C\in \Seg(S)$ with $\OO$ open and $C\cap \OO\neq \emptyset$. Then, there exists $\UU\in \Seg(S)$ open with $\UU \subseteq C\cap \OO$.
\item
\label{cnmnmcviufeiureiure5a}
Let $C,C'\in \Seg(S)$  with $C\cap C'\nsubseteq \partial_C$. Then, there exists $\OO'\in \Seg(S)$ open with $\OO'\subseteq  C\cap C'$.
\item
\label{cnmnmcviufeiureiure5}
Let $C,C'\in \Seg(S)$  with $C\cap C'\neq \partial_C\cap\partial_{C'}$. Then, there exists $\OO'\in \Seg(S)$ open with $\OO'\subseteq  C\cap C'$.\footnote{This is straightforward from Point \ref{cnmnmcviufeiureiure5a}.} 
\item
\label{cnmnmcviufeiureiure2}
	If $\Seg(S)\ni \LL \subset S$ holds, then $\LL$ is homeomorphic to an interval (see also Remark \ref{dslkdslkdslklkdslkdslkslkslkdsdsds}).
\item
\label{cnmnmcviufeiureiure3}
Let $\Sigma,C\in \Seg(S)$ be given with $C\subset \Sigma$. Then, there exists some $x\in \partial[C]\cap \Sigma$ such that 
	$$\UU\cap (C\setm \{x\})\neq \emptyset\neq    
	\UU\cap  (\Sigma\setm C)\quad\text{holds for each neighbourhood}\:\:\: \UU\:\:\text{of}\:\:\: x.
	$$
(\he{}If $C$ is compact, then we necessarily have $x\in \partial[C]=\partial_C $.\he)
\item
\label{cnmnmcviufeiureiure3b}
Let $\Sigma,C\in \Seg(S)$ be given with $C\subseteq \Sigma$ such that $\Sigma\setm C$ is not connected. Then, 
$$\partial[C]\cap \Sigma=\{x^+,x^-\}\qquad\text{holds for certain}\qquad \Sigma\ni x^+\neq x^-\in \Sigma.$$ 
Moreover, for each neighbourhood $\UU^\pm$ of $x^\pm$, we have 
$$\UU^\pm\cap (C\setm\{x^\pm\})\neq \emptyset\neq\UU^\pm\cap  (\Sigma\setm C).\hspace{10pt}$$
\vspace{-17pt}
\item
\label{cnmnmcviufeiureiure3c}
Let $\KK,\LL'\in \Seg(S)$ be given, with $\KK$ compact and $\KK\subset \LL'$.  Then, there exists some $x\in \partial_\KK\cap \innt[\LL']$ such that for each neighbourhood $\UU$ of $x$, we have 
	$$\UU\cap (\KK\setm\{x\})\neq \emptyset\neq\UU\cap  (\LL'\setm \KK).$$
\vspace{-17pt}
\item
\label{cnmnmcviufeiureiure4}
Let $\KK\in \Seg(S)$ be compact with 
$\KK \subset \innt[S]$. 
 Then, there exist $\UU,\LL\in \Seg(S)$ with $\UU$ open and $\LL$ compact, such that $\KK\subset \UU\subset \LL\subset \innt[S]$ holds.
\item 
\label{cnmnmcviufeiureiure6}
Let $\Seg(S)\ni \Sigma \subset S$ and $z\in \partial_\Sigma\cap \innt[S]$ be given. Then, there exists $\Seg(S)\ni \UU\subseteq \innt[S]$ open with $z\in \UU$ such that $\wt{S}:=\Sigma\cup \UU\subset S$ holds, hence $\wt{S}\in \Seg(S)$ is homeomorphic to an interval by Point \ref{cnmnmcviufeiureiure2}. 
\item 
\label{cnmnmcviufeiureiure6cxcxcxcx}
Let $\LL\in \Seg(S)$ and $(U,\psi)$ be a chart of $S$ with $x\in \LL\cap U$. Then, there exists an interval $D_x$ with $\psi(x)\in D_x\subseteq \psi(\LL\cap U)$ (see also Remark \ref{dsoidsoidsoids8ds98ds98ds98dsdsdsdsds}.\ref{sd0909dssoididsoidsoidsoidsoidsdsdsdsdsewew2hjhhjhj}). 
\end{enumerate}
\endgroup
\end{remark}
\begin{lemma}
\label{dsodsoioidsoidsdsds98ds09ds09dsdsds}
Let $\LL\in \Seg(S)$, $A\subseteq S$ closed, and $O\subseteq S$ open. Then,
\begin{align}
\label{soidoidsoidsiodsdsds0909ds0990sd}
\emptyset\neq\LL\cap O= \LL\cap A\qquad\quad\Longrightarrow\qquad\quad \LL\subseteq O\cap A.
\end{align}
\end{lemma}
\begin{proof}
By the left side of \eqref{soidoidsoidsiodsdsds0909ds0990sd}, 
$M:=\LL\cap O$ is non-empty, as well as closed and open in $\LL$ (w.r.t.\ the subspace topology on $\LL$). Hence, $M=\LL$ holds by connectedness of $\LL$; from which the right side of \eqref{soidoidsoidsiodsdsds0909ds0990sd} is clear. 
\end{proof} 

\begin{remark}
\label{dsoidsoidsoids8ds98ds98ds98dsdsdsdsds}
Let $(U,\psi)$ be a chart of $S$, and let $\LL\in \Seg(S)$ be a segment:
\begingroup
\setlength{\leftmargini}{13pt}
{
\renewcommand{\theenumi}{\arabic{enumi})} 
\renewcommand{\labelenumi}{\theenumi}
\begin{enumerate}
\item
\label{sd0909dssoididsoidsoidsoidsoidsdsdsdsdsewew2}
Let 
$O\subseteq S$ be open with $O\subseteq U$, and let $K\subseteq \psi(O)$ be compact. Then,
\begin{align}
\label{dslkdsdsoiuiuewewnbewnmewds998ds98ds98ds09dsdsds}
\emptyset\neq \psi(\LL\cap O)\subseteq K\qquad\Longrightarrow\qquad \LL\subseteq O.
\end{align}
\vspace{-26pt}
\begin{proof}
$A:=\psi^{-1}(K)\subseteq O$ is compact as $K$ is compact, hence closed in $S$. Then, $A\subseteq O$ (first step) together with the left side of \eqref{dslkdsdsoiuiuewewnbewnmewds998ds98ds98ds09dsdsds} (last two steps) yields
\begin{align*}
\psi(\LL\cap A)=\psi(\LL\cap (O \cap A))=\psi((\LL\cap O) \cap A)=\psi(\LL\cap O)\cap\psi(A)=\psi(\LL\cap O)\cap K=\psi(\LL\cap O)\neq \emptyset,
\end{align*} 
so that Lemma \ref{dsodsoioidsoidsdsds98ds09ds09dsdsds}  shows the claim.
\end{proof}
\vspace{-6pt}
\item
\label{sd0909dssoididsoidsoidsoidsoidsdsdsdsdsewew2hjhhjhj} 
Let $x\in \LL\cap U$. Then, there exists an interval $D_x$ with $\psi(x)\in D_x\subseteq \psi(\LL\cap U)$. 
\vspace{-7pt}
\begin{proof}
Assume that the claim is wrong; and set $B:=\psi(U\setminus(\LL\cap U))$:
\begingroup
\setlength{\leftmarginii}{11pt}
\begin{itemize}
\item
If $x\in \partial[S]$,  then there exists $\psi(U)\ni  \psi(x)=0<k\hspace{3.5pt}\in B$. We set $A:=\psi^{-1}([0,k])$ and $O:=\psi^{-1}([0,k))$.
\item
If $x\in \innt[S]$, then there exist $B\ni k'< \psi(x)< k\in B$. We set $A:=\psi^{-1}([k',k])$ and $O:=\psi^{-1}((k',k))$. 
\end{itemize}
\endgroup
In both cases, we have $x\in \LL\cap A=\LL\cap O\neq \emptyset$ (with $A$ closed as compact, and $O$ open), so that  Lemma \ref{dsodsoioidsoidsdsds98ds09ds09dsdsds} yields $\LL\subseteq O\subseteq U$. Then, $D_x:=\psi(\LL\cap U)=\psi(\LL)$ is an interval with $\psi(x)\in D_x\subseteq \psi(U)$,   
which contradicts the assumption that such an interval does not exist.
\end{proof}
\end{enumerate}}
\endgroup
\end{remark}
\begin{remark}
\label{dslkdslkdslklkdslkdslkslkslkdsdsds}
If $S=\UE$ and $\Sigma\subset S$, then  
\begin{align*}
\Sigma\in \Seg(S)\qquad\Longleftrightarrow\qquad 
\Sigma=\e^{\I (a + D)}\quad\text{for}\quad a\in \RR,\:\: D\subseteq (0,2\pi)\:\:\text{an interval}
\end{align*}
whereby $\partial_\Sigma=\e^{\I \partial_D}$ holds, as well as:
\begingroup
\setlength{\leftmargini}{11pt}
\begin{itemize}
\item
$|\partial_\Sigma|=0$\quad$\Leftrightarrow$\quad\hspace{11.5pt} $\Sigma$ open \quad$\Leftrightarrow$\quad $D$ open
\item
$|\partial_\Sigma|=2$\quad$\Leftrightarrow$\quad $\Sigma$ compact  \quad$\Leftrightarrow$\quad $D$ compact
\item
$|\partial_\Sigma|=1$\quad$\Leftrightarrow$\quad 
$\Sigma$ neither closed nor open\quad$\Leftrightarrow$\quad $D$ neither closed nor open
\end{itemize}
\endgroup
\noindent
In fact, write $\psi\colon \RR\ni \alpha\mapsto \e^{\I  \alpha}\in \UE$: 
\begingroup
\setlength{\leftmargini}{11pt}
\begin{itemize}
\item
If $a\in \RR$, then $\psi|_{a+(0,2\pi)}$ is a homeomorphism onto its  image (open) by the inverse function theorem.
\item
If $S\supset \Sigma\in \Seg(S)$, then there exists $a\in \psi^{-1}(S\setminus \Sigma)$, and we have $\Sigma\subseteq  \im[\psi|_{a+(0,2\pi)}]=S\setminus\{\psi(a)\}$.   
\end{itemize}
\endgroup
\end{remark}

\subsubsection{The 1-Manifold Structure of Segments}
\label{kjdskjdskjdskjdskjdsds0ds09ds09ds09ds98ds}
Each $\Sigma\in \Seg(S)$ is an \emph{analytic 1-manifold}, 
\begingroup
\setlength{\leftmargini}{12pt}
\begin{itemize}
\item[$\cp$]
	with interior equal to $\innt[\Sigma]$,
\item[$\cp$]
	with boundary equal to  
	$\partial_\Sigma$,
\end{itemize}
\endgroup
\noindent
that is embedded into $S$ via the (analytic immersive) inclusion map. Specifically, given $z\in \Sigma$, then  either one of the following situations hold:
\begingroup
\setlength{\leftmargini}{12pt}
\begin{itemize}
\item
\label{sdllds}
$z\in \innt[\Sigma]$:\hspace{8.5pt}    
	There exists a chart $(U,\psi)$ with $z\in U\subseteq \innt[\Sigma]$ 
	and  $\psi(z)=0$.
\item 
$z\in \partial_\Sigma\cap\partial_S$:   
There exists a chart $(U,\psi)$ with $z\in U\subseteq \Sigma$, $\psi(z)=0$ and $\psi(U\cap\Sigma)=[0,i)\subseteq \mathbb{H}^1$ for $i>0$. 
\item 
$z\in \partial_\Sigma\cap \innt[S]$:  
There exists a chart $(U,\psi)$ with $z\in U$, $\psi(z)=0$, as well as 
\begin{align}
\label{dskjkjdskjdsdsiudsiuiudsds8787ds87ds87dsdsdsdsdsds}
\psi(U\cap\Sigma)=[0,i)\subseteq \mathbb{H}^1\qquad \text{and}\qquad  \psi(U\setminus \Sigma)=(i',0)\qquad\text{for certain}\qquad i'<0<i.
\end{align}  
\end{itemize}
\endgroup
\noindent
Charts  as described above are called \textbf{submanifold charts of $\boldsymbol{\Sigma}$ centered at $\boldsymbol{z}$}; and, $\Sigma$ is  equipped with the analytic structure induced by them. 
\begin{remark}
\label{remrrr} 
\begingroup
\setlength{\leftmargini}{15pt}
\begin{enumerate}
\item[]
\item
\label{remrrr1} 
Let $\varrho\colon \Sigma\rightarrow \Sigma'$ be a homeomorphism (analytic diffeomorphism) between segments (analytic 1-manifolds) $\Sigma,\Sigma'\in \Seg(S)$.  Then, we have $\varrho(\innt[\Sigma])=\innt[\Sigma']$ as well as $\varrho(\partial_\Sigma)=\partial_{\Sigma'}$.  
\item
\label{remrrr2} 
Let $\wm\colon G\times M\rightarrow M$ be analytic in $M$,  and let $\Sigma\in \Seg(S)$. 
Then, $(\Sigma,g\cdot \iota|_\Sigma)$ is an analytic 1-submanifold of $M$ for each $g\in G$. Moreover, if $\iota|_\Sigma$ is an embedding, then $g\cdot \iota|_\Sigma$ is an embedding for each $g\in G$.
\end{enumerate}
\endgroup
\end{remark}
\subsubsection{Overlaps of Segments}
Let $(S,\iota)$ be an analytic 1-submanifold of $M$.
\begin{corollary}
\label{fddsfd} 
Let $\wm\colon G\times M\rightarrow M$ be analytic in $M$; and let $\Sigma,\Sigma'\in \Seg(S)$ and $g\in G$ with $g\cdot \iota(\Sigma)=\iota(\Sigma')$ such that
\vspace{-4pt} 
$$\varrho\equiv (\iota|_{\Sigma'})^{-1}\cp (g\cdot  \iota|_\Sigma)\colon \Sigma\rightarrow \Sigma'$$ 
is a homeomorphism (e.g.\ if $\iota|_\Sigma,\iota|_{\Sigma'}$ are embeddings). Then, $\varrho$ is an analytic diffeomorphism. 
\end{corollary}
\begin{proof}
Clear from Remark \ref{remrrr}.\ref{remrrr2} and Lemma \ref{jdslkdslk}.
\end{proof}
\noindent
Next, given $\Sigma,\Sigma'\in \Seg(S)$ and $g\in G$, we write
\vspace{-4pt}
\begin{align}
\label{kjdskjskjdskjsd}
	g\cdot \iota|_\Sigma\cpsim \iota|_{\Sigma'}
	\qquad\stackrel{{\rm def.}}{\Longleftrightarrow}\qquad g\cdot \iota(\OO)= \iota(\OO')
\end{align}
 holds for open segments $\OO\subseteq \Sigma$, $\OO'\subseteq \Sigma'$ on which $\iota$ is an embedding. 
We observe the following:
\begingroup
\setlength{\leftmargini}{12pt}
\begin{itemize}
\item
	In the situation of \eqref{kjdskjskjdskjsd}, we have 
	\begin{align*}
	g\cdot \iota|_\OO=\iota\cp \varrho\qquad\text{for}\qquad\varrho\equiv \iota^{-1}\cp (g\cdot \iota)|_\OO\colon \OO\rightarrow\OO'.
\end{align*}
 If $\wm$ is analytic in $M$, then $\varrho$ is an analytic diffeomorphism by Corollary \ref{fddsfd}. 
\item
	Let $\Sigma,\Sigma'\in \Seg(S)$ be given such that $g\cdot \iota|_\Sigma=\iota\cp \varrho$ holds for a homeomorphism  (analytic diffeomorphism) $\varrho\colon \Sigma\rightarrow \Sigma'$. Then,  Corollary \ref{dfdsasasasassa} implies $g\cdot \iota|_\Sigma\cpsim \iota|_{\Sigma'}$. 
\item
We have the implication
\begin{align}
\label{maxicp}
g\cdot \iota|_{\clos[\Sigma]}\cpsim \iota|_{\clos[\Sigma]}\quad \text{for}\quad g\in G,\:\Sigma\in \Seg(S)\qquad\quad&\Longrightarrow\qquad\quad g\cdot \iota|_\Sigma\cpsim \iota|_\Sigma.
\end{align}
In fact, the implication is clear if $S\cong\ID$ holds; and, in the case $S\cong U(1)$, one only has to take a closer look at the situation $\Sigma=S\setm\{z\}$ with $z\in S$.
\end{itemize}
\endgroup
\noindent
Finally, we define 
\begin{align}
\label{dslkjkjskjds}
	\overlap(S):=\{g\in G\:|\: g\cdot \iota\cpsim \iota\}\qquad\text{and observe that}\qquad g\in \overlap(S)\quad\Longleftrightarrow\quad g^{-1}\in \overlap(S)\qquad\text{holds.}\qquad
\end{align}

\subsection{Basic Properties}
In this section, let $\wm\colon G\times M\rightarrow M$ be a Lie group action that is analytic in $M$; and let $(S,\iota)$ denote an analytic 1-submanifold of $M$.\footnote{Notably, Lemma \ref{fdggfg}, Lemma \ref{gffdg},  Corollary \ref{fdsdsffds} and Lemma \ref{kjdskjdskjdskjds}  also hold without the assumption that $\wm$ is analytic in $M$.}  

\subsubsection{Some Elementary Statements}
We will frequently use the following working lemma (recall the definitions made in \eqref{nbcnbcnbvccvcc}): 
\begin{lemma}
\label{lemma:BasicAnalytt1}
Let $(U,\psi)$, $(U',\psi')$ be charts of $S$; and let $\wt{\gamma}_\psi,\wt{\gamma}_{\psi'}$ be analytic immersive extensions of $\gamma_\psi,\gamma_{\psi'}$. Assume that we are given    
$x\in U$, $x'\in U'$, $g\in G$, as well as sequences
\begin{align*} 
& U\setm\{x\}\supseteq \{x_n\}_{n\in \NN}\rightarrow x\in U\qquad\quad\wedge\qquad\quad U'\setminus\{x'\}\supseteq \{x'_n\}_{n\in \NN}\rightarrow x'\in U'\\[3pt]
&\hspace{149pt} \text{with}\\[3pt]
&\hspace{100pt}	g\cdot\iota(x_n)=\iota(x'_n)\qquad\forall\: n\in \NN.
\end{align*}
\vspace{-19pt}

\noindent
Then, the following assertions hold:
\begingroup
\setlength{\leftmargini}{15pt}
\begin{enumerate}
\item
\label{lemma:BasicAnalytt11}
There exist open intervals $J,J'$ with $t\equiv \psi(x)\in J$, $t'\equiv \psi'(x')\in J'$ such that $\wt{\gamma}_\psi|_{J}$ (hence, $g\cdot\wt{\gamma}_\psi|_{J}$) and $\wt{\gamma}_{\psi'}|_{J'}$ are embeddings,  and that $g\cdot \wt{\gamma}_\psi|_J=\wt{\gamma}_{\psi'}\cp\rho$ holds for an (unique) analytic diffeomorphism $\rho\colon J\rightarrow J'$ with $\rho(t)=t'$, hence $g\cdot \wt{\gamma}_\psi(J)=\wt{\gamma}_{\psi'}(J')$. 
\begingroup
\setlength{\leftmarginii}{17pt}
\begin{itemize}
\item[$(\star)$]
Thus, if $S$ admits no boundary, then there exists an open neighbourhood $U_x$ of $x$ as well as an open subset $\UU\subseteq S$ with $g\cdot \iota(U_x)=\iota(\UU)$.\footnote{The statement does not necessarily hold if $S$ admits a boundary --   To see this, let, e.g., $\wm$ denote the additive action of $G\equiv \RR$ on $M\equiv \RR$, and consider $S\equiv (-\infty,0]$, $\iota\colon S\ni t\mapsto t\in M$, $g\equiv -1$, $x\equiv 0$, $x'\equiv -1$.}	
\end{itemize}
\endgroup
\item
\label{slkdlkdslksddsiusdiusiudds98sd98sd98s9dsdsds}
There exist $\Seg(S)\ni \KK\subseteq U$ and $\Seg(S)\ni \KK'\subseteq U'$ compact with $x\in \KK$ and $x'\in \KK'$ such that
$$
g\cdot\iota(\KK)=\iota(\KK')\qquad\quad\text{as well as}\qquad\quad x_n\in \KK,\: x'_n\in \KK'\qquad\forall\: n\geq N 
$$
holds for some $N\in \NN$; whereby we additionally have:
\begingroup
\setlength{\leftmarginii}{15pt}
{
\renewcommand{\theenumii}{{.\alph{enumii}})} 
\renewcommand{\labelenumii}{{\alph{enumii}})}
\begin{enumerate}
\item
\label{lemma:BasicAnalytt12} 
If $x,x'\in\partial[S]$,\hspace{6pt} then $x\in \partial[\KK]$, $x'\in \partial[\KK']$ holds, and $\KK,\KK'$ are  neighbourhoods of $x,x'$, respectively.
	\item
\label{lemma:BasicAnalytt123} 
	If $x,x'\in \innt[S]$, then $x\in \partial[\KK]$, $x'\in \partial[\KK']$ holds, and $\KK,\KK'$ are neighbourhoods of $x,x'$, respectively.
\item
\label{lemma:BasicAnalytt13}
	If $x\in \partial[S]$, $x'\in \innt[S]$, then $x\in \partial[\KK]$, $x'\in \partial[\KK']$ holds, and $\KK$ is a neighbourhood of $x$.
\item
\label{lemma:BasicAnalytt14}
If $x\in \innt[S]$, $x'\in \partial[S]$, then $x\in \partial[\KK]$, $x'\in \partial[\KK']$ holds, and $\KK'$ is a neighbourhood of $x'$.
\end{enumerate}}
\endgroup
\end{enumerate}
\endgroup
\end{lemma}
\begin{proof}
Part \ref{lemma:BasicAnalytt11} is clear from Corollary  \ref{sfdsfdsfdfdhhju}, because Corollary \ref{dfdsasasasassa} shows that there exist intervals $J,J'\subseteq \RR$ with $t\in J$, $t'\in J'$ such that $\wt{\gamma}_\psi|_J$, $\wt{\gamma}_{\psi'}|_{J'}$ are embeddings. For the supplementary statement $(\star)$ observe that if $S$ admits no boundary, then we can choose $\wt{\gamma}_\psi\equiv \gamma_\psi$ and $\wt{\gamma}_{\psi'}\equiv \gamma_{\psi'}$ right from the beginning; so that  the statement just holds for $U_x:=\psi^{-1}(J)$ and $\UU:=\psi^{-1}(J')$. 
Finally, Part \ref{slkdlkdslksddsiusdiusiudds98sd98sd98s9dsdsds} is immediate from Part \ref{lemma:BasicAnalytt11} as well as the homeomorphism property of $\rho$.
\end{proof}
\begin{lemma}
\label{compnobound}
Let $S\cong\UE$ and assume that $g\cdot \iota\cpsim \iota$ holds for some $g\in G$. Then, we have  
$g\cdot \iota(S)=\iota(S)$. 
In particular, if $\Sigma\in \Seg(S)$ is a segment, then $\Sigma':=\iota^{-1}(g\cdot \iota(\Sigma))\in \Seg(S)$ is a segment; and, $g\cdot \iota|_\Sigma=\iota\cp \varrho$ holds for the analytic diffeomorphism (uniqueness) $\varrho\equiv (\iota|_{\Sigma'})^{-1}\cp (g\cdot \iota|_\Sigma)\colon \Sigma\rightarrow \Sigma'$. 
\end{lemma}
\begin{proof} 
Assume first that $g\cdot \iota(S)=\iota(S)$ holds, i.e., that $\iota^{-1}\cp (g\cdot \iota)$ is defined and bijective. Then, $\iota^{-1}\cp (g\cdot \iota)$ is a homeomorphism, because $\iota$ and $g\cdot \iota$ are embeddings by compactness of $S$. Thus, given $\Sigma\in \Seg(S)$, then $\Sigma':=\iota^{-1}(g\cdot \iota(\Sigma))\in \Seg(S)$ is a segment (connected, non-empty, non-singleton) and $\varrho\equiv (\iota|_{\Sigma'})^{-1}\cp (g\cdot \iota|_\Sigma)$ is a homeomorphism, hence an  analytic diffeomorphism by  Corollary  \ref{fddsfd}. 

It remains to show that $g\cdot \iota(S)=\iota(S)$ holds. For this, let $Z\subseteq S$ denote the set of all $z\in S$ such that there exists an open neighbourhood $U_z\subseteq S$ of $z$ and an open subset $\UU\subseteq S$ with $g\cdot \iota(U_z)= \iota(\UU)$. Then, $Z$ is obviously open, as well as non-empty by assumption. The claim thus follows from connectedness of $S$, once we have shown that $Z$ is closed.

Assume thus that $Z$ is not closed, i.e., that there exists some $x\in \clos[Z]\setminus Z$. Let $(U,\psi)$ be a chart around $x$. Since $U'\cap (Z\setminus \{x\})\neq  \emptyset$ holds for each neighbourhood $U'\subseteq U$ of $x$, there exists a sequence $\{x_n\}_{n\in \NN}\subseteq U\cap (Z\setminus \{x\})$ with 
\begin{align}
\label{nmdsnmsnmdsnmyxyxyxyy}
	\textstyle x_m\neq x_n \qquad\forall\: \NN\ni m\neq n\in \NN\qquad\qquad\text{and}\qquad\qquad\lim_n x_n=x.
\end{align}
Since $g\cdot \iota(Z)\subseteq \iota(S)$ holds, $x'_n:=\iota^{-1}(g\cdot \iota|_Z(x_n))\in S$ is defined for all $n\in \NN$; and, by compactness of $S$, we can assume that $\lim_n x'_n =x'\in S$ exists. 
Since $\iota^{-1}\cp (g\cdot \iota|_Z)$ is injective, by the left side of \eqref{nmdsnmsnmdsnmyxyxyxyy} we can pass to a subsequence to achieve that $x'_n\neq x'$ holds for all $n\in \NN$. 
We fix a chart $(U',\psi')$ around $x'$. Then, passing to a subsequence once more, we can achieve $\{x'_n\}_{n\in\NN}\subseteq U'$. Since $S\cong \UE$ holds (no boundary), the statement $(\star)$ in Lemma \ref{lemma:BasicAnalytt1}.\ref{lemma:BasicAnalytt11} shows   $x\in Z$, which contradicts the choice of $x$.
\end{proof}
\begin{lemma}
\label{fdggfg}
Let $g\in G$ and $C_1,C_2\in \Seg(S)$ be given with $C_1\cap C_2\neq \emptyset$, such that the maps
\begin{align*}
	&\alpha_1\equiv \iota^{-1}\cp (g\cdot \iota)|_{C_1}\colon C_1 \rightarrow \alpha_1(C_1)\in \Seg(S)\qquad\quad\text{and}\qquad\quad \alpha_2\equiv \iota^{-1}\cp (g\cdot \iota)|_{C_2} \colon C_2\rightarrow \alpha_2(C_2)\in \Seg(S)
\end{align*}
are defined and homeomorphisms. 
Then, the following assertions hold: 
\begingroup
\setlength{\leftmargini}{15pt}
{
\renewcommand{\theenumi}{{\arabic{enumi}})} 
\renewcommand{\labelenumi}{\theenumi}
\begin{enumerate}
\item
\label{fdggfg11}
$\Sigma\equiv C_1\cup C_2$\: and\: $\Sigma'\equiv\alpha_1(C_1)\cup\alpha_2(C_2)$\: are segments. 
\item
\label{fdggfg21}
If there exist homeomorphisms 
	$\chi\colon \Sigma\rightarrow D$, $\chi'\colon \Sigma'\rightarrow D'$  
for certain intervals $D,D'\subseteq \RR$, then 
\begin{align}
\label{dsoioidsoidsoidsds0ds0909ds09dsdsdsddsd}
	\varrho\equiv (\iota|_{\Sigma'})^{-1}\cp (g\cdot \iota)|_{\Sigma}\colon \Sigma \rightarrow \Sigma'
\end{align}
 is an analytic diffeomorphism. 
For instance,  such homeomorphisms exist 
\begingroup
\setlength{\leftmarginii}{14pt}
{
\renewcommand{\theenumii}{{.\alph{enumii}})} 
\renewcommand{\labelenumii}{{\alph{enumii}})}
\begin{enumerate}
\item
\label{fdggfg1}
	if $S$ is homeomorphic to an interval.
\item
\label{fdggfg2}
	if there exist segments $\Seg(S)\ni \LL,\LL'\subset S$ with 
	 $C_1\cup C_2\subseteq \LL$ and $\varrho(C_1)\cup\varrho(C_2)\subseteq \LL'$ (by Remark \ref{cnmnmcviufeiureiure}.\ref{cnmnmcviufeiureiure2}).
\end{enumerate}}
\endgroup
\end{enumerate}}
\endgroup
\end{lemma}
\begin{proof}
First observe that the assumptions imply that \eqref{dsoioidsoidsoidsds0ds0909ds09dsdsdsddsd} is defined and bijective, and that the restrictions 
\begin{align*}
	\varrho|_{C_1}=\alpha_1,\qquad\varrho|_{C_2}=\alpha_2, \qquad \varrho^{-1}|_{\alpha_1(C_1)}=\alpha_1^{-1}, \qquad \varrho|^{-1}_{\alpha_2(C_2)}=\alpha_2^{-1}
\end{align*}
are continuous.  
Then, $\Sigma,\Sigma'$ are segments (non-empty, non-singleton, connected), because we have
\begin{align}
\label{jdsisdoiusiuds}
	C_1\cap C_2\neq \emptyset\qquad\text{and then also}\qquad \alpha_1(C_1)\cap \alpha_2(C_2)=\varrho(C_1)\cap \varrho(C_2)
 = \varrho(C_1\cap C_2)\neq \emptyset,
\end{align}
which proves Part \ref{fdggfg11}. 
For Part \ref{fdggfg21}, first assume that $C_2\subseteq C_1$ (hence $\varrho=\alpha_1$) or $C_1\subseteq C_2$ (hence $\varrho=\alpha_2$) holds. Then, $\varrho$ is a homeomorphism, and the claim is clear from Corollary \ref{fddsfd}. Finally, assume that we have 
\begin{align*}
&C_1\nsubseteq C_2\quad\:\: \text{hence,}\quad\:\:\alpha_1(C_1)=\varrho(C_1)\nsubseteq\varrho(C_2)= \alpha_2(C_2)\\ 
\text{as well as}\:\:\quad&C_2\nsubseteq C_1\quad\:\: \text{hence,}\quad\:\:\alpha_2(C_2)=\varrho(C_2)\nsubseteq \varrho(C_1)= \alpha_1(C_1).\qquad\qquad 
\end{align*}
Applying the homeomorphisms $\chi,\chi'$ and using \eqref{jdsisdoiusiuds}, we obtain the following statements:
\begingroup
\setlength{\leftmargini}{14pt}
{
\renewcommand{\theenumi}{\small{\bf\arabic{enumi})}} 
\renewcommand{\labelenumi}{\theenumi}
\begin{enumerate}
\item
\label{hdshjdshj1}
	If $x\in \Sigma$ is an accumulation point of $C_1\setm C_2$ in $\Sigma$, then we have $x\in C_1$. 
\item
\label{hdshjdshj2}	
	If $x\in \Sigma$ is an accumulation point of $C_2\setm C_1$ in $\Sigma$, then we have $x\in C_2$. 
\item
\label{hdshjdshj3}
	If $x\in \Sigma'$ is an accumulation point of $\alpha_1(C_1)\setm \alpha_2(C_2)$ in $\Sigma'$, then we have $x\in \alpha_2(C_1)$. 
\item
\label{hdshjdshj4}
	If $x\in \Sigma'$ is an accumulation point of $\alpha_2(C_2)\setm \alpha_1(C_1)$ in $\Sigma'$, then we have $x\in \alpha_2(C_2)$. 
\end{enumerate}}
\endgroup
\noindent
Then, to show that $\varrho$ is an analytic diffeomorphism, by Corollary \ref{fddsfd}, it suffices to show that $\varrho$ is a homeomorphism. We now only prove that $\varrho$ is continuous; because, continuity of $\varrho^{-1}$ then follows by the same arguments, with $C_1,C_2,\alpha_1,\alpha_2$  replaced by $\alpha_1(C_1),\alpha_2(C_2),\alpha_1^{-1},\alpha_2^{-1}$, respectively.
Let now $\{x_n\}_{n\in \NN}\subseteq \Sigma$ be given,  with $\lim_n x_n=x\in \Sigma$. We have to show that then $\lim_n\varrho(x_n)=\varrho(x)$ holds: 
\begingroup
\setlength{\leftmargini}{10pt}
\begin{itemize}
\item
Assume $x\in C_1\cap C_2$. Then, the claim is immediate from   continuity of $\alpha_1$ and $\alpha_2$. 
\item
Assume $x\in C_1\setm C_2$ (the case $x\in C_2\setm C_1$ is treated analogously):
\vspace{-4pt}
\begingroup
\setlength{\leftmarginii}{10pt}
\begin{itemize}
\item
	If $\{x_n\}_{\NN\he\ni\he n\he\geq\he N}\subseteq C_1$ holds for some $N\in \NN$, then we have
$\lim_n\varrho(x_n)=\lim_n \alpha_1(x_n)=\alpha_1(x)=\varrho(x)$.
\item
In the other case, $\{x_n\}_{n\in \NN}$ admits a subsequence $C_2\setm C_1\supseteq \{x_{n(m)}\}_{m\in \NN}\rightarrow x\in C_1\setm C_2$, i.e., $x$ is an accumulation point of $C_2\setm C_1$ in $\Sigma$. Then,  \ref{hdshjdshj2} shows $x\in C_2$, which contradicts $x\in C_1\setminus C_2$.\qedhere
\end{itemize}
\endgroup
\end{itemize}
\endgroup
\end{proof}

\subsubsection{Analytic Extensions}
In this section, we discuss extensions of analytic maps between segments and manifolds. Then, we use such extensions to prove an overlap result for the case $S\cong \ID$ (see Proposition \ref{nobound}) that complements Lemma \ref{compnobound}.
\begin{lemma} 
\label{gffdg}
Let $N$ be an analytic manifold with boundary, $\Sigma\in \Seg(S)$ a segment, and $\alpha,\alpha'\colon \Sigma\rightarrow N$ analytic maps. Then, we have the  implication
\begin{align*}
\alpha|_{C}=\alpha'|_{C}\quad\text{for}\quad \Seg(S)\ni C\subseteq \Sigma\qquad\quad \Longrightarrow\qquad\quad \alpha=\alpha'.
\end{align*}
\end{lemma}
\begin{proof}
Let $\Omega\subseteq \Seg(S)$ denote the set of all $C'\in \Seg(S)$ with 
$$C\subseteq C'\subseteq  \Sigma\qquad\quad \text{and}\qquad\quad \alpha|_{C'}=\alpha'|_{C'}.$$
\vspace{-15pt} 

\noindent
Then, $\Sigma':=\bigcup_{C'\in \Omega} C'\in \Seg(S)$ is a segment, with 
$\alpha|_{\Sigma'}=\alpha'|_{\Sigma'}$. Moreover,    
$\Sigma'$ is closed in $\Sigma$: 
\vspace{6pt}

\noindent
{\it Proof of the Claim.}
Let $\hat{\Sigma}'$ denote the closure of $\Sigma'$ in $\Sigma$: 
\begingroup
\setlength{\leftmargini}{11pt}
\begin{itemize}
\item
Since $\alpha,\alpha'$ are continuous, we have $\alpha|_{\hat{\Sigma}'}=\alpha'|_{\hat{\Sigma}'}$. 
\item
Since $\hat{\Sigma'}\subseteq S$ is connected (hence a segment) the definitions enforce $\Sigma'=\hat{\Sigma}'$.
\hspace*{\fill}$\ddagger$ 
\end{itemize}
\endgroup
\noindent 
Since $\Sigma$ is connected, the claim follows once we have shown that $\Sigma'$ is open in $\Sigma$. 
For this, let $z\in \Sigma'$ be fixed. We choose a submanifold chart $(U,\psi)$ of $\Sigma$ centered at $z\in U$, and a chart $(V,\phi)$ of $N$, such that   $\alpha(U),\alpha'(U)\subseteq V$ holds. We write $\phi=(\phi_1,\dots,\phi_{\dim[N]})$. Since $\Sigma'$ is Hausdorff and  connected, $z$ is not isolated in $\Sigma'$; hence, $\psi(z)$ is an accumulation point of zeroes of the analytic maps
\begin{align*}
	\phi_k\cp \alpha\cp \psi^{-1}|_{\psi(U\he\cap\he \Sigma)}-\phi_k\cp\alpha'\cp \psi^{-1}|_{\psi(U\he\cap\he \Sigma)}\qquad\text{for}\qquad k=1,\dots,\dim[N].
\end{align*}  
Then, $\alpha|_{U\he\cap\he\Sigma}=\alpha'|_{U\he\cap\he \Sigma}$ holds by analyticity, which shows that $\Sigma'$ is open in $\Sigma$.
\end{proof}
\begin{corollary}
\label{fdsdsffds}
Let $\Sigma,\Sigma'\in \Seg(S)$ be segments, and $\varrho,\varrho'\colon \Sigma\rightarrow \Sigma'$ analytic diffeomorphisms. Then, 
\begin{align*}
\varrho|_{C}=\varrho'|_{C}\quad\text{for}\quad \Seg(S)\ni C\subseteq \Sigma \qquad\quad \Longrightarrow\qquad\quad \varrho=\varrho'.
\end{align*}
\end{corollary}
\begin{corollary}
\label{aaaafsdfsfdfs}
We have the inclusion\hspace*{\fill}(\he recall the definition of $\overlap(S)$ in  \eqref{dslkjkjskjds})
\begin{align}
\label{stabiconji}
	\conj_{g}(G_S)\subseteq G_S\qquad\quad\forall\: g\in \overlap(S).
\end{align}
\end{corollary}
\begin{proof}
Let $q\in G_S$ and $g\in \overlap(S)$, and let $\OO,\OO'\subseteq \innt[S]$ be open segments with $g\cdot \iota(\OO)=\iota(\OO')$. We obtain
\begin{align*}
q\cdot \underbrace{(g^{-1}\cdot \iota(z'))}_{\in\he \iota(\OO)}\stackrel{q\:\in\: G_S}{=} g^{-1}\cdot \iota(z')\qquad\forall\:z'\in \OO'
\qquad\quad&\stackrel{\phantom{\text{Lemma } \ref{gffdg}}}{\Longrightarrow}\qquad\quad \conj_{g}(q)\cdot \iota|_{\OO'}=\iota|_{\OO'}\\[-10pt]
&\stackrel{\text{Lemma } \ref{gffdg}}{\Longrightarrow}\qquad\quad \conj_{g}(q)\in G_S,
\end{align*}
\vspace{-22pt}

\noindent
which proves the claim.
\end{proof}
\begin{corollary}
\label{dfdfdfdfd}
We have $G_C=G_S$ for each $C\in \Seg(S)$. 
\end{corollary}
\begin{proof}
	The inclusion $G_S\subseteq G_C$ is evident. For the other  inclusion, we fix $g\in G_C$ and set $\Sigma:=  S$. We define $\alpha:=g\cdot \iota$ as well as $\alpha':=\iota$, and obtain 
	\begin{align*}
		\alpha|_C=g\cdot \iota|_C\stackrel{g\he\in\he G_C}{=}\iota|_C=\alpha'|_C\quad\:\:\stackrel{\text{Lemma } \ref{gffdg}}{\Longrightarrow}\quad\:\: g\cdot \iota=\alpha|_\Sigma=\alpha'|_\Sigma=\iota \qquad\Longrightarrow\qquad g\in G_S,
	\end{align*}
which proves the claim.
\end{proof}
\noindent
Next, let $\varrho\colon \Seg(S)\ni \Sigma\rightarrow \Sigma'\in \Seg(S)$ be an analytic diffeomorphism (observe that then $\Sigma'=\im[\varrho]\in \Seg(S)$ is already enforced by $\Sigma\in \Seg(S)$). We have the following notions: 
\noindent
\begingroup
\setlength{\leftmargini}{11pt}
\begin{itemize}
\item
	An \textbf{extension of $\boldsymbol{\varrho}$} is an analytic diffeomorphism $\wt{\varrho}\colon \Seg(S)\ni\wt{\Sigma}\rightarrow \wt{\Sigma}'\in \Seg(S)$ with $\Sigma \subseteq \tilde{\Sigma}$ and $\wt{\varrho}|_\Sigma=\varrho$. If $\Sigma\subset \wt{\Sigma}$ holds, then $\wt{\varrho}$ is called a \textbf{proper extension of $\boldsymbol{\varrho}$}.
\begingroup
\setlength{\leftmarginii}{11pt}
\begin{itemize}
\item 
We denote the set of all extensions of $\varrho$  by $\EE(\varrho)$.
\item
	$\varrho$ is said to be \textbf{maximal} \defff $\varrho$ admits no proper extension.
\end{itemize}
\endgroup
\item
	An extension of $\varrho$ is called \textbf{maximal extension (of $\boldsymbol{\varrho}$)} \defff it admits no proper extension. If a maximal extension of $\varrho$ exist, then it is denoted by $\ovl{\varrho}$ (necessarily unique by Corollary \ref{fdsdsffds}). 	
\end{itemize}
\endgroup
\begin{lemma}
\label{kjdskjdskjdskjds}
If $S$ is homeomorphic to an interval, then each analytic diffeomorphism $\varrho\colon \Seg(S)\ni \Sigma\rightarrow \Sigma'\in \Seg(S)$ admits a unique maximal extension 
\vspace{-5pt}
\begin{align}
\label{dsjkkjdslkdsdshdsiudsiudsiuds9ds09ds09ds98ds98ds}
\begin{split}
&\hspace{65pt}\textstyle\ovl{\varrho}\colon \Seg(S)\ni\ovl{\Sigma}\rightarrow \ovl{\varrho}(\ovl{\Sigma})\in \Seg(S)\\[2pt]
&\hspace{105pt}\text{whereby}\\ 
&\textstyle\ovl{\Sigma}=\bigcup_{\wt{\varrho}\he\in\he \EE(\varrho)}\dom[\wt{\varrho}]\qquad\quad\wedge\qquad\quad \wt{\varrho}=\ovl{\varrho}|_{\dom[\wt{\varrho}]}\quad\:\:\forall\:  \wt{\varrho}\in \EE(\varrho).
\end{split}
\end{align} 
\end{lemma}
\begin{proof}
Clearly, $\ovl{\Sigma}:=\bigcup_{\wt{\varrho}\he\in\he \EE(\varrho)}\dom[\wt{\varrho}]$ is  a segment. 
We define $\ovl{\varrho}\colon \ovl{\Sigma}\rightarrow S$ by 
    	\begin{align*} 
    		\ovl{\varrho}(x):=\wt{\varrho}(x)\qquad\text{for}\qquad\wt{\varrho}\in \EE(\varrho)\qquad\text{with}\qquad x\in \dom[\wt{\varrho}].
    	\end{align*}
This is well defined by Corollary \ref{fdsdsffds}; because, if we are given $\wt{\varrho},\wt{\varrho}'\in \EE(\varrho)$ with $x\in \dom[\wt{\varrho}],\dom[\wt{\varrho}']$, then 
\begin{align*}
	\Sigma\subseteq \dom[\wt{\varrho}]\cap \dom[\wt{\varrho}']=:C\ni x \quad\:\:\stackrel{\text{Corollary }\ref{fdsdsffds}}{\Longrightarrow}\quad\:\: \wt{\varrho}|_C=\wt{\varrho}'|_C \qquad\Longrightarrow\qquad \wt{\varrho}(x)=\wt{\varrho}'(x),
\end{align*} 
whereby we have used in the first step that   
$C$ is connected (hence a segment) as $S$ is homeomorphic to an interval. It is straightforward to see that $\ovl{\varrho}$ is an extension of $\varrho$ (in particular analytic). Moreover, we obtain from the definition of $\ovl{\Sigma}$:
\begingroup
\setlength{\leftmargini}{11pt}
\begin{itemize}
\item
$\ovl{\varrho}$ is maximal; because, for $\wt{\varrho}\in \EE(\varrho)$, we have $\dom[\wt{\varrho}]\subseteq \ovl{\Sigma}$. Hence, $\wt{\varrho}=\ovl{\varrho}|_{\dom[\wt{\varrho}]}$ holds by Corollary \ref{fdsdsffds} applied to $C\equiv \Sigma$, $\varrho\equiv \wt{\varrho}$, and $\varrho'\equiv\ovl{\varrho}|_{\dom[\wt{\varrho}]}$.  
\item
If $\wt{\varrho}\in \EE(\varrho)$ is maximal, then we necessarily have $\dom[\wt{\varrho}]= \ovl{\Sigma}$. Hence, $\wt{\varrho}=\ovl{\varrho}$ holds by Corollary \ref{fdsdsffds} applied to $C\equiv \Sigma$, $\varrho\equiv \wt{\varrho}$, and $\varrho'\equiv \ovl{\varrho}$. 
\qedhere 
\end{itemize}
\endgroup
\end{proof}

\begin{proposition}
\label{nobound}
Assume that $S$ is homeomorphic to an interval; and let $\varrho\colon \Seg(S)\ni\LL\rightarrow \LL'\in \Seg(S)$ be an analytic diffeomorphism such that $g\cdot \iota|_\LL=\iota\cp \varrho$ holds for some $g\in G$. Let furthermore $\Sigma\in\Seg(S)$ be given with $\LL\subseteq \Sigma$. 
Then, $C:=\dom[\ovl{\varrho}]\cap \Sigma\in \Seg(S)$ holds  with $g\cdot \iota|_C=\iota\cp\ovl{\varrho}|_C$; and furthermore:
\begingroup
\setlength{\leftmargini}{14pt}
\begin{enumerate}
\item
\label{nobound1}
If $C\subset \Sigma$ holds, then $S\setm \ovl{\varrho}(C)$ is connected (possibly empty). 
\item
\label{nobound2}
If $\Sigma\setm C$ is  not connected, then $S=\ovl{\varrho}(C)$ holds.
\end{enumerate}
\endgroup
\end{proposition}
\begin{proof}
We have $\Seg(S)\ni \LL\subseteq C=\dom[\ovl{\varrho}]\cap \Sigma$ with $\dom[\ovl{\varrho}],\Sigma\in \Seg(S)$.  
\begingroup
\setlength{\leftmargini}{11pt}
\begin{itemize}
\item
Hence, $C\in \Seg(S)$ is a segment, because $S\cong \ID$ holds (via $\homeo\colon \ID\rightarrow S$).
\item
Then, Corollary \ref{fdsdsffds} shows $g\cdot \iota|_C=\iota\cp\ovl{\varrho}|_C$ (as we have $\Seg(S)\ni\LL\subseteq C \in \Seg(S)$).
\end{itemize}
\endgroup 
\noindent
We define $B:=\homeo^{-1}(C)\subseteq \ID$ as well as $B':=\homeo^{-1}(\ovl{\varrho}(C))\subseteq \ID$:
\begingroup
\setlength{\leftmargini}{14pt}
\begin{enumerate}
\item
Assume that $C\subset \Sigma$ holds; and let $x\in \partial[C]\cap \Sigma$ 
be as in Remark \ref{cnmnmcviufeiureiure}.\ref{cnmnmcviufeiureiure3}, hence  
	\begin{align}
	\label{nmnmcxnmcxnmcx}
	\UU\cap (C\setm \{x\})\neq \emptyset
	\neq    \UU\cap  (\Sigma\setm C)
	\:\:\:\:\text{holds for each neighbourhood}\:\:\:\UU\:\:\:\text{of}\:\:\:\: x.
	\end{align}  
Then, there exists a converging sequence $C\setm \{x\} \supseteq \{x_n\}_{n\in \NN}\rightarrow  x$ such that 
$\{\homeo^{-1}(x_n)\}_{n\in \NN}\subseteq B$ is monotonous. For each $n\in \NN$, we set\hspace*{\fill}(hence, $t_n'=(\homeo^{-1}\cp \ovl{\varrho})(x_n)$)
\begin{align*}
	x_n':=\ovl{\varrho}(x_n)\in \ovl{\varrho}(C)\qquad\quad\text{and}\qquad\quad t_n':=\homeo^{-1}(x_n')\in \homeo^{-1}(\ovl{\varrho}(C))=B'.
\end{align*}
Since $(\homeo^{-1}\cp \he\ovl{\varrho}\he\cp \homeo|_{B})|^{B'}\colon B\rightarrow B'$ is a homeomorphism (continuous injective), also     
$\{t'_n\}_{n\in \NN}\subseteq B'$ is monotonous. 
\vspace{1pt}

Assume now that $S\setm \ovl{\varrho}(C)$ is not connected (i.e.\ that the claim is wrong). Then, $\ID\setm B'$ is not connected; hence, $\lim_n \homeo^{-1}(x'_n)=\lim_n t_n'=t'\in \ID$ exists by monotonicity. We thus have $\lim_n x'_n=x':=\homeo(t')\in S$ by continuity of $\homeo$. 
Continuity of $\iota$, $g\cdot \iota$ (first implication) as well as injectivity of $\iota$, $g\cdot \iota$ together with $\{x_n\}_{n\in \NN}\subseteq C\setm \{x\}$ (second implication) yield 
\begin{align*}
	 g\cdot \iota(x_n)=(\iota\cp \ovl{\varrho})(x_n)=\iota(x_n')\quad\:\: \forall\: n\in \NN \qquad\Longrightarrow\qquad g\cdot \iota(x)=\iota(x')
	 \qquad\Longrightarrow\qquad \{x'_n\}_{n\in \NN}\subseteq S\setm \{x'\}.
\end{align*}
We fix charts $(U,\psi),(U',\psi')$ around $x,x'$, respectively,  
and pass to subsequences to assure $\{x_n\}_{n\in \NN}\subseteq U\setminus\{x\}$ and $\{x_n'\}_{n\in \NN}\subseteq U'\setminus\{x'\}$. We choose $\KK\subseteq U$, $\KK'\subseteq U'$, $N\in \NN$ as in Lemma \ref{lemma:BasicAnalytt1}.\ref{slkdlkdslksddsiusdiusiudds98sd98sd98s9dsdsds}; in particular, 
\begingroup
\setlength{\leftmarginii}{17pt}
\begin{itemize}
\item[i)]
$g\cdot\iota(\KK)=\iota(\KK')$ holds, and $\KK,\KK'\in \Seg(S)$ are compact.
\item[ii)]
$x\in \KK$, $x'\in \KK'$ holds, as well as $x_n\in \KK$, $x'_n\in \KK'$ for all $n\geq N$.
\item[iii)]
If $x\in \partial[S]$ or $x,x'\in \innt[S]$ holds, then $\KK$ is a neighbourhood of $x$. 
\end{itemize}
\endgroup 
\noindent 
Since $g\cdot \iota|_C=\iota\cp\ovl{\varrho}|_C$ holds, and 
since $\iota|_{\KK'}$, $g\cdot \iota|_\KK$ are embeddings ($\KK,\KK'$ compact by i)), we have the homeomorphisms
\begin{align*}
\alpha_1\equiv \iota ^{-1}\cp (g\cdot \iota)|_{C}\colon C\rightarrow \iota ^{-1}(g\cdot \iota(C))
\qquad\quad\text{and}\qquad\quad 
\alpha_2\equiv (\iota|_{\KK'})^{-1}\cp (g\cdot\iota|_\KK)\colon \KK\rightarrow \KK'.
\end{align*}  
Then, the claim follows once we have shown that 
$\KK\cap C\neq \emptyset\neq \KK\cap (\Sigma\setm C)$ holds:
\vspace{2pt} 

\noindent
{\it{}Proof of the Statement.}
In fact, $\KK\cap C\neq \emptyset\neq \KK\cap (\Sigma\setm C)$ implies 
$$
\qquad\tilde{C}:=\overbrace{(\KK\cup C)}^{\in\: \Seg(S)}\:\cap\: \Sigma \stackrel{S\he\cong\he \ID}{\in} \Seg(S)\qquad\:\:\text{with}\qquad\:\: C\subset \tilde{C} \subseteq \KK\cup C,
$$
and   
Lemma \ref{fdggfg}.\ref{fdggfg1} shows that $\tilde{\varrho}:=\iota ^{-1}\cp (g\cdot \iota)|_{\KK\he\cup\he C}$ is an analytic diffeomorphism onto its image; hence,  
\begin{align*}
g\cdot \iota|_\LL=\iota\cp \varrho\qquad\Longrightarrow\qquad \tilde{\varrho}|_\LL=\varrho \quad\:\: \stackrel{\text{Lemma }\ref{kjdskjdskjdskjds}}{\Longrightarrow} \quad\:\: \tilde{\varrho}|_{\tilde{C}}=\ovl{\varrho}|_{\tilde{C}} \qquad\Longrightarrow\qquad C\subset\tilde{C}\subseteq \dom[\ovl{\varrho}]\cap\Sigma,
\end{align*} 
which contradicts the definition of $C$. \hspace*{\fill}$\ddagger$
\vspace{2pt}

\noindent
Now, the following two situations can occur:
\begingroup
\setlength{\leftmarginii}{15pt}
\begin{itemize}
\item[\small\bf 1)]
We have $x\in \partial[S]$ or $x,x'\in \innt[S]$: 
\vspace{2pt}

Then, $\KK$ is a neighbourhood of $x$ by iii), so that \eqref{nmnmcxnmcxnmcx} yields $\KK\cap C\neq \emptyset\neq \KK\cap (\Sigma\setm C)$. 
\vspace{1pt}
\item[\small\bf 2)]
We have $x'\in \partial[S]=\partial_S$, hence $t'\in \partial_\ID$: 
\begingroup
\setlength{\leftmarginiii}{12pt}
\begin{itemize}
\item
Then, $x'\notin \ovl{\varrho}(C)$ holds, hence  $t'\notin B'$:
\vspace{4pt}

In fact, assume $x'\in \ovl{\varrho}(C)$, hence $t'\in B'$. Since $t'\in \partial_\ID$ holds by assumption, we must have $t'\in \partial_{B'}$ (as $B'\subseteq \ID$), hence $t'\in \partial_\ID\cap\partial_{B'}$. This, however, implies that $\ID\setm B'$ (hence $S\setm\ovl{\varrho}(C)$) is connected, which contradicts the assumptions. 
\vspace{2pt}
\item
Now, $x'\notin \ovl{\varrho}(C)$ together with continuity of $\ovl{\varrho}$ implies     
$x\notin C$, hence $\KK\cap (\Sigma\setminus C)\neq \emptyset$. Moreover, since $\{x_n\}_{n\geq N}\subseteq (\KK\cap C)$ holds by ii), we also have $\KK\cap C\neq \emptyset$.
\end{itemize}
\endgroup 
\end{itemize}
\endgroup
\item
Assume that $\Sigma\setm C$ is not connected (in particular $C\subset \Sigma$), and that $\A:=S\setm  \ovl{\varrho}(C)\neq \emptyset$ holds (i.e., that the claim is wrong). We observe the following:  
\begingroup
\setlength{\leftmarginii}{15pt}
\begin{itemize}
\item[a)]
$\A$ is connected by Part \ref{nobound1} (as $C\subset \Sigma$ holds); hence, $A:=\homeo^{-1}(\A)=\ID\setm B'$ is connected.
\item[b)]
Then, replacing $\homeo$ by $-\homeo$ if necessary, we  can assume that   
$$
\ID\ni\sup(\{s\:|\: s\in A\})=: \tau := \inf(\{s\:|\: s\in B'\})\in \ID\qquad \text{holds with}\qquad \ID\cap [\tau,\infty)=B'\cap [\tau,\infty).
$$
\vspace{-15pt}
\item[c)]
In particular, if $\{t'_n\}_{n\in \NN}\subseteq B'$ is  decreasing with $\lim_n t_n'=t'\in \RR$, then $t'\in \ID$ holds. In this case, we have the implication:
\begin{align}
\label{soidiodsoisiodiosdiooidsiodsds}	
	t'\in \partial_\ID\cap \partial_{B'}\qquad\Longrightarrow\qquad \ID=B' \qquad\Longleftrightarrow\qquad S=\ovl{\varrho}(C), 
\end{align}    
whereby the right side contradicts the assumption $\A=S\setm  \ovl{\varrho}(C)\neq \emptyset$.
\end{itemize}
\endgroup
Now, since $\Sigma\setm C$ is not connected,  Remark \ref{cnmnmcviufeiureiure}.\ref{cnmnmcviufeiureiure3b} shows  $\partial[C]\cap \Sigma=\{x^+,x^-\}$ with $x^+\neq x^-$ and $\homeo^{-1}(x^-)< \homeo^{-1}(x^+)$ (after relabelling if necessary) such that 
$$\UU^\pm\cap (C\setm\{x^\pm\})\neq \emptyset\neq \UU^\pm\cap  (\Sigma\setm C)
\:\:\:\text{holds for each neighbourhood}\:\:\:\UU^\pm\:\:\text{of}\:\:\: x^\pm.
$$
In particular, there exist sequences $\{x^\pm_n\}_{n\in \NN} \subseteq C\setm \{x^\pm\}$ with $\lim_n x_n^\pm=x^\pm$ such that 
\begingroup
\setlength{\leftmarginii}{11pt}
\begin{itemize}
\item[$\cp$]
$\{\homeo^{-1}(x^+_n)\}_{n\in \NN}$ is strictly increasing,
\item[$\cp$]
$\{\homeo^{-1}(x^-_n)\}_{n\in \NN}$ is strictly decreasing.
\end{itemize}
\endgroup
\noindent 
Since $\ovl{\varrho}$ is a homeomorphism, one of the sequences $\{(\homeo^{-1}\cp\ovl{\varrho})(x_n^\pm)\}_{n\in \NN}\subseteq B'$ 
is strictly decreasing. 
Hence,    
there exists $x\in \partial[C]\cap \Sigma$ such that 	\begin{align}
\label{nmnmcxnmcxnmcxa333}
	\UU\cap (C\setm \{x\})\neq \emptyset
	\neq    \UU\cap  (\Sigma\setm C)
	\:\:\:\:\text{holds for each neighbourhood}\:\:\:\UU\:\:\:\text{of}\:\:\: x
\end{align}
as well as a sequence $C\setm\{x\}\supseteq \{x_n\}_{n\in \NN}\rightarrow x$ such that the sequence
$$ 
\{t'_n\}_{n\in \NN}\subseteq B'\quad\:\:\text{with}\quad\:\: 
t_n':= \homeo^{-1}(x'_n)\quad\:\: 
\text{for}\quad\:\: x_n':=\varrho(x_n)\quad\:\:\text{for each}\quad\:\:  n\in \NN
$$ 
is strictly decreasing. Then, $\{t'_n\}_{n\in \NN}$   
converges to some $t'\in \ID$  by c), with $\{t'_n\}_{n\in \NN}\subseteq \ID\setminus\{t'\}$. Hence, $\{x'_n\}_{n\in \NN}$ converges to $x':=\homeo(t')\in S$ with $\{x'_n\}_{n\in \NN} \subseteq S\setminus \{x\}$.  
The rest of the argumentation now very similar to that   in Part \ref{nobound1}: 
\vspace{2pt}

We fix charts $(U,\psi),(U',\psi')$ around $x,x'$, respectively,  
and pass to subsequences to assure $\{x_n\}_{n\in \NN}\subseteq U\setminus\{x\}$ and $\{x_n'\}_{n\in \NN}\subseteq U'\setminus\{x'\}$. We choose $\KK\subseteq U$, $\KK'\subseteq U'$, $N\in \NN$ as in Lemma \ref{lemma:BasicAnalytt1}.\ref{slkdlkdslksddsiusdiusiudds98sd98sd98s9dsdsds}; so that, in particular, i), ii), iii) hold.  Then, we have the homeomorphism
$$(\iota|_{\KK'})^{-1}\cp (g\cdot\iota|_\KK)\colon \KK\rightarrow \KK',$$  
and the claim follows once we have shown   
$\KK\cap C\neq \emptyset\neq \KK\cap (\Sigma\setm C)$ (same argument as in Part \ref{nobound1}).   
Now, the following two situations can occur:
\begingroup
\setlength{\leftmarginii}{17.5pt}
\begin{itemize}
\item[\small\bf 1')]
We have $x\in \partial[S]$ or $x,x'\in \innt[S]$: 
\vspace{2pt}

Then, $\KK$ is a neighbourhood of $x$ by iii), so that \eqref{nmnmcxnmcxnmcxa333} yields $\KK\cap C\neq \emptyset\neq \KK\cap (\Sigma\setm C)$. 
\vspace{1pt}
\item[\small\bf 2')]
We have $x'\in \partial[S]=\partial_S$, hence $t'\in \partial_\ID$:
\begingroup
\setlength{\leftmarginiii}{12pt}
\begin{itemize}
\item
Then, $x'\notin \ovl{\varrho}(C)$ holds, hence  $t'\notin B'$: 
\vspace{4pt}

In fact, assume $x'\in \ovl{\varrho}(C)$, hence $t'\in B'$. Since $t'\in \partial_\ID$ holds by assumption, we must have $t'\in \partial_{B'}$ (as $B'\subseteq \ID$), hence $t'\in \partial_\ID\cap\partial_{B'}$. Then,  \eqref{soidiodsoisiodiosdiooidsiodsds} shows $S=\ovl{\varrho}(C)$, which contradicts the assumptions.
\vspace{-10pt}
\item
Now, $x'\notin \ovl{\varrho}(C)$ together with continuity of $\ovl{\varrho}$ implies     
$x\notin C$, hence $\KK\cap (\Sigma\setminus C)\neq \emptyset$. Moreover, since $\{x_n\}_{n\geq N}\subseteq (\KK\cap C)$ holds by ii), we also have $\KK\cap C\neq \emptyset$.\qedhere
\end{itemize}
\endgroup 
\end{itemize}
\endgroup  
\end{enumerate}
\endgroup
\end{proof}

\subsection{Free Segments}
In this section, let $\wm\colon G\times M\rightarrow M$ be a Lie group action that is analytic in $M$, and let $(S,\iota)$ denote an analytic 1-submanifold of $M$.
\vspace{6pt}

\noindent
An analytic 1-submanifold $(S,\iota)$ of $M$ is said to be {\bf free} \defff it admits a {\bf free segment}, i.e., a 
  segment $\Sigma\in \Seg(S)$ such that the following equivalence holds:
\begin{align}
\label{kfdkdfklfd}
	g\cdot \iota|_\Sigma\cpsim \iota|_\Sigma\quad\text{for}\quad g\in G\qquad\:\:\Longleftrightarrow\qquad\:\:  g\cdot \iota|_\Sigma=\iota|_\Sigma\qquad\big(\:\stackrel{\text{\rm Lemma }\ref{gffdg}}{\Longleftrightarrow}\quad\:\: g\in G_S\:\he\big).
\end{align}
\noindent
\emph{(Notably, by Corollary \ref{dfdsasasasassa}, the right side of \eqref{kfdkdfklfd} implies the left side of \eqref{kfdkdfklfd}.)}  
\begingroup
\setlength{\leftmargini}{11pt}
\begin{itemize}
\item
The set of all free segments (in $S$) is denoted by $\Free(S)\subseteq \Seg(S)$. 
\item
For $S'\in \Seg(S)$, we let $\Free(S')$ denote the set of all free segments in $(S',\iota|_{S'})$. 
Corollary \ref{dfdfdfdfd} yields:
\begin{align}
\label{jdkjfdkjfddf}
\Free(S')=\{\Sigma\in \Free(S)\:|\: \Sigma\subseteq S'\}\qquad\:\text{as well as}\qquad\: \GESS=\GES.\qquad\qquad\:
\end{align}
\item
For $S'\in \Seg(S)$, we let $\MFree(S')$ denote the set of all maximal segments in $(S',\iota|_{S'})$. We have the implication
\begin{align}
\label{hggffdfdxcxfdtrtrtrtrzt7678878787878787878787}
	\MFree(S)\ni \Sigma \subseteq S'\qquad\quad\stackrel{\eqref{jdkjfdkjfddf}}{\Longrightarrow}\qquad\quad \Sigma\in \MFree(S').
\end{align}
\end{itemize}
\endgroup 
\begin{lemma}
\label{jdkjfdkjfd}
$\Seg(S)\ni \Sigma'\subseteq \Sigma\in \Free(S)$ implies $\Sigma'\in \Free(S)$.  
\end{lemma}
\begin{proof}
For $g\in G$, we have
\vspace{-5pt}
\[
	g\cdot \iota|_{\Sigma'}\cpsim \iota|_{\Sigma'}\qquad\Longrightarrow\qquad  g\cdot \iota|_\Sigma\cpsim \iota|_\Sigma
\qquad\stackrel{\eqref{kfdkdfklfd}}{\Longrightarrow}\qquad g\in G_S \qquad\Longrightarrow\qquad g\cdot \iota|_{\Sigma'}=\iota|_{\Sigma'}.\qedhere
\]
\end{proof}
\noindent
A free segment $\Sigma\in \Free(S)$ is said to be \textbf{maximal} 
\defff 
\begin{align*}
	\Sigma\subseteq \Sigma'\in \Free(S)\qquad\Longrightarrow\qquad \Sigma=\Sigma'.
\end{align*}
The set of all maximal free segments (in $S$) is denoted by $\MFree(S)$. 
\begin{lemma}
\label{kjkjdskjkjdsaassaa}
If $\Sigma \in \Free(S)$, then $\clos[\Sigma]\in \Free(S)$. In particular, 
each $\Sigma'\in \MFree(S)$ is closed in $S$.  
\end{lemma}
\begin{proof}
For $g\in G$, we have
\[
	g\cdot \iota|_{\clos[\Sigma]}\cpsim \iota|_{\clos[\Sigma]}
	\qquad\stackrel{\eqref{maxicp}}{\Longrightarrow}\qquad  g\cdot \iota|_\Sigma\cpsim \iota|_\Sigma
\qquad\stackrel{\eqref{kfdkdfklfd}}{\Longrightarrow}\qquad g\in G_S \qquad\Longrightarrow\qquad g\cdot \iota|_{\clos[\Sigma]}=\iota|_{\clos[\Sigma]}.\qedhere
\]
\end{proof}
\begin{lemma}
\label{fhdhh}
For each $\Sigma\in \Free(S)$, there exists $\Sigma'\in \MFree(S)$ with $\Sigma\subseteq \Sigma'$. 
\end{lemma}
\begin{proof}
	Let $\MK:=\{C\in \Free(S)\:|\: \Sigma\subseteq C\}$  
be partially ordered by inclusion (observe $\Sigma
\in\MK\neq \emptyset$).  Then, a given chain $\ML$ in $\MK$ has the upper bound 
	$\textstyle \mathrm{B}(\ML):= \bigcup_{C\in \ML} C$,
	 so that the set of maximal elements in $\MK$ (i.e.. the set $\MFree(S) \cap \MK$) is non-empty by Zorn's lemma: 
\vspace{4pt}

\noindent	 
In fact, to see that $\mathrm{B}(\ML)$ is an upper bound of $\ML$, first observe that $\mathrm{B}(\ML)\in \Seg(S)$ holds (as $\Sigma\subseteq C$ holds for each $C\in \ML$ by definition). Then, $\mathrm{B}(\ML)$ is a free segment (i.e. $\mathrm{B}(\ML)\in \Free(S)$), because we have
\vspace{-6pt}
\begin{align*}
	g\cdot \iota|_{\mathrm{B}(\ML)} \cpsim \iota|_{\mathrm{B}(\ML)}\qquad\Longrightarrow\qquad g\cdot \iota|_{C'} \cpsim   \iota|_{C'}\:\:\:\text{for some}\:\:\: C'\in \ML \quad\:\:\stackrel{C'\he\in\he \Free(S)}{\Longrightarrow}\quad\:\:  g\in G_S,
\end{align*}
whereby the first implication is obtained as follows:
\begingroup
\setlength{\leftmargini}{12pt}
\begin{itemize} 
\item
By assumption, there exist open segments $\Seg(S)\ni \OO,\OO'\subseteq \mathrm{B}(\ML)$ with $g\cdot \iota(\OO) = \iota(\OO')$ such that $g\cdot \iota|_\OO$, $\iota|_{\OO'}$ are embeddings.
\item
According to the definition of $\mathrm{B}(\ML)$, there exists some $C\in \ML$ with $C\cap \OO\neq \emptyset$. 
Then, by Remark \ref{cnmnmcviufeiureiure}.\ref{cnmnmcviufeiureiure1} and the embedding properties of $g\cdot \iota|_\OO$, $\iota|_{\OO'}$, there exist  
 open segments 
$\Seg(S)\ni \UU\subseteq C\cap \OO$ and $\Seg(S)\ni\UU'\subseteq \OO'$ such that  $g\cdot  \iota(\UU)=\iota(\UU')$ holds. 
\item
According to the definition of $\mathrm{B}(\ML)$, there exists some $C'\in \ML$ with $C\subseteq C'$ (hence $\UU\subseteq C\subseteq C'$) such that  $C'\cap \UU'\neq \emptyset$ holds (observe $\UU'\subseteq \OO'\subseteq \mathrm{B}(\ML)$). Then, by Remark \ref{cnmnmcviufeiureiure}.\ref{cnmnmcviufeiureiure1} and the embedding properties of $g\cdot \iota|_\OO$, $\iota|_{\OO'}$, 
there 
exist open segments $\wt{\UU}\subseteq \UU\subseteq C'$ and $\wt{\UU}'\subseteq C'\cap \UU'$ with $g\cdot  \iota(\wt{\UU})=\iota(\wt{\UU}')$. This shows $g\cdot \iota|_{C'}\cpsim\iota|_{C'}$, because $g\cdot \iota|_{\wt{\UU}}$ and $\iota|_{\wt{\UU}'}$ are embeddings (as  $g\cdot \iota|_\OO$, $\iota|_{\OO'}$ are embeddings).
\qedhere
\end{itemize}
\endgroup	
\end{proof}
\begin{lemma}
\label{ghfhh}
Let $\Sigma,\Sigma'\in \Seg(S)$ and $g\in G$ be given such that   
$g\cdot \iota|_\Sigma=\iota|_{\Sigma'}\cp\varrho$ holds for an 
analytic diffeomorphism $\varrho\colon \Sigma\rightarrow\Sigma'$. Then, $\Sigma\in \Free(S)$ implies $\Sigma'\in \Free(S)$. Moreover,    
 if $\Sigma$ (hence, $\Sigma'$) is compact with $\Sigma\subseteq\innt[S]$, then $\Sigma\in \MFree(S)$ implies $\Sigma'\in \MFree(S)$. 
\end{lemma}
\begin{proof}
For the first statement, let $\Sigma\in \Free(S)$ and $q\in G$. Then,\footnote{Since $g\in \overlap(S)$ holds, the last three implications can also be omitted by applying Corollary \ref{aaaafsdfsfdfs} (which directly yields $q\in G_S$).}
\begin{align*}
	q\cdot \iota|_{\Sigma'}\cpsim \iota|_{\Sigma'}\qquad&\Longrightarrow\qquad q\cdot (\iota|_{\Sigma'}\cp\varrho)\cpsim \iota|_{\Sigma'}\cp\varrho\quad\hspace{7.8pt}\Longrightarrow\qquad q\cdot (g\cdot \iota|_\Sigma)\cpsim (g\cdot \iota|_\Sigma)\\
	&\Longrightarrow \qquad (g^{-1}\cdot q\cdot g)\cdot \iota|_\Sigma\cpsim \iota|_\Sigma\quad\hspace{7pt}\Longrightarrow\qquad (g^{-1}\cdot q\cdot g)\in G_S\\
	&\Longrightarrow\qquad (q\cdot g)\cdot \iota|_{\Sigma}=g\cdot \iota|_{\Sigma}\qquad\hspace{12.5pt}\Longrightarrow\hspace{19.9pt} q\cdot (g\cdot (\iota|_{\Sigma}\cp \varrho^{-1}))=g\cdot (\iota|_{\Sigma}\cp \varrho^{-1})\\
	&\Longrightarrow\qquad q\cdot \iota|_{\Sigma'}=\iota|_{\Sigma'}.
\end{align*}
For the second statement, let $\innt[S]\supseteq \Sigma\in \MFree(S)$ be compact. Then, $\Sigma'=\varrho(\Sigma)\in \Free(S)$ holds by the first part, and $\Sigma'$ is compact by continuity of $\varrho$. We distinguish between the cases $\Sigma=\innt[S]$ and $\Sigma\subset\innt[S]$:
\begingroup
\setlength{\leftmargini}{18pt}
{
\renewcommand{\theenumi}{\small{\bf(\alph{enumi})}} 
\renewcommand{\labelenumi}{\theenumi}
\begin{enumerate}
\item
Assume $\Sigma=\innt[S]$:\: Since $\Sigma$ is compact, we necessarily have $\UE\cong S=\innt[S]=\Sigma$, and $g\cdot \iota|_\Sigma=\iota|_{\Sigma'}\cp\varrho$ implies  
\vspace{-4pt}
\begin{align*}
g\cdot \iota\cpsim \iota\quad\:\:\stackrel{\text{Lemma }\ref{compnobound}}{\Longrightarrow}\quad\:\:	\iota(S)=g\cdot \iota(S)=g\cdot \iota(\Sigma)=\iota(\varrho(\Sigma))=\iota(\Sigma')\quad\:\:\stackrel{\iota\text{ injective}}{\Longrightarrow}\quad\:\: \Sigma'=S=\Sigma\in \MFree(S).
\end{align*}
\item
Assume $\Sigma\subset \innt[S]$:\: We choose segments $\Sigma\subset \UU\subset \LL\subset\innt[S]$ as in Remark \ref{cnmnmcviufeiureiure}.\ref{cnmnmcviufeiureiure4} (with $\KK\equiv \Sigma$ there), and assume that the claim is wrong, i.e., that $\Sigma'$ is not maximal. Then, there exists some $\LL'\in \Free(S)$ with $\Sigma'\subset \LL'$. By Lemma \ref{kjkjdskjkjdsaassaa}, we can assume that $\LL'$ is closed in $S$; in particular:
\begingroup
\setlength{\leftmarginii}{17pt}
\begin{itemize} 
\item[($\star$)]
If  $S\cong \UE$ holds, then $\LL'$ is compact (as $\LL'$ in closed in $S$ compact), with $\LL'\subset S=\innt[S]$ (since elsewise  $S=\LL'\in \Free(S)$   
holds, 
which contradicts $\MFree(S)\ni \Sigma \subset \innt[S]= S$).
\end{itemize}
\endgroup
\noindent
We observe the following:
\vspace{-2pt}
\begingroup
\setlength{\leftmarginii}{12pt}
\begin{itemize} 
\item
Since $\Sigma'\subset \LL'$, Remark \ref{cnmnmcviufeiureiure}.\ref{cnmnmcviufeiureiure3c} (with $\KK\equiv \Sigma'$ there) yields some 
$x'\in \partial_{\Sigma'}\cap\innt[\LL']\subseteq \innt[S]$  such that  
\begin{align}
\label{nmdsnmnmddsnmnmdsdsnm}
	\UU'\cap (\Sigma'\setm\{x'\})\neq \emptyset\neq  
	\UU'\cap  (\LL'\setm \Sigma')\:\:\:\text{holds for each neighbourhood}\:\:\:\UU'\:\:\text{of}\:\:\: x'.
\end{align}
Then, $x:=\varrho^{-1}(x')\in \partial_\Sigma\cap \innt[S]$ holds; because, (by definition) $\varrho$ extends to an analytic diffeomorphism (in particular, a homeomorphism) defined on an open subset that contains (the compact segment) $\Sigma$. 
\item
Let $(U,\psi),(U',\psi')$ be charts around $x,x'$. 
Since we have $g\cdot \iota|_\Sigma=\iota|_{\Sigma'}\cp\varrho$, $x\in \partial_\Sigma$, $x'\in \partial_{\Sigma'}$,  there exist sequences that fulfill the assumptions in Lemma \ref{lemma:BasicAnalytt1}. Since $x,x'\in \innt[S]$ holds, Lemma \ref{lemma:BasicAnalytt1}.\ref{lemma:BasicAnalytt123} provides connected compact neighbourhoods $\KK,\KK'$ of $x,x'$, respectively    
such that  $g\cdot \iota(\KK)=\iota(\KK')$ holds. Hence,   
$\alpha_2:=(\iota|_{\KK'})^{-1}\cp(g\cdot \iota)|_\KK\colon \KK\rightarrow \KK'$ is a homeomorphism (by compactness of $\KK,\KK'$). 
\begingroup
\setlength{\leftmarginiii}{12pt}
\begin{itemize} 
\item
Since $x\in \Sigma\subset \UU\subset \LL$ holds with $\UU$ open, we can shrink $\KK$ around $x$ (and $\KK'$ around $x'$) such that additionally $\KK\subseteq \LL$ holds, hence $\hat{\Sigma}:=\KK\cup\Sigma\subseteq \LL$. 
\item
Since $x'\in \innt[\LL']$ holds, we can shrink $\KK'$ around $x'$ (and $\KK$ around $x$) such that additionally $\KK'\subseteq \LL'$ holds, hence $\hat{\Sigma}':=\KK'\cup\Sigma'\subseteq \LL'$.
\end{itemize}
\endgroup
\item
Set $\alpha_1:=\varrho$. Then,  Lemma \ref{fdggfg} shows that \begin{align*}
	\tilde{\varrho}:=(\iota|_{\hat{\Sigma}'})^{-1}\cp (g\cdot \iota)|_{\hat{\Sigma}}\colon \hat{\Sigma}\rightarrow \hat{\Sigma}'
\end{align*}
is an analytic diffeomorphism; because, in the case $S\cong\UE$ the condition in Lemma \ref{fdggfg}.\ref{fdggfg2} is fulfilled by construction (recall $\LL,\LL'\subset\innt[S]\subseteq S$). Notably, the definitions imply $\varrho=\tilde{\varrho}|_{\Sigma}$. 
\end{itemize}
\endgroup
\noindent
The right side of \eqref{nmdsnmnmddsnmnmdsdsnm} (with $\UU'\equiv \KK'$ there) shows     
 $\Sigma'\subset\hat{\Sigma}'$. Moreover, since $\hat{\Sigma}'\subseteq \LL'\in \Free(S)$ holds, Lemma \ref{jdkjfdkjfd} shows  
$\hat{\Sigma}'\in \Free(S)$.
Then, applying the first part to $\tilde{\varrho}^{-1}\colon \hat{\Sigma}'\rightarrow \hat{\Sigma}$, we obtain
\begin{align*}
	\hat{\Sigma}=\tilde{\varrho}^{-1}(\hat{\Sigma}')\in \Free(S)\qquad\text{with}\qquad \Sigma=\varrho^{-1}(\Sigma')=\tilde{\varrho}^{-1}(\Sigma')\subset \tilde{\varrho}^{-1}(\hat{\Sigma}')=\hat{\Sigma}\qquad\text{as}\qquad \Sigma'\subset\hat{\Sigma}',
\end{align*} 
which contradicts that $\Sigma$ is maximal.\qedhere
\end{enumerate}}
\endgroup
\end{proof}
\noindent
We now finally apply Lemma 37 in \cite{Deco} to prove the following statement:
\begin{lemma}
\label{dsdssddsds}
Let $\Free(S)\ni\Sigma\subset S$ be given; and let $z\in \partial_\Sigma\cap\innt[S]$ be a boundary point of $\Sigma$.   
Then, the following assertions hold:
\begingroup
\setlength{\leftmargini}{15pt}
{
\renewcommand{\theenumi}{{\alph{enumi}})} 
\renewcommand{\labelenumi}{\theenumi}
\begin{enumerate}
\item
\label{dsdssddsds1}
There exists a chart $(V,\chi)$ of $S$ around $z$ with\hspace*{\fill}(\:$\RR\ni i'<i\in \RR$\he)
\begin{align}
\label{iudiufdiufdifdopgfpogpogfg}
\begin{split}	
&\chi(z)=0,\quad\:\: 
 \chi(V)\equiv(i',i),\quad\:\:     
 \chi(V\cap \Sigma)=(i',0],\quad\:\: \chi(V\setm \Sigma)=(0,i), \quad\:\: \tilde{S}:=\Sigma\cup V\subset S,
\end{split}
\end{align} 
as well as a a strictly decreasing sequence $(0,i)\supseteq \{\tau_n\}_{n\in \NN}\rightarrow 0$ with 
\begin{align*}
\Sigma'_n:=\chi^{-1}([0,\tau_n])\in \Free(S)\qquad\quad\forall\: n\in \NN.
\end{align*}
\vspace{-17pt}
\item
\label{dsdssddsds2}
For each $g\in G\setm G_S$, we have the  implication:
\begin{align}
\label{pkdslkdslkdslklkdslkdskldssddsds454545}
	g\cdot \iota|_{\Sigma\:\cup\: \Sigma'_n}\cpsim  \iota|_{\Sigma\:\cup\: \Sigma'_n}\qquad\quad\Longrightarrow\qquad\quad g\cdot \iota|_{\Sigma}\cpsim  \iota|_{\Sigma'_n}\qquad\vee\qquad g\cdot \iota|_{\Sigma'_n}\cpsim  \iota|_{\Sigma}.
\end{align} 
\end{enumerate}}
\endgroup
\end{lemma}
\noindent
For the proof of Lemma \ref{dsdssddsds}, we need some facts and definitions from \cite{Deco}:
\begingroup
\setlength{\leftmargini}{15pt}
{
\renewcommand{\theenumi}{\small{\bf\alph{enumi})}} 
\renewcommand{\labelenumi}{\theenumi}
\begin{enumerate}
\item
\label{nvdnmorpore1}
We recall the following definitions:\hspace*{\fill}(Sect.\ 2.3.2, Definition 2, and Definition 5  in \cite{Deco})
\begingroup
\setlength{\leftmarginii}{11pt}
\begin{itemize} 
\item 
We write $\gamma\cpsim \gamma'$ for analytic immersions $\gamma\colon D\rightarrow M$ and $\gamma'\colon D'\rightarrow M$ \defff $\gamma(J)=\gamma'(J')$ holds for open intervals $J\subseteq D$ and $J'\subseteq D'$ on which $\gamma$ and $\gamma'$ are embeddings, respectively. 
\item
An analytic immersion $\gamma\colon D\rightarrow M$ is said to be a free segment \defff the following equivalence holds:
	\begin{align*}
		g\cdot\gamma\cpsim \gamma\quad\text{for}\quad g\in G\qquad\quad\Longleftrightarrow\qquad\quad g\in G_{\gamma}:=\{h\in G\:|\: h\cdot \gamma=\gamma\}.
	\end{align*}
	We recall that $G_\gamma=G_{\gamma|_{D'}}$ holds for each interval $D'\subseteq D$ (see Lemma 14 in \cite{Deco}).
\end{itemize}
\endgroup
\item
\label{nvdnmorpore2}
	Let $(U,\psi)$ be a chart of $S$ around some $z\in \innt[S]$ such that $\gamma\equiv \gamma_\psi\colon I_\psi\rightarrow M$ ($I_\psi\subseteq \RR$ an open interval) is an analytic embedding (Corollary \ref{dfdsasasasassa}):
\begingroup
\setlength{\leftmarginii}{11pt}
\begin{itemize} 
\item 
	 $\gamma_\psi|_J$ is an analytic embedding for each open interval $J\subseteq I$, hence $\iota|_C=\gamma_\psi\cp\psi|_C$ is an analytic embedding for each $U\supseteq C\in \Seg(S)$. 
\item	 
	 Let $D\subseteq I_\psi$ and $C:=\psi^{-1}(D)$. By the previous point, we have the equivalence
\begin{align*}
	g\cdot \gamma_\psi|_D\cpsim \gamma_\psi|_D\qquad\quad\Longleftrightarrow\qquad\quad  g\cdot \iota|_C\cpsim \iota|_C
\end{align*}
for each $g\in G$. 
Since $G_S=G_C=G_{\gamma_\psi|_D}=G_{\gamma_\psi}$ holds by Corollary \ref{dfdfdfdfd} and \ref{nvdnmorpore1}, we see that $C\in \Free(S)$ is equivalent to that ${\gamma_\psi}|_D$ is a free segment (in the sense defined in \ref{nvdnmorpore1}). In particular, if $\gamma|_D$ is a free segment, then $\gamma|_{D'}$ is a free segment for each interval $D'\subseteq D$ (by Lemma \ref{jdkjfdkjfd}).
\end{itemize}
\endgroup
\end{enumerate}}
\endgroup
\noindent
\begin{proof}[Proof of Lemma \ref{dsdssddsds}]
\begingroup
\setlength{\leftmargini}{15pt}
{
\renewcommand{\theenumi}{{\alph{enumi}})} 
\renewcommand{\labelenumi}{\theenumi}
\begin{enumerate}
\item
Since $\partial_\Sigma\ni z\in \innt[S]$ holds, there exists a submanifold chart $(V,\chi)$ of $\Sigma$ that is centered at $z$ such that  $\gamma\equiv \gamma_\chi\colon I_\chi\equiv (i',i)\rightarrow M$ is an analytic embedding. Replacing $\chi$ by $-\chi$, and shrinking $I_\chi$, we can assume that  \eqref{iudiufdiufdifdopgfpogpogfg} holds. Let now $i'<\ell'<0<\ell<i$ be fixed:
\vspace{-3pt}
\begingroup
\setlength{\leftmarginii}{11pt}
\begin{itemize}
\item 
Since $\chi^{-1}([\ell',0])\subseteq \Sigma\in \Free(S)$ holds, $\gamma_\chi|_{[\ell',0]}$ is a free segment by \ref{nvdnmorpore2}. 
\vspace{1pt}
\item
Lemma 37 in \cite{Deco} provides\footnote{For the assumptions made in Lemma 37 in \cite{Deco}, see Definition 1, Convention 3 and Definition 10 in \cite{Deco}.} some $0<k\leq \ell$ such that $\gamma_\chi|_{[0,k]}$ is a free segment; hence, that $\Sigma':=\chi^{-1}([0,k])\in \Free(S)$ holds by  \ref{nvdnmorpore2}. We set $\tau_n:=k/(n+1)$ for all $n\in \NN$. 
\end{itemize}
\endgroup
Then, Lemma \ref{jdkjfdkjfd} shows   that $\Sigma'_n=\chi^{-1}([0,\tau_n])\in \Free(S)$  holds  (observe $\Sigma'_n\subseteq\Sigma'$) for each $n\in \NN$. 
\item
Let $g\in G\setm G_S$ and $n\in \NN$ be such that the left side of \eqref{pkdslkdslkdslklkdslkdskldssddsds454545} holds, i.e., there exist open segments $\Seg(S)\ni \OO,\OO'\subseteq \Sigma_n$ for $\Sigma_n:=\Sigma\cup\Sigma'_n\in \Seg(S)$ such that $g\cdot \iota(\OO)=\iota(\OO')$ holds and  $g\cdot \iota|_{\OO}$, $\iota|_{\OO'}$ are embeddings. Since $\OO\subseteq \Sigma_n$ holds, we have $\OO\cap \Sigma
\neq \emptyset$ or $\OO\cap \Sigma_n'\neq \emptyset$:  
\begingroup
\setlength{\leftmarginii}{12pt}
\begin{itemize}
\item
	Assume $\OO\cap \Sigma\neq \emptyset$.  Since $g\cdot \iota|_{\OO}$, $\iota|_{\OO'}$ are embeddings, by Remark \ref{cnmnmcviufeiureiure}.\ref{cnmnmcviufeiureiure1} we can shrink $\OO,\OO'$ such that $\OO\subseteq \Sigma$ holds. If $\OO'\cap \Sigma\neq \emptyset$ holds, then Remark \ref{cnmnmcviufeiureiure}.\ref{cnmnmcviufeiureiure1} implies  $g\cdot \iota|_\Sigma\cpsim \iota|_\Sigma$. This,   however,  contradicts $g\in G\setm G_S$, as we have $\Sigma\in \Free(S)$. Consequently, $\OO'\subseteq \Sigma'_n$ holds, hence $g\cdot \iota|_\Sigma\cpsim \iota|_{\Sigma'_n}$.
\item
	Assume $\OO\cap \Sigma'_n\neq \emptyset$. Since $\Sigma'_n\in \Free(S)$ holds, 
	 an analogous argumentation as in the previous point shows $g\cdot \iota|_{\Sigma'_n}\cpsim \iota|_{\Sigma}$.\qedhere
\end{itemize}
\endgroup
\end{enumerate}}
\endgroup
\end{proof}

\section{Decompositions}
\label{oisdffiodusfdiu}
In this section, $\wm\colon G\times M\rightarrow M$ denotes a non-contractive Lie group action, and 
$(S,\iota)$ denotes a free analytic 1-submanifold of $M$ with $S\notin \Free(S)$ ($S$ admits a free segment but is not a free segment by itself). 
In this section, we show that $S$ naturally decomposes into maximal free segments that are mutually and uniquely related by the Lie group action. The explicit form of the decomposition depends on whether $S$ admits a compact maximal segment or not. In the first case, one additionally has to distinguish between the positive and the negative case. Simply put:
\begingroup
\setlength{\leftmargini}{11pt}
\begin{itemize}
\item
If $S\cong\UE$ holds, then $S$ admits a compact maximal segment, just because    
each $\Sigma\in \MFree(S)$ is necessarily compact as closed in $S$ by Lemma \ref{kjkjdskjkjdsaassaa} (observe that $\MFree(S)\neq \emptyset$ holds by Lemma \ref{fhdhh}, as $\Free(S)\neq \emptyset$ holds by assumption). Then, $S$ is build up finitely many translates of $\Sigma$ (positive/negative case can occur).
 \item
Assume that $S\cong \ID$ holds. If $S$ admits a compact maximal segment, then $S$ decomposes into countably many translates of $\Sigma$ (positive/negative case can occur). In the other case ($S$ admits no compact maximal segment), $S$ decomposes into two maximal segments (the only existing ones in this case) that are pinned together at their common boundary point (and are related to each other via the group action).
\end{itemize}
\endgroup
\subsection{Compact Maximal Segments (elementary Facts and Definitions)} 
\label{nmnbvccnnvcncvnmcncv}
In this section, we discuss the elementary properties of compact maximal segments. 
Sect.\ \ref{sdssddsdsddsdscxcxxs} contains some technical statements, and in Sect.\ \ref{lkdslkdslkdslkdslkdskldsds8s98d98dsdsds} we introduce the notion of a positive\slash negative segment. 
\subsubsection{Some Technical Preparation}
\label{sdssddsdsddsdscxcxxs}

\begin{lemma}
\label{mfdkjdkjdfjkf}
Let $\Sigma\in \Free(S)$, $C,C'\in \Seg(S)$ be given with $C\subseteq \Sigma$. Assume that $g\cdot \iota|_C=\iota\cp\varrho$ holds for some $g\in G\setm G_S$ and an analytic diffeomorphism $\varrho\colon C\rightarrow C'$. Then,   
$\Sigma\cap \varrho(C)= \partial_\Sigma\cap \partial_{\varrho(C)}$ holds, hence $\varrho(C)\subset S$.
\end{lemma}
\begin{proof}
If the claim is wrong, then Remark \ref{cnmnmcviufeiureiure}.\ref{cnmnmcviufeiureiure5} provides $ \OO'\in \Seg(S)$ open with $\OO'\subseteq \Sigma\cap \varrho(C)$. Then, $\Sigma\supseteq \OO:=\varrho^{-1}(\OO')$ holds; and we obtain 
\vspace{-6pt}
\begin{align*}
	g\cdot \iota|_{\OO}=\iota\cp\varrho|_\OO\qquad\Longrightarrow\qquad g\cdot \iota|_\Sigma\cpsim \iota|_\Sigma\qquad\stackrel{\Sigma\he\in\he \Free(S)}{\Longrightarrow}\qquad g\in G_S,
\end{align*}
which contradicts the choices. 
\end{proof}
\begin{remark}
\label{oidsoidsoidsoidsoidsoidsdsdsds454545remk}
Let $\Sigma  \in \Free(S)$, 
$\LL,\LL'\in \Seg(S)$ be given with $\LL\subseteq \Sigma$. Assume that $g\cdot \iota|_\LL=\iota\cp\varrho$ holds for some $g\in G\setm G_S$ and an analytic diffeomorphism $\varrho\colon \LL\rightarrow \LL'$. 
\begingroup
\setlength{\leftmargini}{15pt}
{
\renewcommand{\theenumi}{{\small\sf\Alph{enumi})}} 
\renewcommand{\labelenumi}{\theenumi}
\begin{enumerate}
\item
\label{nmdsnmdsnmdsnmh2}
Let $S\cong\UE$:\: Lemma \ref{compnobound} yields $g\cdot \iota|_\Sigma=\iota\cp  \ovl{\varrho}|_\Sigma$, and we set $\Sigma':=\ovl{\varrho}(\Sigma)$:\hspace*{\fill}(\he$g\cdot\iota(\Sigma)=\iota(\Sigma')$)
\vspace{-3pt} 
\begingroup
\setlength{\leftmarginii}{14pt}
{
\renewcommand{\theenumii}{{\small\sf\alph{enumii})}} 
\renewcommand{\labelenumii}{\theenumii}
\begin{enumerate}
\item
 \label{kldslkdslksdkldslkdslkdsdsdsds9898ew98ewew2}
If $\Sigma\in \MFree(S)$ holds, then $\Sigma$ is compact  (as closed in $S\cong \UE$ by Lemma \ref{kjkjdskjkjdsaassaa}).
 \item
 \label{kldslkdslksdkldslkdslkdsdsdsds9898ew98ewew1}
If $\Sigma$ is compact, then $\Sigma'=\ovl{\varrho}(\Sigma)$ is compact.
\item
 \label{kldslkdslksdkldslkdslkdsdsdsds9898ew98ewew3}
$\Sigma\cap\Sigma'= \partial_\Sigma\cap \partial_{\Sigma'}$ holds by Lemma \ref{mfdkjdkjdfjkf} (with $C\equiv \Sigma$ there).
  \item  
   \label{kldslkdslksdkldslkdslkdsdsdsds9898ew98ewew4}
 $\Sigma'\in \Free(S)$ holds by Lemma \ref{ghfhh}, whereby $\Sigma\in \MFree(S)$ implies $\Sigma'\in \MFree(S)$ (\he\he$\Sigma\subseteq S=\innt[S]$ compact\he).
\end{enumerate}}
\endgroup
\item
\label{nmdsnmdsnmdsnmh1}
Let $S\cong \ID$:\: Then,  
$g\cdot \iota|_C=\iota\cp\ovl{\varrho}|_C$ holds for $C:=\dom[\ovl{\varrho}]\cap \Sigma$, and $\Sigma\setm C$ is connected (hence, $C$ is a\\ 
\phantom{Let $S\cong \ID$:\: }boundary segment of $\Sigma$). 
\vspace{1pt}

{\it Proof.}
The first statement is clear from Proposition \ref{nobound}. For the second statement, observe that 
if $\Sigma\setm C$ is not connected, then $S=\ovl{\varrho}(C)$ holds by Proposition \ref{nobound}.\ref{nobound2},   which contradicts Lemma \ref{mfdkjdkjdfjkf}. \hspace*{\fill}$\ddagger$

We observe the following:
 \vspace{-3pt} 
\begingroup
\setlength{\leftmarginii}{14pt}
{
\renewcommand{\theenumii}{.{\small\sf\alph{enumii})}} 
\renewcommand{\labelenumii}{{\small\sf\alph{enumii})}}
\begin{enumerate} 
\item
\label{nmdsnmdsnmdsnmh1a}
If $\Sigma$ is compact, then $\iota|_{\ovl{\varrho}(C)}=g\cdot \iota|_C\cp\ovl{\varrho}^{-1}|_{\ovl{\varrho}(C)}$ is an embedding ($g\cdot \iota|_\Sigma$ is an embedding with $C\subseteq \Sigma$).
\item
\label{nmdsnmdsnmdsnmh1b}
One of the following two situations hold:
\begingroup
\setlength{\leftmarginiii}{11pt}
\begin{itemize}
\item
$C=\Sigma$: Lemma \ref{ghfhh} yields $\ovl{\varrho}(C)\in \Free(S)$, as well as  $\ovl{\varrho}(C)\in \MFree(S)$ if $\Sigma\subseteq \innt[S]$ is compact maximal.
\vspace{2pt}
\item
$C\subset \Sigma$ is a (proper) boundary segment of $\Sigma$:  
\vspace{3pt}

Proposition \ref{nobound} shows that $S\setm\ovl{\varrho}(C)$
 is connected, hence    
$\ovl{\varrho}(C)$ is a boundary segment of $S$. 
\end{itemize}
\endgroup
\item
\label{nmdsnmdsnmdsnmh1c}
Lemma \ref{mfdkjdkjdfjkf} shows   
$\Sigma\cap \ovl{\varrho}(C)= \partial_\Sigma\cap \partial_{\ovl{\varrho}(C)}
$, which  
must be empty or  singleton as $S\cong \ID$ holds. 
\item
\label{nmdsnmdsnmdsnmh1d}
If $C$ is compact, then $\ovl{\varrho}(C)$ is compact.
\end{enumerate}}
\endgroup
\end{enumerate}}
\endgroup
\end{remark}
\noindent
We recall from Sect.\ \ref{kdskjsjkdcxccx} (Conventions)  that $[g]$ denotes the class of $g\in G$ in the quotient $\GES\equiv G\slash G_S$. 
We have the following analogue to  
Proposition 4 in \cite{Deco}.
\begin{proposition}
\label{prop:shifttrans}
Let $\MFree(S)\ni \Sigma\subset S$ be compact   
and $z\in \partial_\Sigma\cap \innt[S]$. 
Then, there exists $\GES\ni [g]\neq [e]$ uniquely determined by the following condition:
\begingroup
\setlength{\leftmargini}{18pt}
{
\renewcommand{\theenumi}{{$(\star)$}} 
\renewcommand{\labelenumi}{$(\star)$}
\begin{enumerate}
\item
\label{starlkdslkdslkdslksdlkdslklkdsdsdsds}
There exists 
$\Seg(S)\ni \Sigma_b\subseteq \Sigma$ compact with $b\in\partial_\Sigma\cap \partial_{\Sigma_b}$ as well as $\Sigma_z\in \Free(S)$ compact with $\{z\}=\partial_{\Sigma}\cap \partial_{\Sigma_z}=\Sigma\cap \Sigma_z$, such that $g\cdot\iota(\Sigma_b)=\iota(\Sigma_z)$ and $g\cdot \iota(b)=\iota(z)$ holds.\footnote{Notably, $\Sigma_b\in \Free(S)$ holds by Lemma \ref{jdkjfdkjfd}; and, we have $\Sigma\cup \Sigma_z\in \Seg(S)$  with $z\in \innt[\Sigma\cup \Sigma_z]$ (fix a chart around $z$ and apply Remark \ref{cnmnmcviufeiureiure}.\ref{cnmnmcviufeiureiure6cxcxcxcx}).}
\end{enumerate}}
\endgroup
\end{proposition}
\begin{convention}
Henceforth, we use the following notations for (compact)  segments:
\vspace{-2pt}
\begingroup
\setlength{\leftmargini}{11pt}
\begin{itemize}
\item[$*$]
Given a segment    
$\Sigma_z\in \Seg(S)$, then the index indicates that $z$ is a boundary point of $\Sigma_z$ (\he i.e., $z\in \partial_{\Sigma_z}$). 
\item[$*$]
Given a boundary segment $\Sigma_b\in \Seg(S)$ of $\Sigma\in \Seg(S)$, then the index indicates that  
$b\in \partial_\Sigma\cap\partial_{\Sigma_b}$ holds. By Remark \ref{cnmnmcviufeiureiure}.\ref{bdfoudnnddhhdhdhdseg34930943094309433434434343}, this is equivalent to require $\Seg(S)\ni \Sigma_b\subseteq \Sigma\in \Seg(S)$ with $b\in \partial_\Sigma\cap\partial_{\Sigma_b}$ (\he see e.g.\ \ref{starlkdslkdslkdslksdlkdslklkdsdsdsds}).
\end{itemize}
\endgroup
\end{convention}
\begin{proof} 
{\sf Uniqueness:}  
Assume 
that \ref{starlkdslkdslkdslksdlkdslklkdsdsdsds} holds for $[g]$, $\Sigma_b$, $\Sigma_{z}$ as well as for $[g']$, $\Sigma_{b'}$, $\Sigma'_{z}$.  
Then, $g\cdot\iota|_{\Sigma_b}$, $g'\cdot\iota|_{\Sigma_{b'}}$, $\iota|_{\Sigma_z}$, $\iota|_{\Sigma'_z}$ are embeddings by compactness, so that 
$$
\alpha:=(g\cdot \iota|_{\Sigma_b})^{-1}\cp\iota|_{\Sigma_z}\colon \Sigma_z\rightarrow \Sigma_b\qquad\quad \text{and}\qquad\quad \beta:=(g'\cdot \iota|_{\Sigma_{b'}})^{-1}\cp\iota|_{\Sigma'_{z}}\colon \Sigma_{z'}\rightarrow \Sigma_{b'}
$$ 
are homeomorphisms. Since $\Sigma\cap \Sigma_z=\{z\}=\Sigma\cap \Sigma'_z$, there exists  
an open segment $\Seg(S)\ni\UU\subseteq \Sigma_z\cap \Sigma'_{z}$.\footnote{Consider a chart $(U,\psi)$ around $z$,  as well as intervals (Remark \ref{cnmnmcviufeiureiure}.\ref{cnmnmcviufeiureiure6cxcxcxcx}) $B_z,D_z,D'_z$ with $\psi(z)\in B_z\subseteq \psi(\Sigma\cap U)$, $\psi(z)\in D_z\subseteq \psi(\Sigma_z\cap U)$, $\psi(z)\in D'_z\subseteq \psi(\Sigma'_z\cap U)$.} Then, $g\cdot \iota(\OO)=\iota(\UU)=g'\cdot \iota(\OO')$ holds for $\OO:= \alpha(\UU)\subseteq \Sigma_b\subseteq \Sigma$ and $\OO' := \beta(\UU)\subseteq \Sigma_{b'}\subseteq \Sigma$, whereby  $g\cdot \iota|_\OO$, $g'\cdot \iota|_{\OO'}$ are embeddings. Hence, we have
\begin{align*}
	 g\cdot \iota|_{\Sigma_{b}}\cpsim g'\cdot \iota|_{\Sigma_{b'}}\qquad\Longrightarrow\qquad (g^{-1}\cdot g')\cdot \iota|_\Sigma\cpsim \iota|_\Sigma\qquad\Longrightarrow\qquad [g]=[g'].
\end{align*}
\vspace{-12pt}

\noindent
{\sf Existence:} 
We choose  
$(V,\chi)$, $(0,i)\supseteq \{\tau_n\}_{n\in \NN}\rightarrow 0$,  $\{\Sigma'_n\}_{n\in \NN}\subseteq \Free(S)$,   
and $\tilde{S}$ as in Lemma \ref{dsdssddsds}; hence, 
\begin{align}
\label{dskjdkjdskjdskjsdkjdsds98sd09dsdsdsdsdsdsds}
\begin{split}
&\chi(z)=0,\quad\:\: 
 \chi(V)\equiv(i',i),\quad\:\:     
 \chi(V\cap \Sigma)=(i',0],\quad\:\: \chi(V\setm \Sigma)=(0,i), \quad\:\: \tilde{S}=\Sigma\cup V\subset S\\[6pt]
&\hspace{101pt}\Sigma'_n=\chi^{-1}([0,\tau_n])\in \Free(S)\quad\:\:\:\:\forall\: n\in \NN.
\end{split}
\end{align}
Then, $\tilde{S}$ is homeomorphic to an interval by Remark \ref{cnmnmcviufeiureiure}.\ref{cnmnmcviufeiureiure2}:
\vspace{-2pt}
\begingroup
\setlength{\leftmargini}{11pt}
\begin{itemize}
\item
According to \eqref{jdkjfdkjfddf}, 
it suffices to show that there exists $[g]\in \GGG(\tilde{S})$ such that \ref{starlkdslkdslkdslksdlkdslklkdsdsdsds} holds for $\tilde{S}$ instead of $S$, i.e., with $\Seg(\tilde{S})\ni \Sigma_b\subseteq \Sigma$ compact and $\Sigma_z\in \Free(\tilde{S})$ compact.  
\item
Since $\Sigma\in \MFree(\tilde{S})$ holds by \eqref{hggffdfdxcxfdtrtrtrtrzt7678878787878787878787}, in the following we can assume that $S=\tilde{S}$ holds, i.e., that $S=\Sigma\cup V$ is homeomorphic to an interval via $\homeo\colon \ID\rightarrow S$.   
\end{itemize}
\endgroup
\noindent
We fix $n\in \NN$ and observe the following: 
\begingroup
\setlength{\leftmargini}{11pt}
\begin{itemize}
\item
Since $\Sigma$ is maximal, we have    
$\Seg(S)\ni \Sigma_n:=\Sigma\cup \Sigma'_n\notin \Free(S)$.  
Hence, there exists some $g_n\in G\setm G_S$ with 
$g_n\cdot \iota|_{\Sigma_n}\cpsim\iota|_{\Sigma_n}$, thus
\vspace{-5pt}
\begin{align*}
	g\cdot \iota|_{\Sigma\:\cup\: \Sigma'_n}\cpsim  \iota|_{\Sigma\:\cup\: \Sigma'_n}\qquad\stackrel{\text{Lemma  }\ref{dsdssddsds}.\ref{dsdssddsds2}}{\Longrightarrow}\qquad g\cdot \iota|_{\Sigma}\cpsim  \iota|_{\Sigma'_n}\qquad\vee\qquad g\cdot \iota|_{\Sigma'_n}\cpsim  \iota|_{\Sigma}.
\end{align*}
\vspace{-15pt} 
\item
Then, replacing $g_n$ by $g_n^{-1}$ if necessary, we  can assume that $g_n\cdot\iota|_{\Sigma}(\OO_n)= \iota|_{\Sigma_n}(\OO_n')$ holds for certain open segments $\OO_n\subseteq\Sigma$ and $\OO_n'\subseteq \Sigma_n'$  such that $\iota|_{\OO_n}$, $\iota|_{\OO'_n}$ are embeddings.  
Corollary \ref{fddsfd} shows that there exists 
an analytic diffeomorphism $\varrho_n\colon \Sigma\supseteq \OO_n\rightarrow \OO'_n\subseteq \Sigma_n'$ with 
\begin{align}
\label{fgghfdss}
 g_n\cdot (\iota|_{\Sigma})|_{\OO_n} = \iota|_{\Sigma_n'}\cp \varrho_n.
\end{align}
\end{itemize}
\endgroup
\noindent 
Now, there are two different cases that can occur:
\begingroup
\setlength{\leftmargini}{13pt}
\begin{itemize}
\item[$\blacktriangleright$]
Let $\bigcup_{n\in \NN}\{[g_n]\}\subseteq \GES\setm \{[e]\}$ be finite: Passing to a subsequence, we can assume that $[g_n]=[g]$ holds for all $n\in \NN$, for some $[e]\neq [g]\in \GES$. Then, \eqref{fgghfdss} yields sequences $\{x_n\}_{n\in \NN}\subseteq \Sigma$ and $\{x'_n\}_{n\in \NN}\subseteq \Sigma'_0\setm\{z\}$ such that 
\begin{align}
\label{iudsiudsuiuds}
	x'_n\in \Sigma'_n\qquad\text{as well as}\qquad g\cdot \iota(x_n)=g_n\cdot \iota(x_n)=\iota(x'_n)\qquad\text{holds for all}\qquad n\in\NN.
\end{align}
\begingroup
\setlength{\leftmarginii}{14pt}
\begin{itemize}
\item[a)]
By \eqref{dskjdkjdskjdskjsdkjdsds98sd09dsdsdsdsdsdsds}  (definition of $\Sigma_n'$), we necessarily have $\lim_n x'_n =z$; and, by compactness of $\Sigma$, we can pass to subsequences to achieve that $\lim_n x_n = b\in \Sigma$ exists. Since $\iota$, $g\cdot \iota$ are continuous, \eqref{iudsiudsuiuds} yields 
\begin{align}
\label{oiudsoiudsoiuidsiuds}
	g\cdot\iota(b)=\iota(z).
\end{align}
Then, $\{x_n\}_{n\in \NN}\subseteq \Sigma\setm\{b\}$ holds by injectivity of $g\cdot \iota$, $\iota$ as well as by $\{x'_n\}_{n\in \NN}\subseteq \Sigma'_0\setm\{z\}$,  \eqref{iudsiudsuiuds}, \eqref{oiudsoiudsoiuidsiuds}. 
\item[b)]
Let  $(U,\psi)$ and $(U',\psi')$  be submanifold charts of $\Sigma$ that are centered at $b$ and $z$, respectively. Passing to subsequences if necessary, we can assume that $\{x_n\}_{n\in \NN}\subseteq U$ as well as $\{x_n'\}_{n\in \NN}\subseteq U'$ holds.  
\item[c)]
We have $b\in \partial_\Sigma$:
\vspace{2pt}

{\it Proof.} Assume that $b\in \innt[\Sigma]\subseteq \innt[S]$ holds (i.e.\ that the claim is wrong); hence, $b,z\in \innt[S]$. Then, Point b) together with Lemma \ref{lemma:BasicAnalytt1}.\ref{lemma:BasicAnalytt11} yields open intervals $J,J'$ with $0\in J,J'$ and  an analytic diffeomorphism $\rho\colon J\rightarrow J'$ with $\rho(0)=0$ such that  
$\gamma_\psi|_J$, $\gamma_{\psi'}|_{J'}$ are embeddings,  with  
$g\cdot \gamma_\psi|_J=\gamma_{\psi'}\cp\rho$. 
Then, $b\in \innt[\Sigma]$ together with $z\in \partial_\Sigma$ implies   
$g\cdot \iota|_\Sigma\cpsim \iota|_\Sigma$, hence $g\in G_S$ as $\Sigma\in \Free(S)$. This, however, contradicts  $[g]\neq [e]$.\hspace*{\fill}$\ddagger$  
\item[d)]
According to Point c), the following two situations can occur:
\begingroup
\setlength{\leftmarginiii}{12pt}
\begin{itemize}
\item[$\triangleright$]
$b\in \partial_\Sigma\cap\partial_S$:\hspace{12.7pt} Then, $b\in U\subseteq \Sigma$ holds, with $\psi(b)=0$ and $\Psi(U\cap \Sigma)=[0,i)$ for some $i>0$. 
\vspace{3pt}
\item[$\triangleright$]
$b\in \partial_\Sigma\cap\innt[S]$:\hspace{4pt}Then, $b\in U$ holds, with $\psi(b)=0$ as well as \hspace*{\fill}(see \eqref{dskjkjdskjdsdsiudsiuiudsds8787ds87ds87dsdsdsdsdsds})
\begin{align*}
\psi(U\cap\Sigma)=[0,i)\subseteq \mathbb{H}^1\qquad \text{and}\qquad  \psi(U\setminus \Sigma)=(i',0)\qquad\text{for certain}\qquad i'<0<i.
\end{align*} 
\end{itemize}
\endgroup
\item[e)]
Since $z\in \innt[S]$ holds, Point b) together with Lemma \ref{lemma:BasicAnalytt1}.\ref{lemma:BasicAnalytt11} yields open intervals $J,J'$ with $0\in J,J'$ and an analytic diffeomorphism $\rho\colon J\rightarrow J'$ with $\rho(0)=0$ such that $\wt{\gamma}_\psi|_{J}$, $\gamma_{\psi'}|_{J'}$ are embeddings ($J'\subseteq \dom[\gamma_{\psi'}]$), with $g\cdot \wt{\gamma}_\psi(J)=\gamma_{\psi'}(J')$. 
We fix $k\in J\cap (0,i)$, and set \hspace*{\fill}(defined by Point d))
\begin{align*}
K_b:=[0,k]\subseteq \dom[\gamma_\psi],\qquad\Sigma_b:=\psi^{-1}(K_b),\qquad K_z:= \rho(K_b)\subseteq \dom[\gamma_{\psi'}],\qquad \Sigma_z:=\psi'^{-1}(K_z).
\end{align*}  
\begingroup
\setlength{\leftmarginiii}{10pt}
\begin{itemize}
\item[$\bullet$]
Since $\rho(0)=0$ holds, $K_z$ is either of the form 
\begin{align*}
[s',0]&\quad \text{for some}\quad s'<0\qquad\quad\Longrightarrow\qquad\quad \partial_{\Sigma_z}=\{\psi'^{-1}(0),\psi'^{-1}(s')\}=\{z,\psi'^{-1}(s')\}
\\
\text{or}\qquad\quad\:\: [0,s]&\quad \text{for some}\quad \hspace{3pt} s>0
\qquad\quad\Longrightarrow\qquad\quad \partial_{\Sigma_z}=\{\psi'^{-1}(0),\psi'^{-1}(s)\}\hspace{2.6pt}=\{z,\psi'^{-1}(s)\}.
\end{align*} 
In both cases, we can achieve $\partial_\Sigma\cap \partial_{\Sigma_z}=\{z\}$, just by shrinking $\Sigma_b$ (decreasing $0<k$) if necessary. 
\item[$\bullet$]
Then, $\Sigma_b,\Sigma_z$ are compact segments, with
\begin{align*}
&\Sigma_b=\psi^{-1}(K_b)\subseteq \psi^{-1}([0,i))\subseteq \Sigma\quad\:\:\:\wedge\quad\:\:\: g\cdot \iota(\Sigma_b)= g\cdot \gamma_\psi(K_b)= \gamma_{\psi'}(\rho(K_b))=\gamma_{\psi'}(K_z)=\iota(\Sigma_z)\\[2pt]
&\hspace{190pt}\text{as well as}\\[-2pt]
&\hspace{88pt}\partial_\Sigma\cap \partial_{\Sigma_b}\ni \psi^{-1}(0)=b\qquad\wedge\qquad \partial_{\Sigma}\cap \partial_{\Sigma_z}=\{z\}.
\end{align*}
In particular, $\Sigma_b$ is a compact boundary segment of $\Sigma$. 
\end{itemize}
\endgroup  
\item[f)]
Since $g\cdot \iota|_{\Sigma_b}$, $\iota|_{\Sigma_z}$ are embeddings ($\Sigma_b$, $\Sigma_z$ compact), the map  
$$
\varrho:=(\iota|_{\Sigma_z})^{-1}\cp(g\cdot \iota|_{\Sigma_b})\colon \Sigma_b\rightarrow \Sigma_z
$$
is a homeomorphism, hence an analytic diffeomorphism by  Corollary \ref{fddsfd}: 
\begingroup
\setlength{\leftmarginiii}{10pt}
\begin{itemize}
\item[$\bullet$]
Since $\Sigma_b\subseteq \Sigma \in \Free(S)$, Lemma \ref{jdkjfdkjfd} shows  
 $\Sigma_b\in \Free(S)$. Then, $\Sigma_z=\varrho(\Sigma_b)\in \Free(S)$ holds by Lemma \ref{ghfhh}.
 \vspace{2pt} 
\item[$\bullet$]
Lemma \ref{mfdkjdkjdfjkf} (with $C\equiv \Sigma_b$, hence $\varrho(C)= \Sigma_z$) together with e) shows $\Sigma\cap \Sigma_z= \partial_{\Sigma}\cap\partial_{\Sigma_z}=\{z\}$. 
\end{itemize}
\endgroup     
\end{itemize}
\endgroup 
\noindent
Altogether, this establishes  \ref{starlkdslkdslkdslksdlkdslklkdsdsdsds}.
\item[$\blacktriangleright$]
Let $\bigcup_{n\in \NN}\{[g_n]\}\subseteq \GES\setm\{[e]\}$ be infinite: Passing to a subsequence, we can assume that $[g_m]\neq [g_n]$ holds for all $\NN\ni m\neq n\in \NN$. For each $n\in  \NN$, we define
\begin{align}
\label{fgaASASAghfdss}
\begin{split}
	\CM_n:= \Sigma\cap \dom[\ovl{\varrho}_n]\qquad\quad&\text{as well as}\qquad\quad C'_n:=\ovl{\varrho}_n(C_n)\subseteq S\\[1pt]
	\text{so }& \text{that we have}\\[1pt]
g_n\cdot \iota(C_n)=\iota(C'_n)\qquad\quad &\text{as well as}\qquad\quad  g_n\cdot (\iota|_{\Sigma})|_{C_n}= \iota\cp \ovl{\varrho}_n|_{C_n},
\end{split}
\end{align}
whereby the third line holds by Proposition \ref{nobound} (recall that above, we have replaced $S$ by $\tilde{S}=\Sigma\cup V\cong \ID$).

We observe the following:
\begingroup
\setlength{\leftmarginii}{15pt}
\begin{itemize}
\item[a)]
For each $n\in \NN$, we have 
$\innt[\Sigma\cap C'_n]=\emptyset$; because,   
\begin{align*}
\innt[\Sigma\cap C'_n]\neq \emptyset\qquad&\Longrightarrow\qquad\exists\: \OO'\in \Seg(S)\:\:\text{open with}\:\:\OO'\subseteq \Sigma\cap C'_n\quad
\\[1pt]
&\Longrightarrow \qquad\qquad\quad g_n\cdot \iota(\ovl{\varrho}^{-1}_n(\OO'))=\iota(\OO')\\[1pt]
&\Longrightarrow \qquad\quad\hspace{53pt} g_n\cdot \iota|_\Sigma\cpsim  \iota|_\Sigma\\[1pt]
&\Longrightarrow \qquad\quad\hspace{67pt} [g_n]=[e],
\end{align*}
\vspace{-25pt}

\noindent
which contradicts $[g_n]\neq [e]$. 
\item[b)]
For $\NN\ni m\neq n\in \NN$, we have 
$\innt[C'_m\cap C'_n]= \emptyset$; because, 
\begin{align*}
\innt[C'_m\cap C'_n]\neq \emptyset\qquad&\Longrightarrow\qquad\exists\: \OO'\in \Seg(S)\:\:\text{open with}\:\:\OO'\subseteq C'_m\cap C'_n\\[1pt]
&\Longrightarrow \qquad g_m\cdot \iota(\ovl{\varrho}^{-1}_m(\OO'))=\iota(\OO') = g_n\cdot \iota(\ovl{\varrho}^{-1}_n(\OO'))\\[1pt]
&\Longrightarrow \qquad\hspace{47pt} g_m\cdot \iota|_\Sigma\cpsim g_n\cdot \iota|_\Sigma\\[1pt]
&\Longrightarrow \qquad\hspace{63pt} [g_m]=[g_n],
\end{align*}
\vspace{-25pt}

\noindent
which contradicts the choices. 
\end{itemize}
\endgroup
By definition, $J'_n:=\chi(\OO'_n)\subseteq \chi(\Sigma'_n) \subseteq  [0,\tau_n]$ holds ($J'_n$ an open interval) for all $n\in \NN$:
\begingroup
\setlength{\leftmarginii}{13pt}
{
\renewcommand{\theenumi}{\small{\sf\arabic{enumi})}} 
\renewcommand{\labelenumi}{\theenumi}
\begin{enumerate}
\item
\label{nmnmcnmcxnmcxnm1}
For each $n\in \NN$, there exists some $0<\varepsilon_n<\tau_n$ with $J'_n\subseteq [\varepsilon_n,\tau_n]$. 
\vspace{4pt}

{\it Proof.}
If the claim is wrong, then there exist  $\ell\in \NN$ and  $0<j\leq \tau_\ell<i$ with $J'_\ell=(0,j)$. Let $\NN\ni p>\ell$ be such large that $\tau_p< j$ holds. Then, $J'_p\subseteq J'_\ell$ holds, which contradicts $\innt[C'_\ell\cap C'_p]=\emptyset$ (see b)).\hspace*{\fill}$\ddagger$
\vspace{4pt}

Then, since $\lim_n \tau_n=0$ holds, shrinking $V$ around $z$ and passing to a subsequence if necessary, we can assume that for all $\NN\ni n\geq 1$, we have\hspace*{\fill}(\he with\: $\chi(V\cap \Sigma)=(i',0]$\: and\: $\chi(V\setm \Sigma)=(0,i)$\he)
\begin{align}
\label{dsoioidssoioidsoidsoidsoidsoids}
	0< \varepsilon_{n+1}<\tau_{n+1}<\varepsilon_{n}<\tau_{n}<{\dots}< \varepsilon_1<\tau_1<j_0<i \qquad\text{with}\qquad J'_0=(j_0,i).
\end{align}
\item
\label{nmdsnmdskjdskjdsiudsoiewoiewioiewew98we98ew98ew}
We have $C'_n\subseteq V$ for all $n\geq 1$.
\vspace{4pt}

{\it Proof.} We set $K:=[0,j_0]\subseteq (i',i)=\chi(V)$;  and let $\NN\ni n\geq 1$ be fixed: 
\begingroup
\setlength{\leftmarginiii}{12pt}
\begin{itemize}
\item
We have $\chi(\OO_n')=J_n'\subseteq K$ by \eqref{dsoioidssoioidsoidsoidsoidsoids}, hence $J_n'\subseteq \chi(C_n'\cap V)\neq \emptyset$. 
\vspace{2pt}
\item
Hence, Remark \ref{dsoidsoidsoids8ds98ds98ds98dsdsdsdsds}.\ref{sd0909dssoididsoidsoidsoidsoidsdsdsdsdsewew2} yields $C_n'\subseteq V$ once we have shown $\chi(C_n'\cap V)\subseteq K$. 
\end{itemize}
\endgroup
Assume thus that $\chi(C_n'\cap V)\not\subseteq K$ holds, i.e., that there exists $x\in C_n'\cap V$ with 
$$i'<\chi(x)<0\qquad\quad \text{or}\qquad\quad j_0<\chi(x)<i.$$ 
In both cases,  Remark \ref{cnmnmcviufeiureiure}.\ref{cnmnmcviufeiureiure6cxcxcxcx} provides an interval $D_x$ with $\chi(x)\in D_x\subseteq \chi(C_n' \cap V)$; so that  
\begingroup
\setlength{\leftmarginiii}{11pt}
\begin{itemize}
\item[$*$] 
\hspace{2.4pt}$i'<\chi(x)<0$  
implies\hspace{4.4pt} $\innt[\Sigma\cap C'_n]\neq \emptyset$, which contradicts a).
\vspace{2pt}
\item[$*$]
$j_0<\chi(x)<i$\hspace{1.6pt}  
implies   
$\innt[C_0'\cap C'_n]\neq \emptyset$, which contradicts b).\hspace*{\fill}$\ddagger$
\end{itemize}
\endgroup
\item
\label{spdopodspodsposddsds0sd98ds98s98dds9898ds98sd98sdsdds}
For $O\subseteq S$ open with $z\in O\subseteq V$, there exists $\NN\ne m_O\geq 1$ with $C'_n\subseteq O$ for all $n\geq m_O$.
\vspace{3pt}

{\it Proof.} We fix $\NN\ni m_O\geq 1$ with $K:=[0,\tau]\subseteq \chi(O)$ for $\varepsilon_{m_O-1}\leq \tau:=\inf(J'_{m_O-1})\leq \tau_{m_O-1}$, hence 
\begin{align*}
J_n'\subseteq [0,\tau_n]\subseteq [0,\tau]\subseteq K\qquad\quad\forall\: n\geq m_O.
\end{align*} 
Let now $n\geq m_O\geq 1$ be fixed. Then,  $C_n'\subseteq V$ holds by Point \ref{nmdsnmdskjdskjdsiudsoiewoiewioiewew98we98ew98ew}.  Hence, $D_n':=\chi(C_n')\subseteq (i',i)$ is an interval (as $C_n'$ is a segment)   with $J'_n\subseteq D_n'\cap K\neq \emptyset$. Then, we must have $D_n'\subseteq K=[0,\tau]$, because:
\begingroup
\setlength{\leftmarginiii}{11pt}
\begin{itemize}
\item[$*$]
$D_n'\cap (i',0)\neq \emptyset$ implies\hspace{23.6pt}  $\innt[\Sigma\cap C'_n]\neq \emptyset$, which contradicts a).
\vspace{2pt}
\item[$*$]
$D_n'\cap (\tau,i)\hspace{2.8pt}\neq \emptyset$ implies  $\innt[ C'_{m_O-1}\cap C_n']\neq \emptyset$, which contradicts b).\hspace*{\fill}$\ddagger$
\end{itemize}
\endgroup
\item
\label{nmnmcnmcxnmcxnm2767676}
There exists $m\geq m_V\geq 1$ such that $S \setminus C'_n$ is not connected for all $n\geq m$.
\vspace{3pt}

{\it Proof.}    
We fix  $i'<o'<o<i$, and set $O:=\chi^{-1}((o',o))\subseteq V$ as well as $m:=m_O\geq 1$: 
\begingroup
\setlength{\leftmarginiii}{11pt}
\begin{itemize} 
\vspace{2pt}
\item
We have $\tilde{O}:=\homeo^{-1}(O)=(\tilde{o}',\tilde{o})$  for certain $\inf(\ID)<\tilde{o}'<\tilde{o}<\sup(\ID)$.
\vspace{2pt}
\item
By Point \ref{spdopodspodsposddsds0sd98ds98s98dds9898ds98sd98sdsdds}, for $n\geq m=m_O$, we have 
\vspace{-3pt}
\[
C_n'\subseteq O\quad\text{hence}\quad
\tilde{B}_n:=\homeo^{-1}(C_n')\subseteq \homeo^{-1}(O)=\tilde{O}\quad\text{hence}\quad    
i'<\tilde{o}'\leq \inf(\tilde{B}_n)<\sup(\tilde{B}_n)\leq \tilde{o}<i.   
\]
Thus, $\homeo^{-1}(S\setminus C_n')=\ID\setminus \tilde{B}_n$ is not connected, hence $S\setminus C_n'$  
is not connected. $\hspace*{\fill}\ddagger$ 
\end{itemize}
\endgroup
\item
\label{nmnmcnmcxnmcxnm2}
We have $C_n=\Sigma$ for all $n\geq m$.
\vspace{3pt}

{\it Proof.} 
If $C_n\subset \Sigma$ holds for some $n\geq m$, then Proposition \ref{nobound}.\ref{nobound1}  
shows that $S\setm C'_n$ is connected   (observe $C_n\subseteq \Sigma$ by definition); which contradicts Point \ref{nmnmcnmcxnmcxnm2767676}. 
\hspace*{\fill}$\ddagger$  
\end{enumerate}}
\endgroup
\noindent
Let now $W$ be a neighbourhood of $\iota(z)$, and $O\subseteq V$ open with $\iota(O)\subseteq W$ ($\iota$ is continuous). By Point \ref{spdopodspodsposddsds0sd98ds98s98dds9898ds98sd98sdsdds}, there exists $n\geq m$ with $C_n'\subseteq O$, hence $\iota(C_n')\subseteq W$. We obtain
\begin{align*}
g_n\cdot \iota(\Sigma)\stackrel{\text{Point }\ref{nmnmcnmcxnmcxnm2}}{=}g_n\cdot \iota(C_n)\stackrel{\eqref{fgaASASAghfdss}}{=}\iota(C'_n)\subseteq W. 
\end{align*}
This contradicts that $\wm$ is non-contractive (as $|\Sigma|\geq 2$).\qedhere
\end{itemize}
\endgroup
\end{proof}

\subsubsection{Positive and Negative Segments}
\label{lkdslkdslkdslkdslkdskldsds8s98d98dsdsds}
\begin{definition}
\label{dslkdslkdskldefkjdsodsoidsiidsodsdsdssd}
Let $\MFree(S)\ni \Sigma\subset S$ be compact, with boundary points $\partial_\Sigma=\{z_+,z_-\}$. 
Moreover, let $z\in \{z_+,z_-\}\cap \innt[S]$ be given (recall Remark \ref{cnmnmcviufeiureiure}.\ref{bdfoudnnddhhdhdhdseg34930943094309433434434343dsdsds}), and choose $[g]$, $\Sigma_b$, $\Sigma_z$ as in \ref{starlkdslkdslkdslksdlkdslklkdsdsdsds} in Proposition \ref{prop:shifttrans}: 
\begingroup
\setlength{\leftmargini}{12pt}
\begin{itemize}
\item
$\Sigma$ is said to be\hspace{2.8pt} {\bf positive}\: {\defff}\hspace{0.1pt} $g\notin G_z$ holds (hence, $[g]\subseteq G\setm G_z$). 

Notably, the following holds true:
\vspace{-3pt}
\begingroup
\setlength{\leftmarginii}{12pt}
\begin{itemize}
\item
$\iota(z)\neq \iota(b)$,\: hence\: $z\neq b$\hspace*{\fill} (\he$\iota(z)=\iota(b)\:\:\wedge\:\: g\cdot \iota(b)=\iota(z) \:\:\Rightarrow\:\: g\in  G_z$\he)
\vspace{1pt}
\item
$z=z_\pmm\:\:\:\Leftrightarrow\:\:\: b=z_\mmp$\hspace{147.1pt}(\he since $z\neq b\in \partial_\Sigma$\he)
\end{itemize}
\endgroup
\vspace{-3pt}
\item
$\Sigma$ is said to be {\bf negative}\: \defff\hspace{0.1pt}  
 $g\in G_z$ holds (hence, $[g]\subseteq G_z$). 
 
 Notably, the following holds true:
 \vspace{-3pt}
\begingroup
\setlength{\leftmarginii}{12pt}
\begin{itemize}
\item
$\iota(b)=\iota(z)$,\: hence\: $b=z$\hspace*{\fill} (\he$g\in G_z\:\:\Rightarrow\:\: g^{-1}\in G_z\:\:\Rightarrow\:\:\iota(b)=g^{-1}\cdot \iota(z)=\iota(z)$\he)
\item
$[g]=[g^{-1}]$\hspace{172.7pt}(see Lemma \ref{dspodsopodssdklkdss98ds98ds98ds})
\end{itemize}
\endgroup
\end{itemize}
\endgroup
\end{definition}
\begin{remark}
\label{dsdsdsdskjdsjkdskjdskjdsds09w0909ew09we09ewewewewewewewcx}
Figuratively speaking,
\vspace{-4pt}
\begingroup
\setlength{\leftmargini}{12pt}
\begin{itemize}
\item
\hspace{2.5pt}positive means that $g$ shifts $\iota(\Sigma_{b})$ such that $\iota(b)\neq \iota(z)$ is mapped to $\iota(z)$.
\item
negative means that $g$ flips\hspace{5.6pt} $\iota(\Sigma_{b})$ at $\iota(b)=\iota(z)$.
\end{itemize}
\endgroup
\noindent
Evidently, $\Sigma$ is either positive or negative if only one of its boundary points is contained in $\innt[S]$. The next lemma (Lemma \ref{asklddd}) shows that the same holds true if both boundary points of $\Sigma$ are contained in $\innt[S]$.
\end{remark}
\begin{lemma}
\label{asklddd}
Let $\MFree(S)\ni \Sigma\subset \innt[S]$ be compact. Then, $\Sigma$ is either positive or negative. Specifically, if $\partial_\Sigma=\{z_+,z_-\}$ denote the two boundary points of $\Sigma$ and $[g_\pm]$ the corresponding unique classes from Proposition \ref{prop:shifttrans}, then we have the implication:
\vspace{-4pt} 
\begin{align}
\label{dskjdskjnmdsnmdsiucsiuszudsd76ds87d87ds98798dsdsdsdsds}
g_\pmm\notin G_{z_\pmm}\qquad\:\:\Longrightarrow\qquad\:\: 
g_\mmp\notin G_{z_\mmp}\quad\:\:\wedge\quad\:\: [g_\mmp]=[(g_\pmm)^{-1}] 
\end{align}
\end{lemma}
\begin{proof}
Since $\{z_+,z_-\}=\partial_\Sigma\subseteq \innt[S]$ holds, Proposition \ref{prop:shifttrans} provides  
classes $[g_\pm]$,  compact free segments $\Sigma_{z_\pm}\in \Free(S)$ and compact (boundary) segments  
$\Sigma_{b_\pm}$ of $\Sigma$,    
with
\begin{align}
\label{kfdkjjfdkkjdfd}
	g_\pm\cdot \iota(\Sigma_{b_\pm})=\iota(\Sigma_{z_\pm}),\quad\:\: g_\pm\cdot \iota(b_\pm)=\iota(z_\pm),\quad\:\:  \Sigma\cap  \Sigma_{z_\pm}=\{z_{\pm}\},\quad\:\: b_\pm\in \partial_\Sigma\cap \partial_{\Sigma_{b_\pm}}\subseteq \{z_+,z_-\}.
\end{align} 
\vspace{-19pt}
\begingroup
\setlength{\leftmargini}{15pt}
{
\renewcommand{\theenumi}{{\sf\alph{enumi})}} 
\renewcommand{\labelenumi}{\theenumi}
\begin{enumerate}
\item
\label{nmdsnmdsnmdsnmfdfdfdfcxccxdh1}
If $g_{\pmm}\notin G_{z_\pmm}$, then  $b_{\pmm}=z_{\mmp}$.
\vspace{2pt}

{\it Proof.} We have $b_{\pmm}\in \{z_+,z_-\}$ by  \eqref{kfdkjjfdkkjdfd}, whereby 
\vspace{-4pt}
\[ 
	\quad\:\: b_{\pmm}=z_{\pmm}\quad\:\Longrightarrow\quad\:\: g_{\pmm}\cdot \iota(z_\pmm)=g_{\pmm}\cdot \iota(b_\pmm)\stackrel{\eqref{kfdkjjfdkkjdfd}}{=}\iota(z_\pmm)\quad\:\:\Longrightarrow\quad\:\: g_\pmm\in G_{z_\pmm}.\:\:\qquad\ddagger
\]
\item
\label{nmdsnmdsnmdsnmfdfdfdfcxccxdh2}
Assume that the following implications hold:
\begin{align}
\label{ajaahhdsads1}
& g_+\notin G_{z_+}\qquad \Longrightarrow\qquad\:  [g_-]=[g_+^{-1}]\\
\label{ajaahhdsads2}
& g_-\notin G_{z_-}\qquad \Longrightarrow\qquad\:  [g_+]=[g_-^{-1}]
\end{align} 
\vspace{-27pt}

\noindent
Then, we additionally have 
\vspace{-5pt}
\begin{align*}
	g_{\pmm}\notin G_{z_\pmm}\quad\:&\stackrel{\ref{nmdsnmdsnmdsnmfdfdfdfcxccxdh1}}{\Longrightarrow}\qquad b_\pmm=z_\mmp\qquad\:\:\: \Longrightarrow\qquad  g_\pmm\cdot \iota(z_\mmp)\stackrel{\eqref{kfdkjjfdkkjdfd}}{=}\iota(z_\pmm)\\[5pt]
	\qquad&\Longrightarrow\qquad  g_\pmm\notin G_{z_\mmp}
	\qquad \Longleftrightarrow\qquad  g_\pmm^{-1}\notin G_{z_\mmp}\\
	&\hspace{-8.5pt} \stackrel{\eqref{ajaahhdsads1}\slash \eqref{ajaahhdsads2}}{\Longrightarrow}\quad\:  g_\mmp\notin G_{z_\mmp}.
\end{align*} 
\end{enumerate}}
\endgroup
\noindent
By Point \ref{nmdsnmdsnmdsnmfdfdfdfcxccxdh2}, the claim  follows once we have shown that \eqref{ajaahhdsads1}  and \eqref{ajaahhdsads2} hold.  
Assume thus that $g_+\notin G_{z_+}$ holds (we only prove \eqref{ajaahhdsads1}, because \eqref{ajaahhdsads2} follows analogously). 
Then,   
 $b_+=z_-$ holds by \ref{nmdsnmdsnmdsnmfdfdfdfcxccxdh1}:
\begingroup
\setlength{\leftmargini}{12pt}
\begin{itemize}
\item
We choose charts $(U_\pm,\psi_\pm)$ around $z_\pm$ with  
$U_+\subseteq \Sigma\cup \Sigma_{z_+}\in \Seg(S)$ and $U_-\subseteq \Sigma_{b_+}\cup \Sigma_{z_-}\in \Seg(S)$ (observe $b_+=z_-$).  
The first two identities in \eqref{kfdkjjfdkkjdfd} yield sequences as in Lemma \ref{lemma:BasicAnalytt1} with $g\equiv g_+$, $x\equiv b_+$, and $x'\equiv z_+$ there.
\item
Since $z_\pm\in \innt[S]$ holds (with $z_-=b_+$), Lemma \ref{lemma:BasicAnalytt1}.\ref{lemma:BasicAnalytt123} yields  
compact connected neighbourhoods $\KK_\pm\subseteq U_\pm$ (thus, $\KK_+\subseteq \Sigma\cup \Sigma_{z_+}$ and $\KK_-\subseteq \Sigma_{b_+}\cup \Sigma_{z_-}$) of $z_\pm$  
such that $g_+\cdot \iota(\KK_-)=\iota(\KK_+)$ holds; hence,
    \begin{align*}
\KK_+\cap \Sigma=\{z_+\}\:\dot\cup\: (\KK_+\setm \Sigma_{z_+})\qquad\quad\text{and}\qquad\quad \KK_-\cap \Sigma_{z_-}=\{z_-\}\:\dot\cup\: (\KK_-\setm \Sigma_{b_+}).
\end{align*}
\vspace{-15pt}
\item
Set $\tilde{b}_-:=z_+$. 
Since $\iota$, $g\cdot\iota$ are injective, we obtain
 \hspace*{\fill}($b_+=z_-$) 
\begin{align*}
	g_+^{-1}\cdot \iota(\underbrace{\KK_+\cap \Sigma}_{\displaystyle =: \tilde{\Sigma}_{\tilde{b}_-}})&= \{g_+^{-1}\cdot \iota(z_+)\}\:\dot\cup\: (g_+^{-1}\cdot\iota(\KK_+\setm \Sigma_{z_+}))\\[-23pt]
	&=\{\iota(b_+)\}\:\dot\cup\: ((g_+^{-1}\cdot\iota(\KK_+))\setm (g_+^{-1}\cdot \iota(\Sigma_{z_+})))\\
	&=\{\iota(z_-)\}\:\dot\cup\: (\iota(\KK_-)\setm \iota(\Sigma_{b_+}))\\
	&=\iota(\underbrace{\KK_-\cap\Sigma_{z_-}}_{\displaystyle =: \tilde{\Sigma}_{z_-}}).
\end{align*}
\end{itemize}
\endgroup
\vspace{-7pt}
\noindent
Thus, $g_+^{-1}\cdot \iota(\tilde{\Sigma}_{\tilde{b}_-}) = \iota(\tilde{\Sigma}_{z_-})$ holds for $\tilde{\Sigma}_{\tilde{b}_-},\tilde{\Sigma}_{z_-}\in \Seg(S)$ compact,  with  
\vspace{-15pt}
$$\Free(S)\ni\Sigma_{z_-}\supseteq\tilde{\Sigma}_{z_-}\in \Free(S)\quad\text{(Lemma \ref{jdkjfdkjfd})}\quad\:\:\:\text{as well as}\quad\:\:\: \tilde{b}_-\in\partial_\Sigma\cap \partial_{\tilde{\Sigma}_{\tilde{b}_-}}\quad\:\text{and} \overbrace{\underbrace{\Sigma\cap \tilde{\Sigma}_{z_-}}_{z_-\he \in }}^{\subseteq \: \Sigma\:\cap\: \Sigma_{z_-}=\:\{z_-\}}\hspace{-15pt}=\{z_-\};$$ 
\vspace{-10pt}

\noindent
so that $\tilde{\Sigma}_{\tilde{b}_-},\tilde{\Sigma}_{z_-}$ fulfill the requirements in \ref{starlkdslkdslkdslksdlkdslklkdsdsdsds}. The uniqueness statement in  
Proposition \ref{prop:shifttrans} thus yields $[g_+^{-1}]=[g_-]$, which establishes \eqref{ajaahhdsads1}.
\end{proof} 
\begin{corollary}
\label{dssddsds}
Let $\MFree(S)\ni \Sigma\subset \innt[S]$ be positive with boundary points $\partial_\Sigma=\{z_+,z_-\}\subset \innt[S]$, and let $[g_\pm]$ denote the corresponding (unique) classes from Proposition \ref{prop:shifttrans}. Then, we have 
	 $[g_\pm]=[g_\mp^{-1}]$. 
\end{corollary}
\begin{proof}
Clear from Lemma \ref{asklddd}.
\end{proof}
\begin{lemma}
\label{dspodsopodssdklkdss98ds98ds98ds}
Let $\MFree(S)\ni \Sigma\subset S$ be compact with $z\in \partial_\Sigma\cap \innt[S]$, and let $[g]$ denote the unique class from Proposition  \ref{prop:shifttrans}.  
Then, we have the implication 
\begin{align}
\label{ajaahhdsadfdgdgfs}
g \in G_{z}\qquad \Longrightarrow\qquad  [g]=[g^{-1}].
\end{align}
\end{lemma}
\begin{proof} 
Let $\Sigma_b,\Sigma_z\in \Seg(S)$ be as in  \ref{starlkdslkdslkdslksdlkdslklkdsdsdsds}. Since $g\in G_z$ holds by assumption,  $g\cdot \iota(b)=\iota(z)$ together with  injectivity of $\iota$ yields $b=z$, hence\hspace*{\fill}($b\in \Sigma_b\subseteq \Sigma$ and $z\in \Sigma_z$) 
\vspace{-6pt}
\begin{align*}
	z\ni \Sigma_b\cap\Sigma_z\subseteq \Sigma\cap\Sigma_z\stackrel{\text{\ref{starlkdslkdslkdslksdlkdslklkdsdsdsds}}}{=}\{z\}\qquad\Longrightarrow\qquad \Sigma_b\cap\Sigma_z=\{z\}.
\end{align*}
Since $g\cdot \iota(\Sigma_{b})=\iota(\Sigma_{z})$ holds with $\Sigma_b,\Sigma_z$ compact, we have $g\cdot \iota|_{\Sigma_{b}}=\iota\cp \varrho$ for the analytic diffeomorphism $\varrho:=\iota^{-1}\cp (g\cdot \iota|_{\Sigma_{b}})\colon \Sigma_b\rightarrow \Sigma_z$ (Corollary \ref{fddsfd}). 
Thus, fixing a chart $(U,\psi)$ around $b=z\in \innt[\Sigma_b\cup\Sigma_z] \subseteq\innt[S]$ with $U\subseteq \Sigma_b\cup\Sigma_z$, Lemma \ref{lemma:BasicAnalytt1}.\ref{lemma:BasicAnalytt123} yields compact connected neighbourhoods $\KK,\KK'\subseteq U\subseteq \Sigma_b\cup\Sigma_z$ of $b=z$ with $g\cdot \iota(\KK)=\iota(\KK')$. Then, we have
\begin{align*}
g\in G_z,\qquad g\cdot \iota(\Sigma_{b})=\iota(\Sigma_{z}),\qquad
	\KK\cap \Sigma_z=\{b\}\:\dot\cup\: (\KK\setm \Sigma_b),\qquad\KK'\cap \Sigma_b=\{z\}\:\dot\cup\: (\KK'\setm \Sigma_z).
\end{align*}
Together with injectivity of $\iota$, $g\cdot \iota$, 
this yields
\begin{align*}
	g\cdot \iota(\underbrace{\KK\cap \Sigma_{z}}_{\displaystyle=:\tilde{\Sigma}_z})&= \{g\cdot \iota(b)\}\:\dot\cup \:(g\cdot \iota(\KK\setm \Sigma_b))\\[-19pt]
	&=\{\iota(z)\}\:\dot\cup \: ((g\cdot \iota(\KK))\setm (g\cdot \iota(\Sigma_b)))\\
	&= \{\iota(z)\}\:\dot\cup \: (\iota(\KK')\setm \iota(\Sigma_z))\\
	&=\iota(\underbrace{\KK'\cap\Sigma_b}_{\displaystyle =: \tilde{\Sigma}_b}).
\end{align*}
\vspace{-7pt}

\noindent
Thus, $g^{-1}\cdot \iota(\tilde{\Sigma}_{b}) = \iota(\tilde{\Sigma}_{z})$ holds for $\tilde{\Sigma}_b,\tilde{\Sigma}_z\in \Seg(S)$ compact, with  
$\Free(S)\ni \Sigma_z\supseteq\tilde{\Sigma}_z\in \Free(S)$ 
(Lemma \ref{jdkjfdkjfd}) 
as well as $b\in\partial_\Sigma\cap \partial_{\tilde{\Sigma}_b}$ 
and $\Sigma\cap \tilde{\Sigma}_z=\{z\}$. Hence, $\tilde{\Sigma}_{b},\tilde{\Sigma}_{z}$ fulfill the requirements in \ref{starlkdslkdslkdslksdlkdslklkdsdsdsds}. The uniqueness statement in  
Proposition \ref{prop:shifttrans} thus yields $[g^{-1}]=[g]$. 
\end{proof}
\begin{lemma}
\label{qasggpapd}
Let $S'\in \Seg(S)$ be given, and let $\MFree(S')\ni \Sigma\subset S'$ be  positive w.r.t.\ $(S',\iota|_{S'})$. Then, $\Sigma$ is (compact) maximal w.r.t.\ $(S,\iota)$.  
\end{lemma}
\begin{proof}  
$\Sigma\in \Free(S)$ holds by \eqref{jdkjfdkjfddf}; and we  choose $\Sigma\subseteq \Sigma'\in \MFree(S)$ as in Lemma \ref{fhdhh}. 
Since $\Sigma$ is compact by definition, we only have to show that $\Sigma=\Sigma'$ holds. According to Remark \ref{cnmnmcviufeiureiure}.\ref{bdfoudnnddhhdhdhdseg3493094309430943343443434398ds98ds98ds98dsds},  this follows once we have shown $\partial_{\Sigma}=\partial_\Sigma'$. Write  
$\partial_\Sigma=\{z_+,z_-\}$.  
By definition, at least one boundary  point $z_\pm$ of $\Sigma$ is contained in $\innt[S']$. W.l.o.g., we can assume that $z_+\in \partial_\Sigma\cap \innt[S']$ holds, and choose $[g_+]\in \GGG(S')$, $\Sigma_{b_+}\subseteq \Sigma$, $\Sigma_{z_+}\subseteq S'$ as in Proposition \ref{prop:shifttrans} (applied to $(S',\iota|_{S'})$ and $\Sigma\in \MFree(S')$). Then, we have 
\begin{align}
\label{oidsoioidsdoidoiewweewewew}
	g_+\cdot \iota(\Sigma_{b_+})=\iota(\Sigma_{z_+}),\qquad g_+\cdot \iota(b_+)=\iota(z_+),\qquad \Sigma\cap  \Sigma_{z_+}=\{z_{+}\},\qquad b_+\in \partial_\Sigma\cap \partial_{\Sigma_{b_+}},
\end{align}
whereby  
$b_+ = z_-$ holds as $\Sigma$ is positive w.r.t.\ $(S',\iota|_{S'})$:
\begingroup
\setlength{\leftmargini}{11pt}
\begin{itemize}
\item
We have $z_+\in \partial_{\Sigma'}$:   
\vspace{1pt}

{\it Proof.}
If $z_+\in \innt[\Sigma']$ holds, then  
  \eqref{oidsoioidsdoidoiewweewewew} implies $g_+\cdot \iota|_{\Sigma'}\cpsim \iota|_{\Sigma'}$. Since $\Sigma'\in \MFree(S)\subseteq \Free(S)$ holds,  this contradicts $g_+\notin G_{S'}=G_S$ (Corollary \ref{dfdfdfdfd}).
  \hspace*{\fill}$\ddagger$ 
\item
We have $z_-\in \partial_{\Sigma'}$:\:    
\vspace{1pt}

{\it Proof.}
Assume that $b_+=z_-\in \innt[\Sigma']\subseteq\innt[S]$ holds, hence $z_\pm\in \innt[S]$. Let $(U_\pm,\psi_\pm)$ be fixed charts around $z_\pm$ with $U_+\subseteq \Sigma\cup\Sigma_{z_+}$ and $U_-\subseteq \innt[\Sigma']$. Then, \eqref{oidsoioidsdoidoiewweewewew} and Lemma  \ref{lemma:BasicAnalytt1}.\ref{lemma:BasicAnalytt123}  yield connected compact neighbourhoods $\KK_{\pm}\subseteq U_{\pm}$ of $z_\pm$ such that $g_+\cdot \iota(\KK_-)=\iota(\KK_+)$ holds. Then, 
\begin{align*}
	\KK_+\cap \Sigma=\{z_+\}\:\dot\cup\: (\KK_+\setm \Sigma_{z_+})\qquad\quad\text{as well as}\qquad\quad \emptyset\neq \innt[\KK_-\setminus \Sigma_{b_+}]\subseteq \innt[\Sigma']
\end{align*}
holds, and we obtain\hspace*{\fill}($b_+=z_-$)
\begin{align*}
	g_+^{-1}\cdot\iota(\KK_+\cap \Sigma)&=\{g_+^{-1}\cdot\iota(z_+)\} \:\dot\cup\: g_+^{-1}\cdot\iota(\KK_+\setm \Sigma_{z_+})\\
	&=\{\iota(b_+)\} \:\dot\cup\: ((g_+^{-1}\cdot\iota(\KK_+))  \setm (g_+^{-1}\cdot\iota(\Sigma_{z_+})))\\
	&=\{\iota(z_-)\} \:\dot\cup\: (\iota(\KK_-)  \setm \iota(\Sigma_{b_+}))\\
	&=\{\iota(z_-)\} \:\dot\cup\: \iota(\KK_- \setm\Sigma_{b_+}).
\end{align*}
 This implies $g_+\cdot \iota|_{\Sigma'}\cpsim \iota|_{\Sigma'}$, which  contradicts $g_+\notin G_{S'}=G_S$ (Corollary \ref{dfdfdfdfd}) as $\Sigma'\in \Free(S)$ holds. 
\qedhere
\end{itemize}
\endgroup
\end{proof}
\begin{lemma}
\label{fdfdfdfd}
Let $g_1, \dots, g_n\in \overlap(S)$ and $g_1', \dots, g_n'\in G$ be given. Then, 
\begin{align}
\label{qwepokfdjkhfd}
\:[g_k]=[g'_k]\quad \text{for}\quad k=1,\dots, n\qquad\quad \Longrightarrow\qquad\quad 
[(g_n\cdot {\dots}\cdot g_1)^{\pm}]= [(g'_n\cdot {\dots}\cdot g'_1)^\pm].
\end{align}
\end{lemma}
\begin{proof}
	Let $Q$ denote the set of all finite products of elements in $\overlap(S)$:
\begingroup
\setlength{\leftmargini}{11pt}
\begin{itemize}
\item
	The right side of \eqref{dslkjkjskjds} implies that $Q$ is a subgroup of $G$. 
\item	 
	Corollary \ref{aaaafsdfsfdfs} implies that $G_S\subseteq Q$ is a normal subgroup of $Q$,   
	hence $Q\slash G_S$ is a group.
\item
If $g\in Q$, then the class $[g]$ of $g$ in $\GES$ equals the class $[g]_{*}$ of $g$ in $Q\slash G_S$.
\end{itemize}
\endgroup
\noindent
Now, for each $k\in \{1,\dots,n\}$, we have $g_k'=g_k\cdot h_k$ for some $h_k\in G_S$.  Hence,  $g_k' \in \overlap(S)$ holds for $k=1,\dots,n$, and thus $g_1'\cdot {\dots}\cdot g_n' \in Q$. We obtain
\begin{align}
\begin{split}
[(g'_n\cdot {\dots}\cdot g'_1)^\pm]&=[(g'_n\cdot {\dots}\cdot g'_1)^\pm]_*=([g'_n\cdot {\dots}\cdot g'_1]_*)^\pm\\
&=([g_n']_*\cdot{\dots}\cdot [g_1']_*)^\pm\\
&=([g_n]_*\cdot{\dots}\cdot [g_1]_*)^\pm\\
&=([g_n\cdot {\dots}\cdot g_1]_*)^\pm=[(g_n\cdot {\dots}\cdot g_1)^\pm]_*=[(g_n\cdot{\dots}\cdot g_1)^\pm].
\qedhere
\end{split}
\end{align}
\end{proof}
\subsection{$\Sigma$-Decompositions}
\label{kjdskjsdkjsdoidsoidsds98s98ds98ds9898ds09ds09ds09dssdds} 
In this section, we introduce the notion of a $\Sigma$-decomposition, which is done separately for the case 
 $S\cong \UE$ (Sect.\ \ref{fghgfasa1}) and $S\cong\ID$ (Sect.\ \ref{fghgfasa2}).

\subsubsection{Compact without Boundary ($S\cong \UE$)}
\label{fghgfasa1}
In this subsection, we discuss the case $S\cong\UE$  ($S$ is compact without boundary).
\begin{remark}
\label{kjfdjlkfdjkfd}
Let $S\cong \UE$ be free with $S\notin \Free(S)$.  
 \vspace{-4pt}
 \begingroup
\setlength{\leftmargini}{11pt}
\begin{itemize}
\item 
	 Lemma \ref{fhdhh} yields $S\notin\MFree(S)\neq \emptyset$. 
\item	
	Each $\Sigma\in \MFree(S)$ is necessarily compact (as closed in $S$ by Lemma \ref{kjkjdskjkjdsaassaa}) with $\Sigma\subset S=\innt[S]$.
\end{itemize}
\endgroup
\end{remark} 
\begin{definition}
\label{gfgfgf}
Let $\MFree(S)\ni \Sigma\subset S\cong \UE$.  
A $\Sigma$-decomposition of $S$ is a (finite)  collection of (compact) maximal segments $\Sigma_0,\dots,\Sigma_n\in \MFree(S)$ with $n\geq 1$ as well as classes $[g_0],\dots,[g_n]\in \GES$ such that the following conditions are fulfilled:
\begingroup
\setlength{\leftmargini}{15pt}
{
\renewcommand{\theenumi}{{\bf\small\arabic{enumi})}} 
\renewcommand{\labelenumi}{\theenumi}
\begin{enumerate}
\item
\label{defo00}
$\Sigma_0=\Sigma$,\quad $S=\bigcup_{k=0}^n\Sigma_k$, \quad $g_k\cdot \iota(\Sigma_0)=\iota(\Sigma_k)$\:\: for\:\: $k=0,\dots,n$.
\item
\label{defo1}
\begingroup
\setlength{\leftmarginii}{12pt}
\begin{itemize}
\item
$n=1$:\quad $\Sigma_0\cap\Sigma_1=\partial_{\Sigma_0}\cap \partial_{\Sigma_1}$ consists of the two boundary points of $\Sigma_0$ and $\Sigma_1$.
\item
$n\geq 2$:\quad $\Sigma_p\cap \Sigma_q=\partial_{\Sigma_p}\cap \partial_{\Sigma_q}$ is singleton for $|p-q|\in \{1,n\}$, and empty for $2\leq |p-q|\leq n-1$.
\end{itemize}
\endgroup
\end{enumerate}}
\endgroup
\end{definition}
\begin{remark}
\label{sdffdsfsd}
Assume that we are in the situation of Definition \ref{gfgfgf}.
\begingroup
\setlength{\leftmargini}{15pt}
{
\renewcommand{\theenumi}{{\sf\arabic{enumi})}} 
\renewcommand{\labelenumi}{\theenumi}
\begin{enumerate}
\item
\label{sdffdsfsddsiuiudsiusd}
It follows inductively from Point \ref{defo1} that 
there exist $z_0,\dots,z_{n}\in S$ mutually different such that $\partial_{\Sigma_k}=\{z_k,z_{k+1}\}$ holds for $k=0,\dots,n$ if we set $z_{n+1}:=z_0$. In particular,
\begingroup
\setlength{\leftmarginii}{12pt}
\begin{itemize}
\item
$n=1$:\quad $\Sigma_0\cap\Sigma_1\hspace{11pt}=\partial_{\Sigma_0}\cap \partial_{\Sigma_1}\hspace{9.5pt}=\{z_0,z_1\}$.
\item
$n\geq 2$:\quad $\Sigma_k\cap\Sigma_{k+1}=\partial_{\Sigma_k}\cap \partial_{\Sigma_{k+1}}=\:\{z_{k+1}\}$\quad for\quad $k=0,\dots,n$\quad if we set\quad $\Sigma_{n+1}:=\Sigma_0$.
\end{itemize}
\endgroup
More concretely, it is straightforward from Remark \ref{dslkdslkdslklkdslkdslkslkslkdsdsds} that there exist $0=\alpha_0<{\dots}<\alpha_n=2\pi$ and a homeomorphism $\homeo\colon \UE\rightarrow S$ with 
\begin{align}
\label{dsnmdshjdshjdsiuduids86s76sdds879ds898ds98ds98dsds}
z_k=\homeo(\e^{\I \alpha_k})\qquad\text{and}\qquad\Sigma_k=\homeo(\e^{\I[\alpha_k,\alpha_{k+1}]})
 \qquad\text{for}\qquad k=0,\dots, n-1.\quad
\end{align}  
\item
\label{sdffdsfsd11}
We have $[g_0]=[e]$:
$$
g_0\cdot \iota(\Sigma_0)\stackrel{\ref{defo00}}{=}\iota(\Sigma_0)\qquad\stackrel{\Sigma_0\:{\rm compact}}{\Longrightarrow}\qquad g_0\cdot \iota|_{\Sigma_0}\cpsim\iota|_{\Sigma_0}\qquad\:\:\stackrel{\Sigma\:\in\:\Free(S)}{\Longrightarrow}\qquad\:\: g_0\in G_S.
$$ 
\item
\label{sdffdsfsd1}
The classes $[g_0],\dots,[g_n]$ are mutually different:
\[
[g_p]= [g_q]\quad\text{for}\quad 0\leq p, q\leq n
\qquad\:\:\stackrel{\iota\:{\rm injective}}{\Longrightarrow}\qquad\:\:   
\Sigma_p=\Sigma_q
\qquad\quad\stackrel{\ref{defo1}}{\Longrightarrow}\qquad\quad p=q.
\]
\item
\label{sdffdsfsd199}
If\:  $1\leq p\leq n$, then
\begin{align}
\label{dsnmdsdszucxzuhjcxkjdsiuwiueew433443433443}
h_p\cdot \iota(\Sigma_{p-1})=\iota(\Sigma_p)\qquad\text{holds for}\qquad h_p:=g_{p}\cdot g_{p-1}^{-1}.\qquad\qquad
\end{align}
\item
\label{sdffdsfsd19993}
In Definition \ref{gfgfgf},  
the requirement that the segments $\Sigma_1,\dots,\Sigma_n$ are compact maximal, can be replaced by the requirement that $\iota|_{\Sigma_k}$ is an embedding for $k=1,\dots,n$.
\vspace{3pt}

{\it Proof.} 
The one direction is clear.  
Assume thus that $\iota|_{\Sigma_k}$ is an embedding for some $1\leq k\leq n$. Then, $\varrho:=(\iota|_{\Sigma_k})^{-1}\cp (g\cdot \iota|_{\Sigma_0})$ is defined by \ref{defo00} as well as a homeomorphism by compactness of $\Sigma_0$. Hence, $\Sigma_k$ is compact, and $\varrho$ is an analytic diffeomorphism by Corollary \ref{fddsfd}. Since $\Sigma_0=\Sigma\subseteq S=\innt[S]$ holds (recall $S\cong\UE$), Lemma \ref{ghfhh} shows $\Sigma_k\in \MFree(S)$.\hspace*{\fill}$\ddagger$
\vspace{3pt}

In Sect.\ \ref{fghgfasa2}, we will use the above (alternative)  condition to define $\Sigma$-decompositions in the case   
$S\cong \ID$.  
The reason is that (in the case $S\cong \ID$) a 
$\Sigma$-decomposition can also contain  
 segments  
 (boundary segments) that are not necessarily compact, see Definition \ref{deco}. 
\item
\label{sdffdsfsd3}
If $n=1$, then $[g_1]=[g_1^{-1}]$. 
\vspace{3pt}

{\it Proof.}
If $\Sigma_0$ is negative, then the claim is clear from \eqref{ajaahhdsadfdgdgfs} in Lemma \ref{dspodsopodssdklkdss98ds98ds98ds}. Assume thus  that $\Sigma_0$ is positive, and let $[g_\pm]$ 
be as in Proposition \ref{prop:shifttrans} (applied to $z\equiv z_\pm$ and $\Sigma\equiv\Sigma_0$). According to Definition \ref{gfgfgf}, we have $S=\Sigma_0\cup\Sigma_1$ with $\Sigma_0\cap \Sigma_1=\{z_+,z_-\}$ and $g_1\cdot \iota(\Sigma_0)=\iota(\Sigma_1)$. Then, \ref{starlkdslkdslkdslksdlkdslklkdsdsdsds} in Proposition \ref{prop:shifttrans} yields 
\begin{align}
\label{nmdskjdskjdsiudsiudiudsiudkjdskjds87ds8787dsdsds}
\begin{split}
g_\pm \cdot \iota|_{\Sigma_0}\cpsim \iota|_{\Sigma_1}\qquad\:\:&\Longrightarrow\qquad\:\:
g_\pm \cdot \iota|_{\Sigma_0}\cpsim g_1\cdot  \iota|_{\Sigma_0}\\
\qquad\:\:&\Longrightarrow\qquad\:\:  (g_1^{-1}\cdot g_\pm) \cdot \iota|_{\Sigma_0}\cpsim  \iota|_{\Sigma_0}
 \qquad\stackrel{\Sigma_0\:\in\: \Free(S)}{\Longrightarrow}\qquad [g_\pm]=[g_1].
\end{split}
\end{align}
Corollary \ref{dssddsds} (second step) and Lemma \ref{fdfdfdfd} (third step) yield
\vspace{-3pt}
\[
\hspace{161pt}[g_1]\stackrel{\eqref{nmdskjdskjdsiudsiudiudsiudkjdskjds87ds8787dsdsds}}{=}[g_+]=[g_-^{-1}]\stackrel{\eqref{nmdskjdskjdsiudsiudiudsiudkjdskjds87ds8787dsdsds}}{=}[g_1^{-1}].\hspace{160pt}\ddagger
\]
\item
\label{sdffdsfsd2}
Let $0\leq p\leq n$ be given, and set $\Sigma':=\Sigma_p$.  We define $\Sym_{n+1}\ni \zeta\colon \{0,\dots,n\}\rightarrow \{0,\dots,n\}$ by
\begin{align*}
\zeta(k):=
\begin{cases} 
	p+k &\:\mbox{for }\:\:\:\:\quad\:\:\hspace{21pt} 0\leq k\leq n-p\\ 
	k- (n-p+1) & \:\mbox{for }\:\:\:\: n-p+1\leq k\leq n.
\end{cases} 
\end{align*}
We obtain a $\Sigma'$-decomposition $\Sigma'_0,\dots,\Sigma'_n$, $[g'_0],\dots,[g'_n]$ of $S$, if we set
\begin{align}
\label{fdggf}
	\Sigma'_k:=\Sigma_{\zeta(k)}\qquad\text{and}\qquad [g'_k]:=[g_{\zeta(k)}\cdot g^{-1}_{p}]\qquad\text{for}\qquad k=0,\dots,n.
\end{align} 
In the case $n=p=1$, we obtain from  Lemma \ref{fdfdfdfd} in the second step 
\vspace{-4pt}
\begin{align*}
	[g'_1]\stackrel{\eqref{fdggf}}{=}[g_0\cdot g_1^{-1}]\stackrel{{\rm Point}\:\ref{sdffdsfsd11}}{=}[e\cdot g_1^{-1}]=[g_1^{-1}]\stackrel{{\rm Point}\:\ref{sdffdsfsd3}}{=}[g_1].
\end{align*}
\end{enumerate}}
\endgroup
\end{remark}
\begin{lemma}
\label{asdhlkdsajkd}
Let\: $\MFree(S)\ni \Sigma\subset S\cong \UE$  
be given, with $\Sigma$-decomposition as in Definition \ref{gfgfgf}.  
Then, for each $\Sigma'\in \Seg(S)$ and $g\in G$, we have the  implication:
\begin{align}
\label{fdsdfd}
 g\cdot \iota(\Sigma_0)=\iota(\Sigma')\quad\:\:\Longrightarrow\quad\:\: [g]=[g_k]\:\:\text{ as well as }\:\:\Sigma'=\Sigma_k\:\:\text{ holds for }\:\: k\in \{0,\dots,n\}\:\:\text{ unique.}
\end{align}
In particular, the following assertions hold:
\begingroup
\setlength{\leftmargini}{15pt}
{
\renewcommand{\theenumi}{{\alph{enumi})}} 
\renewcommand{\labelenumi}{\theenumi}
\begin{enumerate}
\item
\label{nmdsnmdnmdssdkjskjkjdskjdsds7ds87ds87ds87dszuiu1}
If $n=1$, then there exists no other $\Sigma$-decomposition of $S$.
\item
\label{nmdsnmdnmdssdkjskjkjdskjdsds7ds87ds87ds87dszuiu2}
If $n\geq 2$, then the only other $\Sigma$-decomposition of $S$ is given by 
\begin{align}
\label{dfddsf}
\ovl{\Sigma}_0=\Sigma,\: [\ovl{g}_0]=[g_0]\quad\:\:\text{as well as}\quad\:\:
	\ovl{\Sigma}_{k}:=\Sigma_{\zeta(k)},\: [\ovl{g}_k]:=[g_{\zeta(k)}]\quad\:\:\text{for}\quad\:\: k=1,\dots,n,
\end{align}
with $\zeta(0):=0$ and $\zeta(k):=n-(k-1)$ for $k=1,\dots,n$.\footnote{Observe that in the case $n=1$, the definitions would yield  $\zeta=\id_{\{0,1\}}$, which is in line with Part \ref{nmdsnmdnmdssdkjskjkjdskjdsds7ds87ds87ds87dszuiu1}.}
\end{enumerate}}
\endgroup 
\end{lemma} 
\begin{proof}
The uniqueness statement in \eqref{fdsdfd} holds, because the classes $[g_0],\dots,[g_n]$  are mutually different by Remark \ref{sdffdsfsd}.\ref{sdffdsfsd1}. For the implication in \eqref{fdsdfd}, assume that the left side of \eqref{fdsdfd} holds:  
\begingroup
\setlength{\leftmargini}{12pt}
\begin{itemize}
\item
There exists $0\leq k\leq n$ with $\innt[\Sigma']\cap \Sigma_k\neq \emptyset$, because $S=\bigcup_{k=0}^n\Sigma_k$ holds by \ref{defo00}. Remark \ref{cnmnmcviufeiureiure}.\ref{cnmnmcviufeiureiure5} implies 
\vspace{-4pt}
\begin{align}
\label{dsdsnnmdsdsjsdiudsiusiudsds87ds8787dsdsdsds}
\iota|_{\Sigma'}\cpsim \iota|_{\Sigma_k}\qquad\quad\stackrel{\ref{defo00}}{\Longrightarrow}\qquad\quad \iota|_{\Sigma'}\cpsim g_k\cdot \iota|_{\Sigma_0}.
\end{align} 
\vspace{-20pt}
\item
Since $\Sigma_0\in \Free(S)$ is free (third step), we obtain from the previous point (first implication):
\vspace{-3pt}
\begin{align*}
g\cdot \iota(\Sigma_0)=\iota(\Sigma')\qquad&\stackrel{\eqref{dsdsnnmdsdsjsdiudsiusiudsds87ds8787dsdsdsds}}{\Longrightarrow}\qquad g\cdot \iota|_{\Sigma_0}\cpsim  g_k\cdot \iota|_{\Sigma_0}\qquad \Longrightarrow\qquad (g_k^{-1}\cdot g)\cdot \iota|_{\Sigma_0}\cpsim  \iota|_{\Sigma_0}\\[2pt]
& \Longrightarrow\qquad\hspace{20pt} [g]=[g_k] \qquad\hspace{21pt}\Longrightarrow\qquad\hspace{51pt} \Sigma'=\Sigma_k
\end{align*}
\vspace{-22pt}
\end{itemize}
\endgroup
\noindent
Finally, let $\Sigma'_0,\dots,\Sigma'_{n'}$, $[g'_0],\dots,[g'_{n'}]$ be  another $\Sigma$-decomposition of $S$: 
\begingroup
\setlength{\leftmargini}{12pt}
\begin{itemize}
\item
By definition, we have $\Sigma_0'=\Sigma_0$, and   
Remark \ref{sdffdsfsd}.\ref{sdffdsfsd11} yields $[g_0']=[e]=[g_0]$. 
\item
Applying \eqref{fdsdfd} to both decompositions, yields $n'=n$, and furthermore that there exists a unique bijection $S_{n+1}\ni \zeta \colon \{0,\dots,n\}\rightarrow \{0,\dots,n\}$ with $\zeta(0)=0$ and (use injectivity of $\iota$)
\begin{align*}
(\Sigma'_k,[g'_k])=(\Sigma_{\zeta(k)},[g_{\zeta(k)}])\qquad\quad\forall\:1\leq k\leq n.
\end{align*}
\vspace{-22pt}
\end{itemize}
\endgroup
\noindent
From this, Point \ref{nmdsnmdnmdssdkjskjkjdskjdsds7ds87ds87ds87dszuiu1} is immediate; so that it remains to prove Point  \ref{nmdsnmdnmdssdkjskjkjdskjdsds7ds87ds87ds87dszuiu2}. Let thus $n\geq 2$. 
Since $\Sigma_0=\Sigma=\Sigma_0'$ holds, 
Point \ref{defo1} in 
Definition \ref{gfgfgf} implies that we either have $\Sigma'_1=\Sigma_1$ or $\Sigma'_1=\Sigma_n$:
\begingroup
\setlength{\leftmargini}{12pt}
\begin{itemize}
\item
If $\Sigma'_1=\Sigma_1$ holds, then Point \ref{defo1} implies $\Sigma_2'=\Sigma_2$, because we have $\Sigma'_2\neq \Sigma_0$ by $\Sigma_0=\Sigma_0'$. 

It follows inductively that $\zeta(k)=k$ holds $k=1,\dots,n$.
\item
If $\Sigma'_1=\Sigma_n$ holds, then Point \ref{defo1} implies $\Sigma_2'=\Sigma_{n-1}$, because we have $\Sigma'_2\neq \Sigma_0$ by $\Sigma_0=\Sigma_0'$.
 
It follows inductively that $\zeta(k)=n-(k-1)$ holds for $k=1,\dots,n$.\qedhere
\end{itemize}
\endgroup	
\end{proof}
\begin{lemma}
\label{dspopoddspopodspodspodspodspodsopdspodspodsopdspodsds}
Let $\MFree(S)\ni  \Sigma,\tilde{\Sigma}\subset S\cong \UE$ be given, with
\begingroup
\setlength{\leftmargini}{12pt}
\begin{itemize}
\item 
$\Sigma$-decomposition $\Sigma_0,\dots,\Sigma_n$ and\: $[g_0],\dots,[g_n]$ of $S$,
\item 
$\tilde{\Sigma}$-decomposition $\tilde{\Sigma}_0,\dots,\tilde{\Sigma}_{\tilde{n}}$ and\: $[\tilde{g}_0],\dots,[\tilde{g}_{\tilde{n}}]$ of $S$. 
\end{itemize}
\endgroup
\noindent
Then,   $\tilde{n}=n$ holds.
\end{lemma}
\begin{proof}
We only prove $\tilde{n}\leq n$, because $n\leq \tilde{n}$ follows analogously:  
\begingroup
\setlength{\leftmargini}{12pt}
\begin{itemize}
\item
By Point \ref{defo00}, we have $S=\bigcup_{k=0}^{\tilde{n}}\tilde{\Sigma}_k$, hence $\innt[\Sigma_0]\cap \tilde{\Sigma}_p\neq \emptyset$ for some $0\leq p\leq \tilde{n}$.   
Remark \ref{cnmnmcviufeiureiure}.\ref{cnmnmcviufeiureiure1} yields 
\begin{align}
\label{kjdskjdsnmdsnmdsnbdszdsdsds98ds98ds98ds98dsdsds}
\begin{split}
\iota|_{\Sigma_0}\cpsim \iota|_{\tilde{\Sigma}_p}\qquad\quad&\Longrightarrow\qquad\quad \hspace{37.5pt}
\tilde{g}_p^{-1}\cdot \iota|_{\Sigma_0}\cpsim \iota|_{\tilde{\Sigma}_0}\qquad\quad\\[2pt]
&\Longrightarrow\qquad\quad
\underbrace{(\tilde{g}_\ell\cdot (\tilde{g}_p)^{-1})}_{\displaystyle =: q_\ell}\:\cdot \:\iota|_{\Sigma_0}\cpsim \iota|_{\tilde{\Sigma}_\ell}\qquad\quad\forall\: 0\leq \ell\leq \tilde{n}.
\end{split}
\end{align}
\vspace{-10pt}
\item
We have $\{[q_0],\dots, [q_{\tilde{n}}]\}\subseteq \{[g_0],\dots,[g_n]\}$, because Lemma \ref{compnobound} together with \eqref{kjdskjdsnmdsnmdsnbdszdsdsds98ds98ds98ds98dsdsds} yields for $\ell=0,\dots, \tilde{n}$:
\begin{align*}
q_\ell\cdot \iota(\Sigma_0)=\iota(\Sigma')\quad\text{for} \quad \Seg(S)\ni \Sigma' :=\iota^{-1}(q_\ell\cdot \iota(\Sigma_0)) 
\qquad\quad
\stackrel{\eqref{fdsdfd}}{\Longrightarrow}\qquad\quad  [q_\ell]\in \{[g_0],\dots,[g_n]\}
\end{align*}
\item
The classes $[q_0],\dots, [q_{\tilde{n}}]$ are mutually different, hence  $\tilde{n}\leq n$ holds by the previous point:
\vspace{1pt}

{\it Proof.} Assume that $[q_\ell]=[q_{\ell'}]$ holds for some  $0\leq \ell,\ell'\leq \tilde{n}$, hence $q_\ell=q_{\ell'}\cdot h$ for some $h\in G_S$.     
We have $\hat{g}:= (\tilde{g}_p)^{-1} \in \overlap(S)$  by \eqref{dslkjkjskjds}, hence $\hat{h}:=\conj_{\hat{g}}(h)\in G_S$ by Corollary \ref{aaaafsdfsfdfs}. We  obtain
$$
\tilde{g}_\ell \cdot \hat{g}=q_\ell=q_{\ell'}\cdot h= (\tilde{g}_{\ell'} \cdot \hat{g})\cdot h
= \tilde{g}_{\ell'} \cdot \hat{h}\cdot \hat{g}
\qquad\Longrightarrow\qquad \tilde{g}_\ell=\tilde{g}_{\ell'} \cdot \hat{h} \qquad\Longrightarrow\qquad [\tilde{g}_{\ell}]=[\tilde{g}_{\ell'}].
$$
The right side implies $\ell=\ell'$, because the classes $\{[\tilde{g}_0],\dots,[\tilde{g}_{\tilde{n}}]\}$ are mutually different by Remark \ref{sdffdsfsd}.\ref{sdffdsfsd1}. 
\qedhere 
\end{itemize}
\endgroup
\end{proof}
\begin{definition}
\label{dskjdslkjdskjdsnmnmdsdsdsoidsoioidsds98ds9898ds0ds09ds}
Let $S\cong \UE$, and assume that there exists some $\MFree(S)\ni  \Sigma\subset S$ that admits a $\Sigma$-decomposition $\Sigma_0,\dots,\Sigma_n$, $[g_0],\dots,[g_n]$ of $S$. In this case, we define 
 the length of $S$ by $\lent{S}:=n\in \NN_{\geq 1}$. According to   Lemma \ref{dspopoddspopodspodspodspodspodsopdspodspodsopdspodsds}, this is well defined (\he{}i.e.\ independent of the explicit choice of $\Sigma$\he). 
\end{definition}
\noindent
We finally have to prove the existence of a $\Sigma$-decomposition for some given (necessarily compact) maximal segment $\MFree(S)\ni \Sigma\subset S\cong \UE$. For this, we first  observe the following:
\begin{remark}
\label{sdoidoidsoidsoidsoioidsoidsoisddsds98s98s98d98ds98}
Let\:  
$\MFree(S)\ni \Sigma_-\subset S\cong\UE$ be given, with  
boundary points $\partial_{\Sigma_-}=\{z_-,z_+\}$. Then, there exist $\GES\ni [h_+]\neq [e]$ and $\Sigma_+\in \MFree(S)$ compact  with
\begin{align}
\label{nmdsnmdsdsiiudsiudsds887ds87ds76ds76ds76s76dsdsds}
z_+\in \Sigma_-\cap\Sigma_+= \partial_{\Sigma_-}\cap \partial_{\Sigma_+}\qquad\text{as well as}\qquad  h_+\cdot \iota(\Sigma_-)=\iota(\Sigma_+).
\end{align} 
{\it Proof.} 
We set $z:=z_+$ and $\Sigma:=\Sigma_-$, and  
choose $[g]$, $\Sigma_b$, $\Sigma_z$ as in \ref{starlkdslkdslkdslksdlkdslklkdsdsdsds} in Proposition \ref{prop:shifttrans}:
\begingroup
\setlength{\leftmargini}{12pt}
\begin{itemize}
\item
We have $\varrho(\Sigma_b)=\Sigma_z$ for the analytic diffeomorphism\footnote{Observe that $\iota|_{\Sigma_b},\iota|_{\Sigma_z}$ are embeddings by compactness of $\Sigma_b,\Sigma_z$. Hence,   $\varrho$ is an analytic diffeomorphism by Corollary \ref{fddsfd}.} $\varrho:= (\iota|_{\Sigma_z})^{-1}\cp (g\cdot \iota|_{\Sigma_b})\colon \Sigma_b\rightarrow \Sigma_z$.
\item
Set $\LL:= \Sigma_b\subseteq \Sigma_-$, $\LL':=\Sigma_z$, and   $g:=h_+$.  Then, Remark   \ref{oidsoidsoidsoidsoidsoidsdsdsds454545remk}.\ref{nmdsnmdsnmdsnmh2}   
shows that 
\begin{align*}
\Sigma_+:=\ovl{\varrho}(\Sigma_-)\quad \text{is compact, with }\quad
\Sigma_+\in \MFree(S), \quad\Sigma_-\cap\Sigma_+= \partial_{\Sigma_-}\cap \partial_{\Sigma_+},\quad  h_+\cdot \iota(\Sigma_-)=\iota(\Sigma_+).
\end{align*}
We additionally 
  have $z_+\in \Sigma_-\cap\Sigma_+$, 
 because 
$
\Sigma_-\ni z_+=z\in \Sigma_z = \varrho(\Sigma_b)\subseteq\ovl{\varrho}(\Sigma_-)=\Sigma_+$ holds. 
\hspace*{\fill}$\ddagger$
\end{itemize}
\endgroup
\end{remark}
\noindent
Now, let $\Sigma$ be as above,    
write $\partial_{\Sigma}=\{z_0,z_1\}$, and fix a homeomorphism $\homeo\colon \UE\rightarrow S$ with\hspace*{\fill}(Remark \ref{dslkdslkdslklkdslkdslkslkslkdsdsds}) 
$$
\homeo(\e^{\I [\varphi_0,\varphi_1]})=\Sigma\quad\wedge\quad
\homeo(\e^{\I \varphi_0})=z_0\quad\wedge\quad \homeo(\e^{\I \varphi_1})=z_1\qquad    
\text{for (unique) angles}\qquad 0=\varphi_0<\varphi_1< 2\pi.$$
We let $\CI$ denote the set of all pairs $\mu\equiv (\{\alpha_k\}_{0\leq k\leq \ci},\{g_k\}_{0\leq k\leq \ci-1})$, with $\len{\mu}:=\ci\in \NN_{\geq 1}\cup\{\infty\}$, $\{\alpha_k\}_{0\leq k\leq \ci}\subseteq [0,2\pi]$ and $\{g_k\}_{0\leq k\leq \ci-1}\subseteq G$, such that the following conditions are fulfilled:
\begingroup
\setlength{\leftmargini}{15pt}
{
\renewcommand{\theenumi}{\small{\bf\alph{enumi})}} 
\renewcommand{\labelenumi}{\theenumi}
\begin{enumerate}
\item
\label{dsdsdskjsdkjdssdnbsdnbdsjdszudszusddzus8d98sd9898ds98dds1}
$\alpha_0=\varphi_0$\quad$\wedge$\quad $\alpha_1=\varphi_1$\qquad\:\:{}as well as\qquad\:\:\hspace{6.8pt}$\alpha_{k}<\alpha_{k+1}$\quad\:{}for\quad $0\leq k\leq\ci$.
\item
\label{dsdsdskjsdkjdssdnbsdnbdsjdszudszusddzus8d98sd9898ds98dds2}
$\Sigma_k:=\homeo(\e^{\I[\alpha_k,\alpha_{k+1}]})\in \MFree(S)$
\quad{}with\quad\:$g_k\cdot \iota(\Sigma_0)=\iota(\Sigma_k)$\quad{}for\quad$0\leq k\leq \ci-1$. 
\end{enumerate}}
\endgroup
\noindent
For $1\leq \hat{\ci}\leq \ci$, we set $\mu|_{\hat{\ci}}:=\{(\alpha_k,g_k)\}_{0\leq k\leq \hat{\ci}}$. 
\vspace{6pt}

\noindent
We let  
 $\CI$ be partially ordered by:
 \vspace{-5pt}
\begin{align*}
	\mu\leq \tilde{\mu}\quad\text{for}\quad \mu,\tilde{\mu}\in \CI \qquad\quad\stackrel{{\rm def.}}{\Longleftrightarrow}\qquad\quad \len{\mu}\leq \len{\tilde{\mu}}\quad\wedge\quad \mu=\tilde{\mu}|_{\len{\mu}}.
\end{align*}
Then, $\CI\neq \emptyset$ holds (consider $\ci=1$), and each chain in $\CI$ admits the obvious upper bound. Hence, Zorn's lemma provides a maximal element $\CI\ni \mu \equiv (\{\alpha_k\}_{0\leq k\leq \ci},\{g_k\}_{0\leq k\leq \ci-1})$:
\begingroup
\setlength{\leftmargini}{18pt}
{
\renewcommand{\theenumi}{\small{\bf\Alph{enumi})}} 
\renewcommand{\labelenumi}{\theenumi}
\begin{enumerate}
\item
\label{wclm1}
We have $\ci\in \NN_{\geq 1}$.

{\it Proof.}
If the claim is wrong, then 
$\ci=\infty$ holds; hence, 
$\alpha:=\lim_n\alpha_n\leq 2\pi$ exists by monotonicity of $\{\alpha_n\}_{n\in \NN_{\geq 1}}$. Then, to each neighbourhood $V\subseteq \UE$ of $s:=\e^{\I \alpha}$, there  exists some $N\in \NN$ with $\homeo^{-1}(\Sigma_{N})=\e^{(\I[\alpha_{N},\alpha_{N+1}])}\subseteq V$. Since $\iota\cp\homeo$ is continuous, to each neighbourhood $U$ of $x:=(\iota\cp\homeo)(s)$, there thus  exist some $N\in \NN$ with $g_N\cdot \iota(\Sigma_0)=\iota(\Sigma_N)\subseteq U$. This, however,   contradicts that $\wm$ is non-contractive. 
\hspace*{\fill}$\ddagger$
\item
\label{wclm2}
We have $\alpha_\ci=2\pi$ (in particular, $\ci \geq 2$ as $\alpha_1=\varphi_1<2\pi$).

{\it Proof.}  Remark \ref{sdoidoidsoidsoidsoioidsoidsoisddsds98s98s98d98ds98} applied to $\Sigma_-\equiv \Sigma_{\ci-1}$ and $z_+\equiv \homeo(\e^{\I \alpha_\ci})$ yields some $\GES\ni [h_{\ci}]\neq [e]$ and some $\Sigma_{\ci}\in \MFree(S)$, such that 
\begin{align}
\label{dskjdskjdsiudszudszudszuds76ds76ds76ds76dsdsdsds}
z_+\in \Sigma_{\ci-1}\cap\Sigma_{\ci}= \partial_{\Sigma_{\ci-1}}\cap \partial_{\Sigma_{\ci}}\qquad\text{as well as}\qquad  h_+\cdot \iota(\Sigma_{\ci-1})=\iota(\Sigma_{\ci})\qquad\text{holds}.
\end{align} 
\begingroup
\setlength{\leftmarginii}{12pt}
\begin{itemize}
\item
For $g_{\ci}:=h_+\cdot g_{\ci-1}$, we have
\vspace{-8pt}
\begin{align}
\label{dsnmnmdsdshcxcxiucxikjwlkew122121wqqqw989sd9898dssdcx}
g_{\ci}\cdot \iota(\Sigma_0)=h_+\cdot (g_{\ci-1} \cdot \iota(\Sigma_0))=h_+\cdot \iota(\Sigma_{\ci-1})\stackrel{\eqref{dskjdskjdsiudszudszudszuds76ds76ds76ds76dsdsdsds}}{=}\iota(\Sigma_{\ci}).
\end{align}
\item
Choose (Remark \ref{dslkdslkdslklkdslkdslkslkslkdsdsds}) $\RR\ni \varphi'<\varphi\in \RR$ with 
\begin{align}
\label{podsoidskjdskjdskjdsds8798ds89ds98dsdsdsdsd}
\Sigma_{\ci}=\homeo(\e^{\I [\varphi',\varphi]})\qquad\wedge\qquad \partial_{\Sigma_{\ci}}= \{\homeo(\e^{\I \varphi'}),\homeo(\e^{\I \varphi})\}.
\end{align} 
Then, the left side of \eqref{dskjdskjdsiudszudszudszuds76ds76ds76ds76dsdsdsds} implies $\alpha_{\ci}=\varphi'+2\pi\cdot \ell'$ for some $\ell'\in \ZZ$, and we set  
$\alpha_{\ci+1}:=\varphi+2\pi\cdot \ell'$. Then, we have
\vspace{-7pt} 
\begin{align}
\label{sddsnmdsnmdskjsdsdiuwiuoiuewewwepowepoewwekjsdkj87sd87s9898s}
\alpha_{\ci}<\alpha_{\ci+1}\qquad\text{with}\qquad \homeo(\e^{\I [\alpha_\ci,\alpha_{\ci+1}]})=\Sigma_{\ci}\in \MFree(S).
\end{align}
{\it Proof of the Implication.} 
By \eqref{podsoidskjdskjdskjdsds8798ds89ds98dsdsdsdsd} and the left side of \eqref{dskjdskjdsiudszudszudszuds76ds76ds76ds76dsdsdsds}, we have\hspace*{\fill}($z_+=\e^{\I \alpha_\ci}$)
\begin{align*}
\text{either}\qquad& 
z_+=\homeo(\e^{\I (\varphi')})\qquad \Longrightarrow\qquad \alpha_\ci=\varphi'+2\pi\cdot \ell'\quad\text{for some}\quad \ell'\in \ZZ\\
\text{or}\qquad& 
z_+=\homeo(\e^{\I (\varphi)})\qquad\hspace{2.8pt}  \Longrightarrow\qquad\alpha_\ci=\varphi\hspace{2.6pt}+2\pi\cdot \ell\hspace{3pt}\quad\text{for some}\quad \ell\hspace{2.7pt}\in \ZZ.
\end{align*} 
In the second case (second line), for $\alpha_{\ci+1}:=\varphi'+2\pi\cdot \ell$, we have 
$$
\alpha_{\ci+1}<\alpha_{\ci}\qquad\text{with}\qquad   \homeo(\e^{\I [\alpha_{\ci+1},\alpha_{\ci}]})=\Sigma_{\ci}.\hspace{36.8pt}
$$
Since $\Sigma_{\ci-1}= \homeo(\e^{\I [\alpha_{\ci-1},\alpha_{\ci}]})$ holds, this contradicts 
the left side of  \eqref{dskjdskjdsiudszudszudszuds76ds76ds76ds76dsdsdsds} (specifically, that $\Sigma_{\ci-1}\cap\Sigma_{\ci}$ contains at most two elements).
\hspace*{\fill}$\ddagger$
\item
If $\alpha_\ci < 2\pi$, then $\alpha_{\ci+1}\leq 2\pi$.
\vspace{3pt}

\noindent
{\it Proof.} If $\alpha_{\ci+1}> 2\pi$ holds, then we have
\begin{align*}
	\iota|_{\Sigma_{\ci}}\cpsim \iota|_{\Sigma_0}\qquad
	&\stackrel{\phantom{\Sigma_0\:=\:\Sigma\:{\rm free}}}{\Longrightarrow}\quad\:\: g_{\ci}\cdot \iota|_{\Sigma_0}\cpsim \iota|_{\Sigma_0}\\
	\quad\:\:&\stackrel{\Sigma_0\:=\:\Sigma\:{\rm free}}{\Longrightarrow}\quad\:\: g_{\ci}\in G_S\\
	\qquad&\stackrel{\phantom{\Sigma_0\:=\:\Sigma\:{\rm free}}}{\Longrightarrow}\quad\:\: \iota(\Sigma_{\ci})\stackrel{\eqref{dsnmnmdsdshcxcxiucxikjwlkew122121wqqqw989sd9898dssdcx}}{=} g_{\ci}\cdot \iota(\Sigma_0)=\iota(\Sigma_0)\\
	\qquad &\stackrel{\phantom{\Sigma_0\:=\:\Sigma\:{\rm free}}}{\Longrightarrow}\quad\:\: 
	\Sigma_{\ci}=\Sigma_0\\
%	%
	\qquad &\stackrel{\phantom{\Sigma_0\:=\:\Sigma\:{\rm free}}}{\Longrightarrow}\quad\:\: \e^{\I [\alpha_\ci,\alpha_{\ci+1}]}=\e^{\I [\alpha_0,\alpha_{1}]},
\end{align*}
which implies $\alpha_m=2\pi$ as $0=\alpha_0<\alpha_1<\alpha_\ci\leq 2\pi$ holds.
\hspace*{\fill}$\ddagger$ 
\end{itemize}
\endgroup 
Assume now that $\alpha_\ci<2\pi$ holds (i.e., that the claim is wrong).  
Then, $\alpha_{\ci+1}\leq 2\pi$ holds by the third point, hence  
$\tilde{\mu}:=(\{\alpha_k\}_{0\leq k\leq \ci+1},\{g_k\}_{0\leq k\leq \ci})\in \CI$  holds by \eqref{sddsnmdsnmdskjsdsdiuwiuoiuewewwepowepoewwekjsdkj87sd87s9898s}. Since $\mu <\tilde{\mu}$ holds by construction ($\mu=\tilde{\mu}|_\ci$ with $\len{\mu}=\ci <\ci+1=\len{\tilde{\mu}}$), this contradicts maximality of $\mu$. 
\end{enumerate}}
\endgroup
\noindent
Altogether, we obtain the following statement:
\begin{proposition}
\label{ghdgddgf}
Let $S\cong \UE$ be free with $S\notin \Free(S)$.    
\begingroup
\setlength{\leftmargini}{15pt}
{
\renewcommand{\theenumi}{{\arabic{enumi})}} 
\renewcommand{\labelenumi}{\theenumi}
\begin{enumerate}
\item
\label{ghdgddgf1}
We have $\MFree(S)\neq \emptyset$, and each $\Sigma\in \MFree(S)\neq \emptyset$ is compact with $\Sigma\subset S=\innt[S]$. 
\item
\label{ghdgddgf2}
Each  $\Sigma\in \MFree(S)$ admits a $\Sigma$-decomposition that is unique in the sense of Lemma \ref{asdhlkdsajkd}.
\item
\label{ghdgddgf3}
The length $\lent{S}\in \NN_{\geq 1}$ of $S$ is a well-defined quantity.
\end{enumerate}}
\endgroup 
\end{proposition}
\begin{proof}
\begingroup
\setlength{\leftmargini}{15pt}
{
\renewcommand{\theenumi}{{\arabic{enumi})}} 
\renewcommand{\labelenumi}{\theenumi}
\begin{enumerate}
\item
Clear from Remark \ref{kjfdjlkfdjkfd}. 
\item
Let $\mu\in \CI$ be maximal and $\ci:=\len{\mu}$. Then, the Points \ref{wclm1} and \ref{wclm2} show that the classes $[g_0],\dots,[g_\ci]$ and the (compact) maximal segments $\Sigma_0,\dots,\Sigma_\ci$ form a $\Sigma$-decomposition of $S$. 
\item
Clear from Point \ref{ghdgddgf1} and Lemma \ref{dspopoddspopodspodspodspodspodsopdspodspodsopdspodsds}  (see Definition \ref{dskjdslkjdskjdsnmnmdsdsdsoidsoioidsds98ds9898ds0ds09ds}).\qedhere 
\end{enumerate}}
\endgroup
\end{proof}

\subsubsection{Homeomorphic to an Interval ($S\cong \ID$)}
\label{fghgfasa2}
In this subsection, we assume that $S\cong\ID$ holds ($S$ is   homeomorphic to an interval). 

\noindent
Let $\CN$ denote the set of all subsets of $\ZZ$ that are of the form\footnote{Of course, if $\cn_-=-\infty$ holds, then both $\cn_-\leq n$ and $\cn_-<n$ means $n\in \ZZ$. Analogously, if $\cn_+=\infty$ holds, then both $n\leq \cn_+$ and $n< \cn_+$ means $n\in \ZZ$.} 
\begin{align*}
	\cn=\{n\in \ZZ \: | \: \cn_{-} \leq n \leq \cn_+\}\qquad\text{for}\qquad  \cn_-,\cn_+ \in \ZZ_{\neq 0}\cup \{-\infty,\infty\} \qquad\text{with}\qquad \cn_-< 0< \cn_+.
\end{align*}
For $\cn\in \CN$ fixed, we define
$$
\cnN:=\cn\cap \ZZ_{\neq 0}\qquad\:\:\text{as well as}\qquad\:\: \cnK:= \{n\in \ZZ\:|\: \cn_-<n<\cn_+ \}.\quad
$$
\begin{definition}
\label{deco}
Assume that $S\cong \ID$ holds,      
and let $\MFree(S)\ni\Sigma\subset\innt[S]$ be compact. 
A $\Sigma$-decomposition of $S$ is a pair $(\{\Sigma_n\}_{n\in \cn},\{[g_n]\}_{n\in \cn})$ (with $\cn\in \CN$) 
that consists of classes 
$\{[g_n]\}_{n\in \cn}\subseteq \GES$ as well as segments $\{\Sigma_n\}_{n\in \cn}\subseteq \Seg(S)$ on which $\iota$ is an embedding, such that the following holds:
\begingroup
\setlength{\leftmargini}{15pt}
{
\renewcommand{\theenumi}{{\bf\small\arabic{enumi})}} 
\renewcommand{\labelenumi}{\theenumi}
\begin{enumerate}
\item
\label{ddefo0}
$\Sigma_0=\Sigma$, and $\Sigma_{\cn_\pm}$ is a boundary segment of $S$ if $\cn_{\pm}\neq \pm\infty$ holds.
\item
\label{ddefo1}
$\Sigma_p\cap\Sigma_q=\partial_{\Sigma_p}\cap \partial_{\Sigma_q}$ is singleton for $|p-q|=1$, and empty for $|p-q|\geq 2$.
\item
\label{ddefo2}
$g_n\cdot \iota(\Sigma_0)=\iota(\Sigma_n)$ holds for all $\cn_-<n<\cn_+$, and we have
\begin{align} 
\label{kjdskjdsnmdsnmdsnmdsiudsiuds987dsds98ds98ds98dsdsds}
	g_{\cn_{-}}\cdot \iota(\Sigma_{-})=\iota(\Sigma_{\cn_-})\quad\text{if}\quad \cn_-\neq -\infty\qquad\:\:\text{as well as}\qquad\:\: g_{\cn_{+}}\cdot \iota(\Sigma_{+})=\iota(\Sigma_{\cn_+})\quad\text{if}\quad \cn_+\neq \infty
\end{align}
for certain boundary segments $\Sigma_\pm$ of $\Sigma_0$.
\end{enumerate}}
\endgroup
\end{definition}
\begin{remark}
\label{dfxhbghg}
Assume that we are in the situation of Definition \ref{deco}.
\begingroup
\setlength{\leftmargini}{15pt}
{
\renewcommand{\theenumi}{{\sf\arabic{enumi})}} 
\renewcommand{\labelenumi}{\theenumi}
\begin{enumerate}
\item
\label{sd9898dsoiewioewuizewzuuz34387kjredsmx}
It follows inductively from Point \ref{ddefo1} that there exists $\{z_n\}_{n\in \cnN}\subseteq S$   
with
\begin{align}
\label{iuiuiujkjhggftfdfdfd54545454545454545trrttrtrtr}
\{z_n\}=
\begin{cases} 
	\Sigma_{n+1}\cap \Sigma_n=\partial_{\Sigma_{n+1}}\cap\partial_{\Sigma_n} &\:\:\mbox{for}\quad\:\: \cn\ni  n\leq -1\\ 
	\Sigma_{n-1}\cap \Sigma_n=\partial_{\Sigma_{n-1}}\cap\partial_{\Sigma_n} &\:\:\mbox{for}\quad\:\: 1\leq n\in \cn. 
\end{cases} 
\end{align}
Moreover, let $\homeo\colon \ID\rightarrow S$ be a fixed  homeomorphism with $\homeo^{-1}(z_{-1})<\homeo^{-1}(z_1)$, and define $\{a_n\}_{n\in \cnN}\subseteq \ID$ by $a_n:=\homeo^{-1}(z_n)$ for all $n\in \cnN$. 
Then, the points \ref{ddefo1} and \ref{ddefo2} imply the following:
\begingroup
\setlength{\leftmarginii}{11pt}
\begin{itemize}
\item  
 $a_m<a_n$\qquad\hspace{1.7pt}\:\:\: holds for all\quad\: $\cnN  \ni m<n\in \cnN$,
 \vspace{2pt}
\item 
 $\homeo(A_n)=\Sigma_n$\quad\: holds for all\quad\hspace{2.2pt} $n\in \cn$,\quad\: with\:\footnote{For the second line, observe that $\Sigma_{\cn_\pm}$ are boundary segments of $S$.}
\begin{align*}
A_n:=[a_{n-1},a_{n}]\quad\text{for}\quad \cn_-< n\leq -1, \qquad\quad A_0:=[a_{-1},a_1],\qquad\quad A_n:=[a_n,a_{n+1}]\quad\text{for}\quad 1\leq n<\cn_+,\\[6pt]
A_{\cn_-}:=\ID\cap(-\infty,a_{\cn_-}]\quad\text{if}\quad \cn_-\neq -\infty\qquad\text{and}\qquad A_{\cn_+}:=\ID\cap [a_{\cn_+},\infty)\quad\text{if}\quad \cn_+\neq \infty.\hspace{36pt}
\end{align*}
\end{itemize}
\endgroup   
\item
We have $[g_0]=[e]$:
$$
g_0\cdot \iota(\Sigma_0)\stackrel{\ref{ddefo1}}{=}\iota(\Sigma_0)\qquad\stackrel{\iota|_{\Sigma_0}\hspace{2pt}{\rm embedding}}{\Longrightarrow}\qquad g_0\cdot \iota|_{\Sigma_0}\cpsim\iota|_{\Sigma_0}\qquad\:\:\stackrel{\Sigma\:\in\:\Free(S)}{\Longrightarrow}\qquad\:\: g_0\in G_S.
$$
\item
\label{dfxhbghg9}
We have the implication:\hspace*{\fill}(\he$\iota$ injective\he)
\vspace{-3pt}
\begin{align*}
	[g_m]=[g_n]\quad\text{for}\quad \cn_-\ni m< n\in \cn_+\qquad\:\:\stackrel{\ref{ddefo2}}{\Longrightarrow}\qquad\:\: -\infty<\cn_-=m<n=\cn_+<\infty. 
\end{align*}
\item 
\label{dfxhbghgFree}
By assumption, $\iota|_{\Sigma_n}$ is an embedding for each $n\in \cn$, and we have $\Sigma_0=\Sigma\subset\innt[S]$. Hence, Corollary \ref{fddsfd} and Lemma \ref{ghfhh} show:
\vspace{-1pt}
\begingroup
\setlength{\leftmarginii}{11pt}
\begin{itemize}
\item
We have $\{\Sigma_n\}_{n\in \cn}\subseteq \Free(S)$.
\vspace{1pt}
\item 
We have the implication (the left side in particular holds for all $\cn_-<n<\cn_+$)
\begin{align*}
g_n\cdot \iota(\Sigma_0)=\iota(\Sigma_n)\quad \text{for}\quad n\in \cn\qquad\quad\Longrightarrow\qquad\quad \Sigma_n\:\:\text{compact}\quad\wedge\quad \Sigma_n\in \MFree(S). 
\end{align*}
\vspace{-22pt}
\end{itemize}
\endgroup
\item
\label{dfxhbghg998ds98ds}
We define  $\{h_n\}_{n\in \cnN}\subseteq G$  by
\begin{align}
\label{dsndsnmdsnmdsjdskjdskjdsoidsoidsoisdds98ds9898ds98dsdsdd}
\begin{split}	
 h_n:=g_{n}\cdot g_{n+1}^{-1}\qquad\forall\: \cn\ni n\leq -1 \qquad
\quad&\text{as well as}\qquad\quad	
	h_n:=g_n\cdot g_{n-1}^{-1}  \qquad\hspace{6pt}\forall\: 1\leq n\in \cn\\
	&\hspace{10pt}\text{hence}\\
g_n=h_n\cdot {\dots}\cdot h_{-1} \qquad\forall\: \cn\ni n\leq -1	
	 \qquad
\quad&\text{as well as}\qquad
\quad g_n=h_n\cdot {\dots}\cdot h_1 \qquad\forall\: \cn\ni n\geq 1.\hspace{2.2pt}
\end{split}
\end{align}
Then, we have 
\begin{align}
\label{dasposaposao988mdsjdskjdskjdsoidsoidsoisd98ds98dsdsdd11}
\begin{split}
h_n\cdot \iota(\Sigma_{n+1})=\iota(\Sigma_n)\qquad\forall\: \cn_{-}< n\leq -1\qquad\quad&\text{and}\qquad\quad  
h_n\cdot \iota(\Sigma_{n-1})=\iota(\Sigma_n)\qquad\forall\: 1\leq n<\cn_+\\
\text{as }&\text{well as}\\
\: h_{\cn_{-}}\cdot \iota(\Sigma'_{-})=\iota(\Sigma_{\cn_-})\quad\text{if}\quad \cn_-\neq -\infty\qquad\quad&\text{and}\qquad\quad h_{\cn_{+}}\cdot \iota(\Sigma'_{+})=\iota(\Sigma_{\cn_+})\quad \text{if}\quad \cn_+\neq \infty
\end{split}
\end{align}
for the boundary segment\:  
	$\Sigma'_\pm:=((\iota|_{\Sigma_{\cn_\pm \mp1}})^{-1}\cp (g_{\cn_{\pm}\mp 1}\cdot \iota|_{\Sigma_0}))(\Sigma_\pm)$\: of\: $\Sigma_{\cn_\pm\mp 1}$. 
\item
\label{dfxhbghg3}
Let $\cn_-<p<\cn_+$ be given, and set $\Sigma':=\Sigma_p$ as well as 
$$
\cn'_\pm:=\cn_\pm -p
\qquad\text{and}\qquad
\cn':=\{n\in \ZZ\: | \: \cn'_{-} \leq n \leq \cn'_+\}\in \CN.
$$ 
Then, $(\{\Sigma'_n\}_{n\in \cn'},\{[g'_n]\}_{n\in \cn'})$ defined by
\begin{align*}
\Sigma'_n:=\Sigma_{n+p}\quad\:\:\text{and}\qquad [g'_n]:=[g_{n+p}\cdot g^{-1}_{p}]\qquad\quad\forall\: n\in \cn'
\end{align*}
is a $\Sigma'$-decomposition of $S$.
\item
\label{dfxhbghg2}
Define $\ocn_\pm:=-\cn_\mp$ as well as  
$\ocn:=\{n\in \ZZ\: | \: \ocn_{-} \leq n \leq \ocn_+\}\in \CN.$ 
Then, $(\{\ovl{\Sigma}_n\}_{n\in \ocn},\{[\ovl{g}_n]\}_{n\in \ocn})$ defined by
\begin{align*}
 \ovl{\Sigma}_n:=\Sigma_{-n}\quad\:\:\text{as well as}\qquad [\ovl{g}_n]:=[g_{-n}]\qquad\quad\forall\: n\in \ocn
\end{align*} 
is a $\Sigma$-decomposition of $S$.
\item
\label{dfxhbghg1}
We have $S=\bigcup_{n\in \cn}\Sigma_n$.
\vspace{-5pt} 
\begin{proof}
Let $\homeo\colon \ID\rightarrow S$ and $\{A_n\}_{n\in \cn}$ be as in Point \ref{sd9898dsoiewioewuizewzuuz34387kjredsmx}.  
It suffices to show $\ID=D:=\bigcup_{n\in \cn}A_n$. If $-\infty<\cn_-<\cn_+<\infty$ holds, this is immediate from Point \ref{sd9898dsoiewioewuizewzuuz34387kjredsmx}.  
Assume thus that $\cn_+=\infty$ holds (the case $\cn_-=-\infty$ is treated analogously) and that there exists some $a_1<t\in \ID\setminus D$.   
Then, $\lim_{n\rightarrow \infty} a_n=a\leq t\in \ID$ holds for some $a\in D$, and we conclude as follows:
\vspace{-3pt}
\begingroup
\setlength{\leftmarginii}{11pt}
\begin{itemize}
\item
Given $\varepsilon>0$, there exists some $n_\varepsilon\geq 1$ with $A_{n_\varepsilon}\subseteq [a-\varepsilon,a]$.
\item
Since $\iota$ is continuous, for each neighbourhood $U$ of $(\iota\cp\homeo)(a)$, there exists some $\varepsilon>0$ with $$(\iota\cp\homeo)([a-\varepsilon,a])\subseteq U
\qquad\quad\Longrightarrow\qquad\quad g_{n_U}\cdot \iota(\Sigma_0)=\iota(\Sigma_{n_\varepsilon})=\iota((\iota\cp\homeo)(A_{n_\varepsilon}))\subseteq U. 
$$  
\end{itemize}
\endgroup   
Since $|\Sigma_0|\geq 2$ holds, the second point contradicts that   
$\wm$ is non-contractive.
\end{proof}
\end{enumerate}}
\endgroup	
\end{remark}
\begin{lemma}
\label{dfggfgf}
Assume that $S\cong \ID$ holds,      
and let $\MFree(S)\ni\Sigma\subset\innt[S]$ be compact with $\Sigma$-decomposition as in Definition \ref{deco}. Let  $g\in G$, and let $\Seg(S)\ni\OO,\OO'\subseteq S$  be open segments such that $\iota|_\OO, \iota|_{\OO'}$ are embeddings with $\OO\subseteq \Sigma_0$. Then,  
 the following implication holds:
\begin{align}
\label{dfgfg}
 (g\cdot \iota|_{\Sigma_0})(\OO)=\iota(\OO')\qquad\Longrightarrow\qquad [g]=[g_n]\:\:\text{ and }\:\:\OO'\subseteq \Sigma_n\:\:\text{ holds for }\:\: n\in \cn\:\:\text{ unique}.
\end{align} 
In particular, a $\Sigma$-decomposition is unique up to change of orientation as described in Remark \ref{dfxhbghg}.\ref{dfxhbghg2}.
\end{lemma}
\begin{proof}
First observe that the uniqueness statement in \eqref{dfgfg} is clear from the condition \ref{ddefo1} in Definition \ref{deco} (if $\OO'\subseteq \Sigma_m\cap\Sigma_n$ for $m,n\in \cn$, then $m=n$ as $|\OO'|>1$). 
Assume now that the left side of \eqref{dfgfg} holds: 
\vspace{6pt}

\noindent
According to Remark \ref{dfxhbghg}.\ref{dfxhbghg1}, we have
$$
\textstyle S=\bigcup_{n\in \cn}\Sigma_n\qquad\quad\text{hence}\qquad\quad 
Z:=\{n\in \cn\:|\:\OO'\cap \Sigma_n\neq \emptyset \}\neq \emptyset.
$$ 
Let $n\in Z$ be given. 
According to Remark \ref{cnmnmcviufeiureiure}.\ref{cnmnmcviufeiureiure1}, there  exists an open segment $\UU'\subseteq \Sigma_n\cap \OO'$, so that the assumed embedding properties (first implication) yield: 
\begin{align*}
(g\cdot \iota|_{\Sigma_0})(\OO)=\iota(\OO')\qquad\quad&\Longrightarrow\qquad\quad  g\cdot \iota|_{\Sigma_0}\cpsim \iota|_{\Sigma_n}
\\
\qquad&\stackrel{\ref{ddefo2}}{\Longrightarrow}\qquad\quad  g\cdot \iota|_{\Sigma_0}\cpsim  g_n\cdot \iota|_{\Sigma_0}
\qquad\:\: \stackrel{\Sigma_0=\Sigma\: \text{free}}{\Longrightarrow}\qquad\:\: [g]=[g_n]
\end{align*}
\begingroup
\setlength{\leftmargini}{12pt}
\begin{itemize}
\item
Assume that there exists some $Z\ni n\neq \cn_\pm$. Then, 
Remark \ref{dfxhbghg}.\ref{dfxhbghg9} shows $Z=\{n\}$, so that  $\OO'\subseteq \Sigma_n$ is clear from the definition of $Z$. 
\item
In the other case, we have $\emptyset\neq Z\subseteq \{\cn_-,\cn_+\}\cap \ZZ$, and then necessarily
\begin{align*}
	either\qquad\quad Z=\{\cn_-\}\quad\text{hence}\quad \OO'\subseteq \Sigma_{\cn_-}\qquad\quad\text{or}\qquad\quad 
	Z=\{\cn_+\}\quad\text{hence}\quad \OO'\subseteq \Sigma_{\cn_+}.  
\end{align*}
In fact, if $Z=\{\cn_-,\cn_+\}$ holds, then  connectedness of $\OO'$ together with $\Sigma_{\cn_\pm}\cap \OO'\neq \emptyset$, implies $\Sigma_n\subseteq \OO'$ for all $\cn_-<n<\cn_+$ (recall that $\Sigma_\pm$ are boundary segments of $S\cong \ID$), which contradicts $Z= \{\cn_-,\cn_+\}$.
\end{itemize}
\endgroup
\noindent
This proves the implication \eqref{dfgfg}, from which the last statement follows by the same elementary arguments that we have already used to prove \ref{nmdsnmdnmdssdkjskjkjdskjdsds7ds87ds87ds87dszuiu1} and \ref{nmdsnmdnmdssdkjskjkjdskjdsds7ds87ds87ds87dszuiu2} in  Lemma \ref{asdhlkdsajkd}. 
\end{proof}
\noindent
We now finally have to show the existence of a $\Sigma$-decomposition for some given compact $\MFree(S)\ni\Sigma\subset \innt[S]$. The argumentation is similar to that in Sect.\ \ref{fghgfasa1}:
\vspace{6pt}

\noindent
We write $\partial_{\Sigma}=\{z_{-1},z_1\}$, and fix a homeomorphism $\homeo\colon \ID\rightarrow S$ with 
$$
\homeo([c_{-1},c_1])=\Sigma\quad\wedge\quad
\homeo(c_{-1})=z_{-1}\quad\wedge\quad \homeo(c_1)=z_1\qquad    
\text{for (unique)}\qquad \ID\ni c_{-1}<c_1\in \ID.$$ 
We 
let $\CI$ denote the set of all pairs $\mu\equiv (\{a_n\}_{n\in \cnN},\{g_n\}_{n\in \cnK})$ with $\lenhhh{\mu}:=\cn\in \CN$ and $g_0=e$, such that  
the following conditions are fulfilled:
\begingroup
\setlength{\leftmargini}{17pt}
{
\renewcommand{\theenumi}{\small{\bf\alph{enumi})}} 
\renewcommand{\labelenumi}{\theenumi}
\begin{enumerate}
\item
\label{dsdsdskjsdkjdssdnbsdnbdsjdszudszusddzus8d98sd9898ds98dds1}
$a_{-1}=c_{-1}$\quad$\wedge$\quad $a_1=c_1$\qquad\:\:{}as well as\qquad\:\:$a_{m}<a_{n}$\quad\:{}for\quad $\cn\ni m<n\in \cn$.
\vspace{2pt}
\item
\label{dsdsdskjsdkjdssdnbsdnbdsjdszudszusddzus8d98sd9898ds98dds2}
$g_n\cdot \iota(\Sigma_0)=\iota(\Sigma_n)$\: holds for all\: $n\in \cnK$,\: for\: $\Sigma_0:=\homeo([a_{-1},a_1])$\: as well as 
$$
\Sigma_n:=\homeo([a_{n-1},a_{n}])\quad \text{for}\quad \cn_-<n\leq -1\qquad\:\: \text{and}\qquad\:\: \Sigma_n:=\homeo([a_n,a_{n+1}])\quad \text{for}\quad 1\leq n< \cn_+.
$$
Notably, $\Sigma_n\in \MFree(S)$ holds for all $\cn_-<n<\cn_+$ by Corollary \ref{fddsfd} and Lemma \ref{ghfhh}.   
\end{enumerate}}
\endgroup
\noindent
For 
$\CN\ni \hat{\cn}\subseteq \cn$, 
we set $\mu|_{\hat{\cn}}:=(\{a_n\}_{n\in \hcnN},\{g_n\}_{n\in \hcnK})$. 
\vspace{6pt}

\noindent
We let  
 $\CI$ be partially ordered by
 \vspace{-4pt}
\begin{align*}
	\mu\leq \tilde{\mu}\quad\text{for}\quad \mu,\tilde{\mu}\in \CI \qquad\quad\stackrel{{\rm def.}}{\Longleftrightarrow}\qquad\quad \lenhhh{\mu}\leq \lenhhh{\tilde{\mu}}\quad\wedge\quad \mu=\tilde{\mu}|_{\lenhhh{\mu}}.
\end{align*}
Then, $\CI\neq \emptyset$ holds (consider $\cn=\{-1,1\}$), and each chain in $\CI$ admits the obvious upper bound. Zorn's lemma thus provides a maximal element $\CI\ni \mu \equiv (\{a_n\}_{n\in \cnN},\{g_n\}_{n\in \cnK})$. We observe the following:
\begin{claim}
\label{wclm3}
Assume that $\cn_\pmm\neq \pmm\he \infty$ holds, with $z_\pmm:=\homeo(a_{\cn_\pmm})\in \innt[S]$.  
Then, 
there exists some $g_{\cn_\pmm}\in G$  and a boundary segment $\Sigma_\pmm$  
 of $\Sigma_0$, such that
\begin{align}
\label{nmcnmcxnmcxkjdskjdskjskjskjds989898dssddsdsdsds}
 g_{\cn_\pmm}\cdot  \iota(\Sigma_{\pmm})= \iota(\Sigma_{\cn_\pmm})\qquad\text{holds for}\qquad \Sigma_{\cn_\pmm}:=\homeo(\ID\cap [a_{\cn_\pmm},\pmm\:\infty)) 
\end{align}
and\: $\iota|_{\Sigma_{\cn_\pmm}}$\! is an embedding.
\begin{proof} 
We only discuss the case $\cn_+\neq \infty$, because the case  $\cn_-\neq -\infty$ follows analogously.  
Let $[g]$, $\Sigma_b$, $\Sigma_z$ be as in \ref{starlkdslkdslkdslksdlkdslklkdsdsdsds} in Proposition \ref{prop:shifttrans}, applied to $z:=z_+$ and $\Sigma\equiv \Sigma_{\cn_+ -1}$:
\begingroup
\setlength{\leftmargini}{12pt}
\begin{itemize}
\item
Then, $\Sigma_{\cn_+-1}\cap \Sigma_z=\partial_{\Sigma_{\cn_+-1}}\cap \partial_{\Sigma_z}=\{z\}$ holds; and we have 
$\varrho(\Sigma_b)=\Sigma_z$ for the analytic diffeomorphism\footnote{Observe that $\iota|_{\Sigma_b},\iota|_{\Sigma_z}$ are embeddings by compactness of $\Sigma_b,\Sigma_z$. Hence,   $\varrho$ is an analytic diffeomorphism by Corollary \ref{fddsfd}.} $\varrho:= (\iota|_{\Sigma_z})^{-1}\cp (g\cdot \iota|_{\Sigma_b})\colon \Sigma_b\rightarrow \Sigma_z$. We set  $\LL:= \Sigma_b\subseteq \Sigma_{\cn_+-1}$ and $\LL':=\Sigma_z$. 
\item
Remark   \ref{oidsoidsoidsoidsoidsoidsdsdsds454545remk}.\ref{nmdsnmdsnmdsnmh1}  
yields the following:
\vspace{-2pt}
\begingroup
\setlength{\leftmarginii}{12pt}
\begin{itemize}
\item
$\tilde{\Sigma}_+:=\dom[\ovl{\varrho}]\cap \Sigma_{\cn_+ -1}$ is a boundary segment of $\Sigma_{\cn_+ -1}$, 
and $\iota$ is an embedding on $\Sigma_{\cn_+}:=\ovl{\varrho}(\tilde{\Sigma}_+)$ (observe that $\Sigma_{\cn_+ -1}$ is compact).
\vspace{3pt} 
\item
$g\cdot \iota(\tilde{\Sigma}_+)=\iota(\Sigma_{\cn_+})$ holds; whereby 
$\tilde{\Sigma}_{+} \cap \Sigma_{\cn_+}= \partial_{\tilde{\Sigma}_{+}}\cap \partial_{\Sigma_{\cn_+}}
$ is empty or singleton. 
\vspace{4pt}

Together with the first point, the second statement yields the implication:  
\begin{align}
\label{dslkdslkdsdsooidsoidsdsds98ds98dsdsds}
\tilde{\Sigma}_+=\Sigma_{\cn_+-1}\qquad\:\:\Longrightarrow\qquad\:\:  \tilde{\Sigma}_+ \cap \Sigma_{\cn_+}= \partial_{\tilde{\Sigma}_+}\cap \partial_{\Sigma_{\cn_+}}=\{z\}
\end{align}
\vspace{-15pt} 
\item
If $\tilde{\Sigma}_+\subset\Sigma_{\cn_+-1}$ holds, then $\Sigma_{\cn_+}$ is a boundary segment of $S$. 
\end{itemize}
\endgroup
\item
The left side of \eqref{nmcnmcxnmcxkjdskjdskjskjskjds989898dssddsdsdsds} holds for
\begin{align*}
	g_{\cn_+}:= g\cdot g_{\cn_+-1} \qquad\quad\text{and}\qquad\quad \Sigma_+:= ((\iota|_{\Sigma_0})^{-1}\cp (g^{-1}_{\cn_{+} - 1}\cdot \iota|_{\Sigma_{{\cn_+}- 1}}))(\tilde{\Sigma}_+),
\end{align*}
whereby $\Sigma_+$ is a boundary segment of $\Sigma_0$ by the embedding properties of $\iota|_{\Sigma_0},\iota|_{\Sigma_{{\cn_+}- 1}}$.
\end{itemize}
\endgroup
\noindent
Now, the following two cases can occur:
\begingroup
\setlength{\leftmargini}{12pt}
\begin{itemize}
\item
$\tilde{\Sigma}_+\subset\Sigma_{\cn_+-1}$:\:\: Then, the right side of \eqref{nmcnmcxnmcxkjdskjdskjskjskjds989898dssddsdsdsds}  holds, because $\Sigma_{\cn_+}$ is a boundary segment of $S$.     
\item
$\tilde{\Sigma}_+=\Sigma_{\cn_+-1}$:\:\:  Then,  \eqref{dslkdslkdsdsooidsoidsdsds98ds98dsdsds} shows 
$\homeo^{-1}(\Sigma_{\cn_+})=[a_{\cn_{+}}, \hat{a}]$ for some $a_{\cn_{+}}< \hat{a}\in \ID$. We set 
$$
\hcn_-:=\cn_-,\qquad \hcn_+:=\cn_++1,\qquad \hcn:=\{n\in \ZZ\:|\: \hcn_-\leq n\leq \hcn_+\},
$$
and define $\CI\ni\hat{\mu}\equiv (\{\hat{a}_n\}_{n\in \hcnN},\{\hat{g}_n\}_{n\in \hcnK})$ by
\begingroup
\setlength{\leftmarginii}{12pt}
\begin{itemize}
\item
$\hat{g}_n:= g_n$\:\: for all\:\: $n\in \cnK$\:\: as well as\:\: $g_{\hcn_+}:=h_+\cdot g_{\cn_+}$
\item
$\hat{a}_n:=a_n$\:\: for all\:\: $n\in \cnN$\:\: as well as\:\: $\hat{a}_{\hcn_+}:=\hat{a}$.
\end{itemize}
\endgroup
\noindent
By construction, we have $\mu < \hat{\mu}$, which contradicts that $\mu\in \CI$ is maximal. \qedhere
\end{itemize}
\endgroup
\end{proof}
\end{claim}
\begin{lemma}
\label{edffds}
Assume that $S\cong \ID$ holds.  
Then, each compact maximal segment $\MFree(S)\ni\Sigma\subset \innt[S]$ admits a $\Sigma$-decomposition of $S$ that is unique in the sense of Lemma \ref{dfggfgf}. 
\end{lemma}
\begin{proof}
Let $\mu\in \CI$ be maximal: 
\begingroup
\setlength{\leftmargini}{11pt}
\begin{itemize}
\item
If $\cn_\pmm=\scaleobj{0.8}{\pmm}\: \infty$ holds, then we set $\hcn_\pmm:=\cn_\pmm$.
\item 
If $\cn_\pmm\neq \scaleobj{0.8}{\pmm}\:\infty$ with $\homeo(a_{\cn_\pmm})\in \partial_S$ holds, then we set 
\begin{align*}
	\hcn_\pmm:= \cn_\pmm\: \scaleobj{0.8}{\mmp}\: 1\qquad\quad\text{as well as}\qquad\quad \Sigma_\pmm:= \Sigma_0.
\end{align*}
(\he{}Observe $|\cn_\pmm| \geq 2$, as $\homeo(a_{\pm 1})\in \innt[S]$ and $\homeo(a_{\cn_\pmm})\in \partial_S$ holds.\he)
\item
If $\cn_\pmm\neq \scaleobj{0.8}{\pmm}\:\infty$ with $\homeo(a_{\cn_\pmm})\in \innt[S]$ holds, then we 
set 
$$\hcn_\pmm:=\cn_\pmm\quad\:\: \text{and choose}\quad\:\: \Sigma_{\pmm}\quad\:\:  
\text{and}\quad\:\: g_{\cn_\pmm}\quad\:\:\text{as in}\quad\:\:  \eqref{nmcnmcxnmcxkjdskjdskjskjskjds989898dssddsdsdsds}.$$ 
\end{itemize}
\endgroup
\vspace{-6pt}

\noindent
Then,  
$(\{\Sigma_n\}_{n\in \hcn},\{[g_n]\}_{n\in \hcn})$ is a $\Sigma$-decomposition of $S$. 
\end{proof}

\subsection{$z$-Decompositions}
\label{lkjkjlfdjfdkjfkjfdlkfdj}
So far, we have discussed the situation where $S$ admits a compact maximal segment that is contained in the interior of $S$. According to Remark \ref{kjfdjlkfdjkfd}, this is always the case if  $S\cong \UE$ is free with $S\notin \Free(S)$.  
In this subsection, we show that  exactly one further situation can occur if  $S\cong \ID$ is free with $S\notin \Free(S)$: 
\begin{proposition}
\label{dhjfdhjfssasa}
Let $S\cong \ID$ be free with $S\notin \Free(S)$.  
Then, $S$ either admits a unique $z$-decomposition or a compact maximal  segment $\MFree(S)\ni \Sigma\subset \innt[S]$. 
\end{proposition}
\noindent
The notion of  a $z$-decomposition is defined as follows:
\begin{definition}
\label{fdgffgdaaa}
Assume that $S\cong \ID$ holds, and let $z\in \innt[S]$ be given. 
 A $z$-decomposition of $S$ is a class $[e]\neq  [g]\in \GES$ with $g\in G_{z}$, such that the following two conditions are fulfilled:
\begingroup
\setlength{\leftmargini}{17pt}
{
\renewcommand{\theenumi}{\small{\bf\alph{enumi})}} 
\renewcommand{\labelenumi}{\theenumi}
\begin{enumerate}
\item
\label{fdgffgdaaa1}
There exist compact segments $\KK_\pm\subseteq S$ with $g\cdot \iota(\KK_-)=\iota(\KK_+)$ and $\KK_-\cap\KK_+=\{z\}$. 
\item
\label{fdgffgdaaa2}
$\Sigma_\pm\in \Free(S)$ holds for the unique boundary segments $\Sigma_\pm\in \Seg(S)$  of $S$, with
\begin{align}
\label{sdllkdskdsllksdlkdssd9sdds9898dsdsdsdsds}
 S=\Sigma_-\cup \Sigma_+,\qquad \Sigma_-\cap\Sigma_+=\partial_{\Sigma_-}\cap\partial_{\Sigma_+}=\{z\},\qquad
  \KK_\pm\subseteq \Sigma_\pm.
\end{align}
\vspace{-20pt}
\end{enumerate}}
\endgroup
\end{definition}
\begin{convention}
\label{kjfdkjfdj}
Let $(S,\iota)$ be free with $S\notin \Free(S)$.
\vspace{-4pt} 
\begingroup
\setlength{\leftmargini}{12pt}
\begin{itemize}
\item
If we say that $S$ admits a $z$-decomposition, then we implicitly mean that $S\cong \ID$ holds. 
\item
If we say that $S$ admits no $z$-decomposition, then we explicitly include the case $S\cong \UE$. 

Note that Remark \ref{kjfdjlkfdjkfd} and Proposition \ref{dhjfdhjfssasa}  show that $S$ admits some compact maximal  $\MFree(S)\ni \Sigma\subset \innt[S]$ \defff it admits no $z$-decomposition (which is always the case if $S\cong\UE$ holds).
\item
By a decomposition of $S$, we understand 
\vspace{-4pt}
\begingroup
\setlength{\leftmarginii}{11pt}
\begin{itemize}
\item[$\cp$]
a $\Sigma$-decomposition if $S\cong\UE$ holds.
\item[$\cp$]
a $\Sigma$-decomposition or a $z$-decomposition if $S\cong\ID$ holds.
\end{itemize}
\endgroup
\end{itemize}
\endgroup
\end{convention}
\begin{lemma}
\label{dfdffdfddfdf} 
In the situation of Definition \ref{fdgffgdaaa},     
exactly one of the following two cases holds:
\begin{align}
\label{sit2}
(\iota|_{C_+})^{-1}\cp(g\cdot \iota|_{\Sigma_-})\colon\hspace{37.3pt} \Sigma_-&\rightarrow C_+\subseteq \Sigma_+\:\quad\text{is an analytic diffeomorphism}\\
\label{sit1}
(\iota|_{\Sigma_+})^{-1}\cp(g\cdot \iota|_{C_-})\colon \quad\Sigma_-\supset C_-&\rightarrow \Sigma_+\:\quad\hspace{27.2pt}\text{is an analytic diffeomorphism}
\end{align}
for boundary segments $C_\pmm\in \Free(S)$ of $\Sigma_\pmm$ with $z\in \partial_{C_\pmm}$ and $\KK_\pmm\subseteq C_\pmm$.
\end{lemma}
\begin{proof}
We set $\LL:=\KK_-\subseteq \Sigma_-$ and $\LL':=\KK_+\subseteq \Sigma_+$.  
Then, 
$\varrho:=\iota^{-1}\cp (g\cdot \iota|_{\KK_-})\colon \LL\rightarrow \LL'$
 is an analytic diffeomorphism by Corollary \ref{fddsfd}  (since $\KK_\pm$ is compact). We define
\begin{align*}
 C_-:=\dom[\ovl{\varrho}]\cap \Sigma_-\supseteq \KK_-\qquad\text{as well as}\qquad C_+:=\ovl{\varrho}(C_-)\supseteq \ovl{\varrho}(\KK_-)=\KK_+.
\end{align*}
\vspace{-20pt}
\begingroup
\setlength{\leftmargini}{14pt}
{
\renewcommand{\theenumi}{\small{\sf\roman{enumi})}} 
\renewcommand{\labelenumi}{\theenumi}
\begin{enumerate}
\item
\label{dsdsoidsoidsilkdslkdslkdslkdslk1}
Remark \ref{oidsoidsoidsoidsoidsoidsdsdsds454545remk}.\ref{nmdsnmdsnmdsnmh1} shows that $C_-:=\dom[\ovl{\varrho}]\cap \Sigma_-$ is a boundary segment of $\Sigma_-$. Moreover,  $z\in \partial_{C_-}$ is implied by $\partial_{\Sigma_-}\ni z\in \KK_-\subseteq C_-\subseteq \Sigma_-$:
\vspace{2pt}

{\it Proof.} If $z\in \innt[C_-]$ holds, then $C_-\subseteq \Sigma_-$ implies $z\in \innt[\Sigma_-]$, which contradicts $z\in \partial_{\Sigma_-}$.\hspace*{\fill}$\ddagger$
\item
\label{dsdsoidsoidsilkdslkdslkdslkdslk2}
$C_+$ is a boundary segment of $\Sigma_+$; specifically, 
$C_+\subseteq \Sigma_+$ holds with $z\in \partial_{C_+}$:
\vspace{2pt}

{\it Proof.} 
Remark \ref{oidsoidsoidsoidsoidsoidsdsdsds454545remk}.\ref{nmdsnmdsnmdsnmh1c} shows that $\Sigma_-\cap C_+=\partial_{\Sigma_-}\cap \partial_{C_+}$ is empty or singleton;  so that
$$
z\ni \Sigma_- \quad\wedge\quad z\in \KK_+=\ovl{\varrho}(\KK_-)\subseteq\ovl{\varrho}(C_-)=C_+
\qquad\quad\Longrightarrow\qquad\quad \Sigma_-\cap C_+=\partial_{\Sigma_-}\cap \partial_{C_+}=\{z\}.
$$
The claim is now immediate from the right side and the first two points in \eqref{sdllkdskdsllksdlkdssd9sdds9898dsdsdsdsds}.\hspace*{\fill}$\ddagger$
\end{enumerate}}
\endgroup
\noindent
Now, the following two cases can occur:
\begingroup
\setlength{\leftmargini}{12pt}
\begin{itemize}
\item
$C_-=\Sigma_-$:\:\:\quad   
$C_+$ is a boundary segment of $\Sigma_+$ by Point \ref{dsdsoidsoidsilkdslkdslkdslkdslk2}.
\item
$C_-\subset \Sigma_-$:
\vspace{-20pt}
\begingroup
\setlength{\leftmarginii}{71pt}
\begin{itemize}
\item
 $C_-$ is a boundary segment of $\Sigma_-$ by Point \ref{dsdsoidsoidsilkdslkdslkdslkdslk1}.
\item 
$C_+=\Sigma_+$ holds, because:
\vspace{-1pt}
\begingroup
\setlength{\leftmarginiii}{11.5pt}
\begin{itemize}
\item
$\Sigma_+$ is a boundary segment of $S\cong\ID$.
\item
$C_+$ is a boundary segment of $\Sigma_+$ with $z\in \partial_{C_+}\cap \partial_{\Sigma_+}$ by Point \ref{dsdsoidsoidsilkdslkdslkdslkdslk2}. 
\item 
$C_+$ is a boundary segment of $S$ by Remark \ref{oidsoidsoidsoidsoidsoidsdsdsds454545remk}.\ref{nmdsnmdsnmdsnmh1b}.
\qedhere
\end{itemize}
\endgroup
\end{itemize}
\endgroup
\end{itemize}
\endgroup
\end{proof}

\begin{lemma}
\label{gfhhgh}
Assume that $S\cong \ID$ holds,      
and let $z\in \innt[S]$ be given with $z$-decomposition as in Definition \ref{fdgffgdaaa}.  
Then, the following assertions hold:
\begingroup
\setlength{\leftmargini}{15pt}
{
\renewcommand{\theenumi}{{\arabic{enumi})}} 
\renewcommand{\labelenumi}{\theenumi}
\begin{enumerate}
\item
\label{gfhhgh1}
We have $\MFree(S)=\{\Sigma_-,\Sigma_+\}$.
\item
\label{gfhhgh2}
We have the implication
\vspace{-2pt}
\begin{align}
\label{ddhjdkjd}
g'\cdot \iota|_{\Sigma_-}\cpsim\iota\quad\text{for}\quad g'\in G\qquad\quad\Longrightarrow\qquad\quad [g']\in \{[e],[g]\}.\hspace{50pt}
\end{align}
\vspace{-17pt}
\item
\label{fddffddf1}
We have $[g^{-1}]=[g]$. 
\item
\label{fddffddf}
There exists no further $z$-decomposition of $S$: 

If $[g']$ is a $z'$-decomposition of $S$ (\he$[e]\neq [g']\in \GES$ and $z'\in \innt[S]$\he), then we have $z'=z$ and $[g']=[g]$. 
\end{enumerate}}
\endgroup
\end{lemma}
\begin{proof}
\begingroup
\setlength{\leftmargini}{15pt}
{
\renewcommand{\theenumi}{{\arabic{enumi})}} 
\renewcommand{\labelenumi}{\theenumi}
\begin{enumerate}
\item
Since $\Sigma_\pm\in \Free(S)$ holds by assumption, 
it suffices to show that for each $C\in \Free(S)$, we either have $C\subseteq \Sigma_-$ or $C\subseteq \Sigma_+$. 
Let thus $C\in \Free(S)$ be given: 
\vspace{-2pt}
\begingroup
\setlength{\leftmarginii}{11pt}
\begin{itemize}
\item
We have $z\notin \innt[C]$, because   
Point \ref{fdgffgdaaa1} in Definition \ref{fdgffgdaaa} provides the implication
\begin{align}
\label{kjfdkjfdkjf}
	z\in \innt[C]\qquad\quad\stackrel{\ref{fdgffgdaaa1}}{\Longrightarrow}\qquad\quad g\cdot \iota|_C\cpsim \iota|_C \qquad\:\:\stackrel{C\:\in\: \Free(S)}{\Longrightarrow}\qquad\:\:  [g]=[e]
\end{align}
whereby the right contradicts $[g]\neq e$. 
\item
Since $C$ is connected, the previous point together with the first two points in \eqref{sdllkdskdsllksdlkdssd9sdds9898dsdsdsdsds} implies  
that either $C\subseteq \Sigma_-$ or $C\subseteq \Sigma_+$ holds. 
\end{itemize}
\endgroup
\item
The left side of \eqref{ddhjdkjd} together with \eqref{sdllkdskdsllksdlkdssd9sdds9898dsdsdsdsds} implies 
\begin{align}
\label{lsdldlksklkldskldssd98sd9s8d989ds89sdsdssxccxcxcxd}
g'\cdot \iota|_{\Sigma_-}\cpsim\iota|_{\Sigma_-}\qquad \vee\qquad\:\: g'\cdot \iota|_{\Sigma_-}\cpsim\iota|_{\Sigma_+}.
\end{align}
\begingroup
\setlength{\leftmarginii}{12pt}
\begin{itemize}
\item
The left side of \eqref{lsdldlksklkldskldssd98sd9s8d989ds89sdsdssxccxcxcxd} together with $\Sigma_-\in \Free(S)$ implies $[g']=[e]$.
\item
Assume that the right side of \eqref{lsdldlksklkldskldssd98sd9s8d989ds89sdsdssxccxcxcxd}  holds. By Lemma \ref{dfdffdfddfdf}, either \eqref{sit2} or \eqref{sit1} holds, whereby 
\begin{align*}
\eqref{sit2}\qquad\Longrightarrow\quad\quad\: g'^{-1}\cdot \iota|_{\Sigma_+}\cpsim g^{-1}\cdot \iota|_{\Sigma_+}	   \qquad\: &\stackrel{\eqref{qwepokfdjkhfd}}{\Longrightarrow}\qquad\:\: [g']=[g]\\[-3pt]
\eqref{sit1}\qquad\Longrightarrow\quad\qquad\: g'\cdot \iota|_{\Sigma_-}\cpsim \phantom{^{-1}}g\cdot \iota|_{\Sigma_-}\qquad\hspace{1.6pt} &\Longrightarrow\qquad\:\: [g']=[g] 
\end{align*}
\end{itemize}
\endgroup 
as $\Sigma_\pm\in \Free(S)$ holds. 
\item 
It suffices to show that 
$g^{-1}\cdot \iota|_{\Sigma_-}\cpsim \iota$ holds, because  then \eqref{ddhjdkjd} yields 
$$
[g^{-1}]\in\{[e],[g]\}\qquad\quad\stackrel{\eqref{qwepokfdjkhfd}}{\Longrightarrow}\qquad\quad [g^{-1}]=[g].
$$ 
The argumentation is analogous to that in Lemma \ref{dspodsopodssdklkdss98ds98ds98ds}:

We fix a chart $(U,\psi)$ around $z\in \innt[S]$ with $U\subseteq \KK_-\cup \KK_+$. Lemma \ref{dfdffdfddfdf} and Lemma \ref{lemma:BasicAnalytt1}.\ref{lemma:BasicAnalytt123} provide  compact connected neighbourhoods $\KK,\KK'\subseteq U\subseteq \KK_-\cup \KK_+$ of $z$ with $g\cdot \iota(\KK)=\iota(\KK')$. Then, we have
\begin{align*}
g\in G_z,\qquad g\cdot \iota(\KK_-)=\iota(\KK_+),\qquad
	\KK\cap \KK_+=\{z\}\:\dot\cup\: (\KK\setm \KK_-),\qquad\KK'\cap \KK_-=\{z\}\:\dot\cup\: (\KK'\setm \KK_+).
\end{align*}
Together with injectivity of $\iota$, $g\cdot \iota$, 
this yields
\begin{align*}
	\hspace{30pt} g\cdot \iota(\underbrace{\KK\cap \KK_+}_{\displaystyle=:\tilde{\KK}_+})&= \{g\cdot \iota(z)\}\:\dot\cup \:(g\cdot \iota(\KK\setm \KK_-))\\[-19pt]
	&=\{\iota(z)\}\:\dot\cup \: ((g\cdot \iota(\KK))\setm (g\cdot \iota(\KK_-)))\\
	&= \{\iota(z)\}\:\dot\cup \: (\iota(\KK')\setm \iota(\KK_+))\\
	&=\iota(\underbrace{\KK'\cap\KK_-}_{\displaystyle =: \tilde{\KK}_-})
\end{align*}
\vspace{-13pt}

\noindent
with $\Sigma_\pm\supseteq \tilde{\KK}_\pm \in \Seg(S)$ compact, hence $g^{-1}\cdot \iota|_{\Sigma_-}\cpsim \iota$.
\item
Let $\Sigma'_\pm\in \MFree(S)$ denote the boundary segments that correspond to the $z'$-decomposition $[g']$. Then, Part \ref{gfhhgh1} shows that either $\Sigma'_\pm=\Sigma_\pm$ or $\Sigma'_\pm=\Sigma_\mp$ holds, hence (in both cases)
\vspace{-4pt}
$$
\{z\}\stackrel{\eqref{sdllkdskdsllksdlkdssd9sdds9898dsdsdsdsds}}{=}\Sigma_+\cap\Sigma_-=\Sigma'_+\cap\Sigma'_-\stackrel{\eqref{sdllkdskdsllksdlkdssd9sdds9898dsdsdsdsds}}{=}\{z'\}
\qquad\quad\Longrightarrow\qquad\quad z=z'.
$$
\vspace{-22pt}
\begingroup
\setlength{\leftmarginii}{11pt}
\begin{itemize}
\item
Assume $\Sigma'_\pm=\Sigma_\pm$:\: 
The conditions  \ref{fdgffgdaaa1} and \ref{fdgffgdaaa2} (for the $z'$-decomposition $[g']$) yield
\hspace*{\fill}($[g']\neq [e]$)
\vspace{-4pt}
	 \begin{align*}
g'\cdot \iota|_{\Sigma_-'}\cpsim\iota|_{\Sigma_+}
	\quad\:\:\Longrightarrow\quad\:\:
g'\cdot \iota|_{\Sigma_-}\cpsim\iota
	\quad\:\:\stackrel{\eqref{ddhjdkjd}\text{ for }[g]}{\Longrightarrow}\quad\:\: [g']\in \{[e],[g]\}
	\quad\:\:\Longrightarrow\qquad [g']=[g].
\end{align*}  
\item
Assume $\Sigma'_\pm=\Sigma_\mp$:\: The conditions  \ref{fdgffgdaaa1} and \ref{fdgffgdaaa2} (for the $z$-decomposition $[g]$) yield\hspace*{\fill}($[g]\neq [e]$)
\vspace{-4pt}
	 \begin{align*}
g\cdot \iota|_{\Sigma_-}\cpsim\iota|_{\Sigma_+}
	\quad\:\Longrightarrow\quad\:
g^{-1}\cdot \iota|_{\Sigma_-'}\cpsim\iota
	\quad\stackrel{\eqref{ddhjdkjd}\text{ for }[g']}{\Longrightarrow}\quad [g]\stackrel{\text{Part }\ref{fddffddf1}}{=}[g^{-1}]\in \{[e],[g']\}
	\quad\:\Longrightarrow\quad\: [g]=[g'].
\end{align*}  
\end{itemize}
\endgroup
\noindent
This proves the claim.\qedhere
\end{enumerate}}
\endgroup
\end{proof}
\begin{corollary}
\label{fggfd}
Assume that $S\cong \ID$ holds.
\begingroup
\setlength{\leftmargini}{15pt}
\begin{enumerate}
\item
\label{fggfd1}
If $S$ admits a $z$-decomposition for some $z\in \innt[S]$, then $S$ admits no further decomposition (\he in particular, there exists no compact $\MFree(S)\ni \Sigma\subset \innt[S]$\he). 
\item
\label{fggfd2}
If $S$ admits some compact $\MFree(S)\ni \Sigma\subset \innt[S]$ (\he hence, a $\Sigma$-decomposition by Lemma \ref{edffds}\he), then $S$ cannot admit a $z$-decomposition for some $z\in \innt[S]$.
\end{enumerate}
\endgroup
\end{corollary}
\begin{proof}
\begingroup
\setlength{\leftmargini}{15pt}
\begin{enumerate}
\item
Assume that $S$ admits a $z$-decomposition with $z\in \innt[S]$. By Lemma \ref{gfhhgh}.\ref{gfhhgh1}, $S$ cannot admit some  compact $\MFree(S)\ni \Sigma\subset \innt[S]$ (hence, no $\Sigma$-decomposition). The rest is clear from Lemma \ref{gfhhgh}.\ref{fddffddf}. 
\item
Clear from Part \ref{fggfd1}.\qedhere
\end{enumerate}
\endgroup
\end{proof}
\noindent
We are ready for the proof of Proposition \ref{dhjfdhjfssasa}.
\begin{proof}[Proof of Proposition \ref{dhjfdhjfssasa}] 
By Corollary \ref{fggfd}, it remains to show that $S$ admits a $z$-decomposition for some $z\in \innt[S]$ if $S$ admits no  compact maximal segment properly contained in $\innt[S]$. 
Assume thus that $S$ admits no compact $\MFree(S)\ni \Sigma\subset \innt[S]$, and let $\homeo\colon \ID\rightarrow S$ be a fixed homeomorphism: 
\begingroup
\setlength{\leftmargini}{17pt}
{
\renewcommand{\theenumi}{{\Alph{enumi})}} 
\renewcommand{\labelenumi}{\theenumi}
\begin{enumerate}
\item
\label{step0}
Since $S\notin \Free(S)$ is free, Lemma \ref{fhdhh} provides some $\MFree(S)\ni \Sigma_-\subset S$.
\item 
\label{step1} 
$\Sigma_-$ is a (proper) boundary segment of $S$.
\vspace{1pt}

{\it Proof.} 
 If $\Sigma_-$ is not a boundary segment of $S$, then there exist $\ID\ni a<b\in \ID$ with $D_-:=\homeo^{-1}(\Sigma_-)\subseteq [a,b]$. Lemma \ref{kjkjdskjkjdsaassaa} shows that $\Sigma_-$ is closed in $S$; hence, $D_-$ is closed in $[a,b]$, thus compact. This shows that $\Sigma_-\in \MFree(S)$ is compact, which contradicts the assumptions. 
\hspace*{\fill}$\ddagger$
\item
\label{step2}
Since $D_-:=\homeo^{-1}(\Sigma_-)$ is closed in $\ID$ ($\Sigma_-$ is closed in $S$  by  
Lemma \ref{kjkjdskjkjdsaassaa}), Point \ref{step1} shows that
$$
\text{either}\qquad D_- = \ID\cap (-\infty,\tau]\qquad\text{or}\qquad D_- = \ID\cap [\tau,\infty)\qquad\text{holds  for some}\qquad \tau\in \innnt[\ID]. 
$$
Hence, modifying $\homeo$ if necessary, we can assume that $D_-=\ID\cap (-\infty,0]$ holds, and set 
$$
z:= \homeo(0)\in\Sigma_- \cap\innt[S]\qquad\text{as well as}\qquad \Seg(S)\ni \Sigma_+:=\homeo(D_+)\qquad\text{for}\qquad D_+:=\ID\cap [0,\infty).
$$
Then, $\Sigma_+$ is a proper boundary segment of $S$, such that the first two identities in  \eqref{sdllkdskdsllksdlkdssd9sdds9898dsdsdsdsds} hold.
\item 
\label{step3}
We have $\Sigma_+\in \Free(S)$.
\vspace{1pt}

{\it Proof.} 
Lemma \ref{dsdssddsds} applied to $\Sigma_-$ and $z$, yields $0<\tau\in \ID$ with 
$$
[0,\tau]\subseteq D_+\qquad\quad\text{and}\qquad\quad\Free(S)\ni\Sigma_z:= \homeo([0,\tau])\subseteq \Sigma_+.
$$ 
Then, Lemma \ref{fhdhh} provides some $\Sigma'\in \MFree(S)$ with $\Sigma_z\subseteq \Sigma'$. The same arguments as in \ref{step1} show that $\Sigma'$ is 
a boundary segment of $S$, so that 
\begin{align}
\label{sdoiidsoidskldskldslksdlksdlkdslksdiudsiuds8798ds89dssd}
[0,\tau]\subseteq D':=\homeo^{-1}(\Sigma')\qquad\quad\Longrightarrow\qquad 
\quad D_+\subseteq D'\qquad\vee\qquad D_-\subseteq D'.
\end{align} 
\begingroup
\setlength{\leftmarginii}{11pt}
\begin{itemize}
\item
If $D_+\subseteq D'$ holds, then we have $\Sigma_+\subseteq \Sigma'\in \Free(S)$, hence $\Sigma_+\in \Free(S)$ by Lemma \ref{jdkjfdkjfd}.
\item
If $D_-\subseteq D'$ holds, then the left side of \eqref{sdoiidsoidskldskldslksdlksdlkdslksdiudsiuds8798ds89dssd}  yields
\begin{align*}
D_-\subset D_-\cup [0,\tau]\subseteq D'\qquad\quad&\Longrightarrow\qquad\quad \Sigma_-\subset \Sigma'_-:=\Sigma_-\cup \Sigma_z\subseteq \Sigma'\in \Free(S)\\
&\hspace{-9.5pt}\stackrel{{\rm Lemma}\: \ref{jdkjfdkjfd}}{\Longrightarrow}\qquad\hspace{0.3pt} \Sigma_-\subset \Sigma_-'\in \Free(S),
\end{align*} 
\vspace{-23pt}

\noindent
which contradicts maximality of $\Sigma_-$.
\hspace*{\fill}$\ddagger$
\end{itemize}
\endgroup
\item
\label{step4}
To prove the claim, it suffices to show that there exists some compact $\Seg(S)\ni S'\subset\innt[S]$ with $z\in \innt[S']$ such that $\Sigma:=  S'\cap \Sigma_-\in \MFree(S')$ holds.\footnote{Hence, $\Sigma$ is maximal w.r.t.\  $(S',\iota|_{S'})$, thus compact as closed in $S'$ by Lemma \ref{kjkjdskjkjdsaassaa}.}
\vspace{1pt}

{\it Proof.}
We have 
 $S'=\homeo([a,b])$ for certain $\inf(D)<a<0<b<\sup(D)$,   because $S'$ is not a boundary segment of $S$, hence 
 $$
 	\Sigma=\homeo([a,0]) \qquad\quad\text{and}\qquad\quad \partial_{\Sigma}\cap \innt[S']=\{z\}.
 $$
Now, $\Sigma\subset S'$ is either positive or negative w.r.t.\ $(S',\iota|_{S'})$ (see Definition \ref{dslkdslkdskldefkjdsodsoidsiidsodsdsdssd} and Remark \ref{dsdsdsdskjdsjkdskjdskjdsds09w0909ew09we09ewewewewewewewcx}):
\begingroup
\setlength{\leftmarginii}{11pt}
\begin{itemize}
\item
If $\Sigma$ is positive w.r.t.\ $(S',\iota_{S'})$, then Lemma \ref{qasggpapd} shows $\Sigma\in \MFree(S)$. This, however,  contradicts $\Sigma\subset \Sigma_-$, as $\Sigma_-\in \MFree(S)$ holds by Lemma \ref{gfhhgh}.\ref{gfhhgh1}.
\item
Hence, $\Sigma$ is negative w.r.t.\ $(S',\iota_{S'})$. According to Definition \ref{dslkdslkdskldefkjdsodsoidsiidsodsdsdssd}, there exists 
$$
\GES\stackrel{\eqref{jdkjfdkjfddf}}{=}\GESS\ni[g]\neq [e]\qquad \text{with} \qquad g\in G_z
$$ 
as well as compact segments $\KK_\pm\subseteq \Sigma_\pm\subseteq S$ with $\KK_-\cap\KK_+=\{z\}$ and $g\cdot \iota(\KK_-)=\iota(\KK_+)$.\footnote{In the context of Definition \ref{dslkdslkdskldefkjdsodsoidsiidsodsdsdssd}, set $\KK_-:=\Sigma_b$ and $\KK_+:=\Sigma_z$.}  
\hspace*{\fill}$\ddagger$ 
\end{itemize}
\endgroup
\end{enumerate}}
\endgroup
\noindent
Assume now that a compact segment $S'$ as described in  \ref{step4} does not exist: 
\begingroup
\setlength{\leftmargini}{11pt}
\begin{itemize}
\item
We fix  $t\in \ID$ and a sequence $\{a_n\}_{n\in \NN}\subseteq \ID$, with
\begin{align*}
	\textstyle\lim_n a_n=\inf(\ID) \qquad\quad\text{as well as}\qquad\quad  \inf(\ID)<a_{n+1}<a_n<0<t<\sup(\ID)\qquad\forall\: n\in \NN.
\end{align*} 
\item
For each $n\in \NN$, we define 
$$\Seg(S)\ni S'_n:=\homeo([a_n,t])\subset \innt[S]\qquad\quad\text{as well as}\qquad\quad \Sigma_n:=S_n'\cap \Sigma_-=\homeo([a_n,0]).
$$
Then, $\Sigma_n\notin \MFree(S')$ holds by assumption (as $z\in \innt[S_n']$), and $\Sigma_n\in \Free(S')$ holds by Lemma \ref{jdkjfdkjfd} and \eqref{jdkjfdkjfddf}. 
\vspace{1pt}

\noindent
{\it Proof of the second Claim.}
Since $\Sigma_n\subseteq \Sigma_-\in \Free(S)$ holds, Lemma \ref{jdkjfdkjfd} yields $\Sigma_n\in \Free(S)$. Then, $\Sigma_n\in \Free(S')$ holds by \eqref{jdkjfdkjfddf} and $\Sigma_n\subseteq S_n'$. 
\hspace*{\fill}$\ddagger$

By Lemma \ref{fhdhh} and Lemma \ref{kjkjdskjkjdsaassaa}, there thus exists some $0< t_n\leq t$ such that $\homeo([a_n,t_n])\in \MFree(S'_n)$ holds.
\end{itemize}
\endgroup
\noindent
The following two situations can occur:
\begingroup
\setlength{\leftmargini}{11pt}
\begin{itemize}
\item
Assume that $t_n=t$ holds for all $n\in \NN$. Then, the same arguments that we have used in the  proof of Lemma \ref{fhdhh} show that \hspace*{\fill}($S_n'\in \Free(S)$ for all $n\in \NN$)
\vspace{-1pt}
\begin{align*}
	\textstyle\Free(S)\stackrel{\eqref{jdkjfdkjfddf}}{\supseteq}\Free(S')\ni \bigcup_{n\in \NN} S_n'=\bigcup_{n\in \NN} \homeo([a_n,t])=\homeo(\ID\cap (-\infty,t])\supset \homeo(\ID\cap (-\infty,0]) =\Sigma_-
\end{align*}
holds; which contradicts  
$\Sigma_-\in \MFree(S)$. 
\item
Assume that $t_n<t$ holds for some $n\in \NN$, hence $S_n'\supset \Sigma:=\homeo([a_n,t_n])\in \MFree(S_n')$.
\vspace{-2pt}
\begingroup
\setlength{\leftmarginii}{11pt}
\begin{itemize} 
\item
$\Sigma$ is negative w.r.t.\ $(S_n',\iota|_{S_n'})$.
\vspace{1pt}

{\it Proof.}
If $\Sigma$ is not negative, then $\Sigma$ is positive (Remark \ref{dsdsdsdskjdsjkdskjdskjdsds09w0909ew09we09ewewewewewewewcx}). 
Hence,  
Lemma \ref{qasggpapd} yields $\Sigma\in \MFree(S)$, which contradicts that $S$ admits no compact maximal segment that is properly contained in $\innt[S]$. 
\hspace*{\fill}$\ddagger$
\item
Since $\Sigma$ is negative with boundary point $z:=\homeo(t_n)\in \innt[S']$ (observe $t_n<t$ by assumption), there exist  
 $\GES=\GESS\ni [g]\neq [e]$ and $\Sigma_b,\Sigma_z\in \Seg(S')$ compact as in  \ref{starlkdslkdslkdslksdlkdslklkdsdsdsds} (in Proposition \ref{prop:shifttrans} applied to $(S',\iota|_{S'}),\Sigma,z$) with $b=z$ and $g\in G_z$. Then, $\Sigma_b=\homeo([x,t_n])$ and $\Sigma_z=\homeo([t_n,y])$ holds for certain $a_n\leq x<t_n<y\leq t$, so that
\begin{align*}
	g\cdot\iota(\Sigma_b)=\iota(\Sigma_z)\qquad\Longrightarrow\qquad g\cdot \iota|_{\homeo([0,t_n])}\cpsim \iota|_{\homeo([t_n,t])}\qquad\Longrightarrow\qquad 
	g\cdot \iota|_{\Sigma_+}\cpsim \iota|_{\Sigma_+}.
\end{align*}
Since $\Sigma_+\in \Free(S)$ holds, the right side implies $[g]=[e]$;     
which contradicts the choices.
\qedhere
\end{itemize}
\endgroup
\end{itemize}
\endgroup
\end{proof}

\subsection{Decomposition Types -- The Classification}
\label{jhdfdfdfkfjdjkdf}
\begin{definition}
\label{dpodspopodslkewlkewewnmnmewnmewwe98ew98ew98ew}
Let $(S,\iota)$ be free with $S\notin \Free(S)$, and assume that $S$ admits no $z$-decomposition. In this case, we say that $S$ is  positive\slash negative \defff each compact $\Sigma\in \MFree(S)$ is 
positive\slash negative.
\end{definition}
\noindent
In this section, we prove the following theorem:\hspace*{\fill}(recall Convention \ref{kjfdkjfdj})
\begin{theorem}
\label{dfsofgofg}
	Assume that $\wm$ is non-contractive; and let $(S,\iota)$ be free with $S\notin\Free(S)$. Then, $S$ either admits a unique $z$-decomposition or some compact $\MFree(S)\ni \Sigma\subset \innt[S]$. Moreover, the following assertions hold:
\begingroup
\setlength{\leftmargini}{15pt}
{
\renewcommand{\theenumi}{{\sf\arabic{enumi})}} 
\renewcommand{\labelenumi}{\theenumi}
\begin{enumerate}
\item
\label{dfsofgofg1}
	If $S\cong \UE$ holds, then 
	each $\Sigma\in \MFree(S)\neq \emptyset$ is compact with $\Sigma\subset \innt[S]$ (no $z$-decomposition possible). 
\item
\label{dfsofgofg2}
If $S$ admits some compact $\MFree(S)\ni \Sigma\subset \innt[S]$,  
then  $S$ is either positive or negative. Moreover, $S$  admits a $\Sigma'$-decomposition for each such compact $\MFree(S)\ni\Sigma'\subset \innt[S]$ that is unique up to change of orientation as described in Lemma \ref{asdhlkdsajkd} ($S\cong\UE$) and Lemma \ref{dfggfgf} ($S\cong\ID$).
\item
\label{dfsofgofg3}   
If $S$ is negative, then $\MFree(S)$ consist of the segments that occur in a fixed $\Sigma$-decomposition of $S$ (for $\MFree(S)\ni \Sigma \subset \innt[S]$ compact). In particular,   up to change of orientation,  the possible decompositions of $S$ are characterized by Remark \ref{sdffdsfsd}.\ref{sdffdsfsd2} ($S\cong\UE$) and Remark \ref{dfxhbghg}.\ref{dfxhbghg3} ($S\cong \ID$).  
\end{enumerate}}
\endgroup
\end{theorem} 
\noindent
In Sect.\ \ref{asasaghhgfg}, we will investigate the positive and the negative case in more detail. In particular, we will   provide explicit formulas for the classes $[g_n]$ that occur in a  given $\Sigma$-decomposition of $S$. Moreover, we will show that in the positive case, the elements in $\MFree(S)$ are continuously distributed in $S$.
\vspace{6pt}

\noindent
Note that the first statement in Theorem \ref{dfsofgofg} as well as Part \ref{dfsofgofg1} are covered by Proposition \ref{dhjfdhjfssasa} and Remark \ref{kjfdjlkfdjkfd}. The parts  \ref{dfsofgofg2} and \ref{dfsofgofg3} are now established step by step:
\begin{lemma}
\label{dsoioidslkdskldskldskldsds98ds09ds9ds9898ds}
Let $\MFree(S)\ni \Sigma,\Sigma'\subset S$ be compact with $z\in \partial_{\Sigma}\cap \partial_{\Sigma'}=\Sigma\cap\Sigma'$. Then, 
$$
\Sigma\quad\text{positive/negative}\qquad\Longleftrightarrow\qquad \Sigma'\quad\text{positive/negative}.
$$ 
\end{lemma}
\begin{proof} 
Since a compact maximal segment is either positive or negative (Remark \ref{dsdsdsdskjdsjkdskjdskjdsds09w0909ew09we09ewewewewewewewcx}), it suffices to show  
that $\Sigma$ is negative\:\:\defff\:$\Sigma'$ is negative. We only prove that $\Sigma'$ is negative if $\Sigma$ is negative, because the converse implication follows analogously. Assume thus that $\Sigma$ is negative: 
\begingroup
\setlength{\leftmargini}{10pt}
\begin{itemize}
\item
Combining Remark \ref{cnmnmcviufeiureiure}.\ref{cnmnmcviufeiureiure6cxcxcxcx} (for $x\equiv z$) with the assumption $z\in \partial_{\Sigma}\cap \partial_{\Sigma'}=\Sigma\cap\Sigma'$, we see that there exists a chart $(U,\psi)$ around $z$ with $\psi(z)=0$ and $\psi(U)=(i',i)$ such that
\begin{align*}
\begin{split}
&\psi(\Sigma\cap U)=(i',0]\qquad\quad\text{and}\qquad\quad \psi(\Sigma'\cap U)=[0,i)\\
&\hspace{103pt}\text{hence}\\
&\hspace{54pt}\partial_{\Sigma}\cap\innt[S]\ni z\in \partial_{\Sigma'}\cap\innt[S].
\end{split}
\end{align*}
\item
We choose   
$[g], \Sigma_b, \Sigma_z$ as in  \ref{starlkdslkdslkdslksdlkdslklkdsdsdsds} in Proposition \ref{prop:shifttrans}. 
Since $\Sigma$ is negative, we have $b=z$ and $g,g^{-1}\in G_z$.  
Shrinking $\Sigma_b,\Sigma_z$ around $b=z$ if necessary, we can achieve that $\Sigma_b\subseteq \Sigma\cap U$ and $\Sigma_z\subseteq \Sigma'\cap U$ holds, hence
\begin{align*}
\psi(\Sigma_b)=[j',0]\qquad\text{and}\qquad \psi(\Sigma_z)=[0,j] \qquad\text{with}\qquad i'<j'<0<j<i.  
\end{align*} 
\item
We set $g':=g^{-1}$ as well as $\Sigma'_{z}:=\Sigma_b$ and $\Sigma'_b:=\Sigma_z$. Then,  \ref{starlkdslkdslkdslksdlkdslklkdsdsdsds} together with the previous two points   
yields
\begin{align*}
g'\cdot \iota(\Sigma'_b)=\iota(\Sigma'_z),\quad\:\:\: g'\cdot \iota(b)=\iota(z),\quad\:\:\: \Sigma'_b\subseteq \Sigma'\quad\text{with}\quad  b\in \partial_{\Sigma'}\cap \partial_{\Sigma'_b},\quad\:\:\: \{z\}=\partial_{\Sigma'}\cap \partial_{\Sigma'_z}=\Sigma'\cap \Sigma'_z.  
\end{align*}
\end{itemize}
\endgroup
\noindent
Consequently,  \ref{starlkdslkdslkdslksdlkdslklkdsdsdsds} holds for $\Sigma', [g'],\Sigma'_b,\Sigma'_z$ with $g'\in G_z$, so that $\Sigma'$ is negative. \qedhere
\end{proof}

\begin{corollary}
\label{kjdsjkdsjkds}
Let $\MFree(S)\ni \Sigma\subset \innt[S]$ be compact. If $\Sigma$ is positive\slash negative, then each compact maximal segment that occurs in a $\Sigma$-decomposition of $S$ is positive\slash negative. 
\end{corollary}
\begin{proof}
The claim just follows inductively from Lemma \ref{dsoioidslkdskldskldskldsds98ds09ds9ds9898ds} as well as 
\begingroup
\setlength{\leftmargini}{10pt}
\begin{itemize}
\item
Remark\hspace{5pt} \ref{sdffdsfsd}.\ref{sdffdsfsd199} if $S\cong\UE$ holds.
\item
Remark \ref{dfxhbghg}.\ref{dfxhbghgFree} and Remark \ref{dfxhbghg}.\ref{dfxhbghg998ds98ds} if $S\cong \ID$ holds.\qedhere
\end{itemize}
\endgroup
\end{proof}
\begin{lemma}
\label{jkdsklsdlkdsjkl}
Let $(S,\iota)$ be free with $S\notin\Free(S)$, and   
assume that $S$ admits no $z$-decomposition for $z\in \innt[S]$. If $\MFree(S)\ni \Sigma\subset S$ is negative, then there exists some negative $\MFree(S)\ni\Sigma'\subset\innt[S]$. 
\end{lemma}
\begin{proof}
The claim is clear if $\Sigma\subset \innt[S]$ holds, e.g., if $S\cong \UE$ (by Remark \ref{kjfdjlkfdjkfd}). Assume thus that $S\cong\ID$ holds, and that $\Sigma$ is a boundary segment of $S$: 
\begingroup
\setlength{\leftmargini}{11pt}
\begin{itemize}
\item
We write $\partial_\Sigma=\{w,z\}$ with $w\in \partial_S$ and  $z\in \innt[S]$, and fix a homeomorphism $\homeo\colon \ID\rightarrow S$ with 
$$
\homeo^{-1}(w)=:i'=\inf(\ID)\in \ID,\qquad \homeo^{-1}(z)=0,\qquad \homeo^{-1}(\Sigma)=[i',0].
$$
\vspace{-17pt}
\item
Let $[g],\Sigma_b,\Sigma_z$ be as in \ref{starlkdslkdslkdslksdlkdslklkdsdsdsds} in Proposition \ref{prop:shifttrans}; in particular,
\begin{align*}
 \homeo^{-1}(\Sigma_z)=[0,i]\quad\text{for some}\quad 0<i\in \ID,\quad\text{with}\quad b=z\quad\text{and}\quad g\in G_z\quad\text{as}\quad\Sigma\quad\text{is negative.}
\end{align*}
\vspace{-17pt}
\item 
Lemma \ref{fhdhh} provides some $\Sigma_z\subseteq \Sigma'\in \MFree(S)$, and we set $D:=\homeo^{-1}(\Sigma')$: 
\begingroup
\setlength{\leftmarginii}{11pt}
\begin{itemize}
\item
We have $0=\inf(D)\in D$:
\vspace{3pt}

{\it Proof.} We have $\inf(D)\leq 0$ as $z\in \Sigma_z\subseteq \Sigma'$ holds (hence $0\in D$). Then, $\inf(D)<0$ together with  
$g\in G_z$ implies $g\cdot \iota|_{\Sigma'}\cpsim \iota|_{\Sigma'}$,  which contradicts $[g]\neq [e]$ and $\Sigma'\in \Free(S)$.   
\hspace*{\fill}$\ddagger$
\vspace{2pt}
\item
We have $D=[0,a]$, for $a:=\sup(D)\in \innnt[\ID]$. 
\vspace{3pt}

{\it Proof.} $\Sigma'$ is closed in $S$ by Lemma \ref{kjkjdskjkjdsaassaa}; hence, $D$ is closed in $\ID$.  
It thus suffices to show $\sup(D)<\sup(\ID)$.  
Assume thus that $\sup(D)\geq \sup(\ID)$ holds; hence, $\sup(D)=\sup(\ID)$ by $D\subseteq \ID$. Then, 
since $D$ is closed in $\ID$,  we have $D=\ID\cap [0,\infty)$; hence, $[g]$ is a $z$-decomposition of $S$ (in Definition  \ref{fdgffgdaaa}, choose $\Sigma_-\equiv \Sigma$, $\Sigma_+:=\Sigma'$, $\KK_-:=\Sigma_b$, $\KK_+:=\Sigma_z$), which contradicts the assumptions.
\hspace*{\fill}$\ddagger$
\end{itemize}
\endgroup
\end{itemize}
\endgroup
\noindent
The first point together with the third point shows 
$$
\{z\}= \partial_{\Sigma}\cap \partial_{\Sigma'}=\Sigma\cap\Sigma'\qquad\text{with}\qquad \MFree(S)\ni \Sigma'\subset \innt[S]\qquad\text{hence}\qquad \Sigma'\subset S,
$$ 
so that Lemma \ref{dsoioidslkdskldskldskldsds98ds09ds9ds9898ds} shows that $\Sigma'$ is negative.  
\qedhere 
\end{proof}

\begin{lemma}
\label{fhjfdsfdd}
Let $\MFree(S)\ni\Sigma\subset \innt[S]$ be negative.  Then, the following assertions hold:
\begingroup
\setlength{\leftmargini}{15pt}
\begin{enumerate}
\item
\label{fhjfdsfdd1}
 $\MFree(S)$  consists of the segments that occur in a fixed $\Sigma$-decomposition of $S$.
\item
\label{fhjfdsfdd2}
Each compact $\Sigma'\in \MFree(S)$ is negative.
\item
\label{fhjfdsfdd3}
Up to change of orientation as described in 
\begingroup
\setlength{\leftmarginii}{12pt}
\begin{itemize}
\item[$\bullet$]
Lemma \ref{asdhlkdsajkd}\hspace{4pt} for\: $S\cong \UE$,
\item[$\bullet$]
Lemma \ref{dfggfgf}\hspace{4pt} for\: $S\cong \ID$,
\end{itemize}
\endgroup
\noindent
the only further decompositions of $S$ are provided by
\begingroup
\setlength{\leftmarginii}{12pt}
\begin{itemize}
\item[$\bullet$]
Remark \hspace{5.3pt}\ref{sdffdsfsd}.\ref{sdffdsfsd2}\hspace{3.8pt} for\: $S\cong \UE$
\item[$\bullet$]
Remark \ref{dfxhbghg}.\ref{dfxhbghg3}\hspace{4pt} for\: $S\cong \ID$.
\end{itemize}
\endgroup
\end{enumerate}
\endgroup 
\end{lemma}
\begin{proof}
Proposition \ref{ghdgddgf}.\ref{ghdgddgf2} and Lemma \ref{edffds} show that $S$ admits a $\Sigma$-decomposition that we denote by $\SD$.   
\begingroup
\setlength{\leftmargini}{11pt}
\begin{itemize}
\item
By Corollary \ref{kjdsjkdsjkds}, each compact maximal segment that occurs in $\SD$ is negative; hence,  
\vspace{-3pt}
\begingroup
\setlength{\leftmarginii}{11pt}
\begin{itemize}
\item
Part \ref{fhjfdsfdd2} is clear from Part \ref{fhjfdsfdd1}.
\item
If $C\in \Free(S)$ is given, then the same arguments that we have used in the beginning of the proof of Lemma \ref{gfhhgh}.\ref{gfhhgh1} (see \eqref{kjfdkjfdkjf}) show that $\innt[C]$ cannot contain any boundary point of any segment occurring in $\SD$; which implies that $C$ is completely contained in one of these segments. From this, the claim is clear, because each segment that occurs in $\SD$ is free (see Definition \ref{gfgfgf} for $S\cong \UE$ as well as Remark \ref{dfxhbghg}.\ref{dfxhbghgFree} for $S\cong \ID$). 
\end{itemize}
\endgroup
\item
By Proposition \ref{dhjfdhjfssasa}, $S$ admits no $z$-decomposition. Hence, Part \ref{fhjfdsfdd3} is clear from  Part  \ref{fhjfdsfdd1}.\qedhere
\end{itemize}
\endgroup
\end{proof}

\begin{lemma}
\label{kjdfdfjkdfjkldfjkldfjkldfjkldfjklfk}
If $S$ admits some compact $\MFree(S)\ni \Sigma\subset\innt[S]$, then   
$S$ is either positive or negative.
\end{lemma}
\begin{proof}
Lemma \ref{asklddd} shows that $\Sigma$ is either positive or negative:
\begingroup
\setlength{\leftmargini}{12pt}
\begin{itemize}
\item
	Assume that $\Sigma$ is negative. Then, Lemma \ref{fhjfdsfdd}.\ref{fhjfdsfdd2} shows that $S$ is negative.
\item
	Assume that $\Sigma$ is positive. We show that $S$ admits no negative segment. Then, 
	$S$ is positive  
	by Remark   \ref{dsdsdsdskjdsjkdskjdskjdsds09w0909ew09we09ewewewewewewewcx} (Lemma \ref{asklddd}). 
	Now, Proposition  \ref{dhjfdhjfssasa} shows that $S$ admits no $z$-decomposition. Hence, if $S$ admits a negative segment, then Lemma \ref{jkdsklsdlkdsjkl} shows that $S$ admits a negative segment $\MFree(S)\ni \Sigma'\subset \innt[S]$. Then, Lemma \ref{fhjfdsfdd}.\ref{fhjfdsfdd2} shows  $\Sigma$ is negative, which contradicts the assumption that $\Sigma$ is positive. \qedhere
\end{itemize}
\endgroup
\end{proof}
\noindent
We are ready for the proof of Theorem \ref{dfsofgofg}.
\begin{proof}[Proof of Theorem \ref{dfsofgofg}]
The first statement and Part \ref{dfsofgofg1} are clear from Proposition \ref{dhjfdhjfssasa} and Remark \ref{kjfdjlkfdjkfd}. Moreover, Part \ref{dfsofgofg2} is clear from Lemma  \ref{kjdfdfjkdfjkldfjkldfjkldfjkldfjklfk} as well as Proposition \ref{ghdgddgf}.\ref{ghdgddgf2} ($S\cong \UE$) and Lemma \ref{edffds} ($S\cong \ID$). Finally, Part \ref{dfsofgofg3} is clear from Lemma \ref{fhjfdsfdd}.\ref{fhjfdsfdd1} and  Lemma \ref{fhjfdsfdd}.\ref{fhjfdsfdd3}.
\qedhere
\end{proof}
\begin{corollary}
\label{odsosdopds}
Assume that $S\cong\UE$ is negative. Then, the length $\lent{S}$ of $S$ is odd.
\end{corollary}
\begin{proof}
Let $\MFree(S)\ni\Sigma\subset \innt[S]$ be negative, with $\Sigma$-decomposition $\Sigma_0,\dots,\Sigma_n$ and $[g_0],\dots,[g_n]$ for $n\geq 1$. 
\begingroup
\setlength{\leftmargini}{11pt}
\begin{itemize}
\item 
We write\:  
$\Sigma_0\cap\Sigma_n=\partial_{\Sigma_0}\cap\partial_{\Sigma_n}=\{z_0\}$\: as well as\:   
$\Sigma_{p-1}\cap\Sigma_p=\partial_{\Sigma_{p-1}}\cap\partial_{\Sigma_p}=\{z_p\}$\: for\: $p=1,\dots,n$.
\item
We define\:      
$h_p=g_{p}\cdot g_{p-1}^{-1}$\: for\: $p=1,\dots,n$,\: and recall  (see \eqref{dsnmdsdszucxzuhjcxkjdsiuwiueew433443433443})
\begin{align}
\label{nmsdnmdsnmsnmsdnmsdkjsdsdiusdiuds98ds98ds98ds98dsew}
g_n\cdot \iota(\Sigma_0)=\iota(\Sigma_n)\qquad\quad\text{as well as}\qquad\quad h_p\cdot \iota(\Sigma_{p-1})=\iota(\Sigma_p)\quad\:\: \text{for}\quad\:\: p=1,\dots,n.\hspace{7.4pt}
\end{align}
\vspace{-22pt}
\end{itemize}
\endgroup
\noindent
Since $\Sigma_p$ is negative for $p=1,\dots,n$, we have
\begin{align*}
	g_{n}\in G_{z_0}\qquad\quad \text{as well as}\qquad\quad h_p\in G_{z_p}\quad\: \text{for}\quad\:\: p=1,\dots,n.\hspace{5.9pt}
\end{align*}   
Together with the right side of \eqref{nmsdnmdsnmsnmsdnmsdkjsdsdiusdiuds98ds98ds98ds98dsew}, this implies
\begin{align*}
h_{p+2}\cdot h_{p+1} \cdot \iota(z_p) = h_{p+2}\cdot \iota(z_{p+2})=\iota(z_{p+2})\qquad\quad\forall\: p=0,\dots,n-2.
\end{align*} 
Then, if $n\geq 2$ is even (i.e., if the claim is wrong), we obtain
\begin{align*}
	\iota(z_0)=g_n\cdot \iota(z_0)=(h_n\cdot g_{n-1})\cdot \iota(z_0)={\dots}=(h_{n}\cdot h_{n-1})\cdot (h_{n-2}\cdot h_{n-3})\cdot  {\dots}\cdot (h_2\cdot h_1)\cdot \iota(z_0)=\iota(z_n),
\end{align*}
which contradicts Point \ref{defo1} in Definition \ref{gfgfgf}.
\end{proof}

\subsection{Positive and Negative Decompositions}
\label{asasaghhgfg}
In this final section, 
we investigate the positive and the negative case in more detail. In particular, we provide explicit formulas for the classes  that occur in a $\Sigma$-decomposition of $S$, and show  that in the positive case, the maximal free segments are  continuously distributed in $S$.
\subsubsection{Positive Decompositions}
\label{osdfuisfdauoifufsda}
In this subsection, we investigate  the case where $S$ is positive in more detail:
\begin{proposition}
\label{sadfpoifdsjk}
Assume that $S$ is positive; and let $\MFree(S)\ni \Sigma\subset\innt[S]$ be compact. 
\begingroup
\setlength{\leftmargini}{15pt}
\begin{enumerate}
\item
\label{sadfpoifdsjk2}
Let $S\cong\UE$ with  
$\Sigma$-decomposition $\Sigma_0,\dots,\Sigma_n$ and $[g_0],\dots,[g_n]$. Then, 
\begin{align}
\label{sdaodfsiofdsj1}
\:[g_k]=[g^k]\qquad\hspace{2pt}\forall\: k=0,\dots,n
\qquad\quad\text{holds for each}\qquad\quad g\in [g_1],
\end{align}
hence $[g^n]=[g^{-1}]$ by Lemma \ref{asklddd}.
\item
\label{sadfpoifdsjk1}
Let $S\cong \ID$ 
with $\Sigma$-decomposition $(\{\Sigma_n\}_{n\in \cn},\{[g_n]\}_{n\in \cn})$. Then, 
\begin{align}
\label{sdaodfsiofdsj}
\:[g_n]=[g^n]\qquad \forall\: n\in \cn
\qquad\quad\text{holds for each}\qquad\quad g\in [g_1].
\end{align}
\end{enumerate}
\endgroup 
\end{proposition}
\begin{proof}
We only prove Part \ref{sadfpoifdsjk1},  
because Part \ref{sadfpoifdsjk2} follows 
analogously. According to Lemma \ref{fdfdfdfd}, it suffices to show that \eqref{sdaodfsiofdsj} holds for $g=g_1$. For this, let  $\{z_n\}_{n\in \cnN}\subseteq S$  be as in  Remark \ref{dfxhbghg}.\ref{sd9898dsoiewioewuizewzuuz34387kjredsmx}, and let $\{h_n\}_{n\in \cnN}\subseteq G$ be as in Remark \ref{dfxhbghg}.\ref{dfxhbghg998ds98ds}. Together with  
\eqref{iuiuiujkjhggftfdfdfd54545454545454545trrttrtrtr} and  \eqref{dasposaposao988mdsjdskjdskjdsoidsoidsoisd98ds98dsdsdd11},    Lemma \ref{dfggfgf} applied to the $\Sigma$-decompositions (from Remark \ref{dfxhbghg}.\ref{dfxhbghg3}) that correspond to the compact maximal segments (see Remark \ref{dfxhbghg}.\ref{dfxhbghgFree}) $\{\Sigma_n\}_{\cn_-<n<\cn_+}$,   
shows that
\begingroup
\setlength{\leftmargini}{11pt}
\begin{itemize}
\item
	$[h_n]$\hspace{28.7pt} is the unique class from Proposition \ref{prop:shifttrans} applied to $\Sigma_{n-1}$ and $z_n$,\hspace{10.2pt} for $1\leq n\in \cn$.
\item
	$[(h_{n-1})^{-1}]$ is the unique class from Proposition \ref{prop:shifttrans} applied to $\Sigma_{n-1}$ and $z_{n-1}$,  for $2\leq n\in \cn$.
\end{itemize}
\endgroup
\noindent
Since $\Sigma_{n-1}$ is positive for $1\leq n\in \cn$, Lemma \ref{asklddd} (with  $z_-\equiv z_{n-1}$ and $z_+\equiv z_{n}$ there) yields 
\vspace{-3pt}
\begin{align*}
[h_{n}]\stackrel{\eqref{dskjdskjnmdsnmdsiucsiuszudsd76ds87d87ds98798dsdsdsdsds}}{=}[((h_{n-1})^{-1})^{-1}]=[h_{n-1}]\qquad\quad \forall\:2\leq n\in \cn.
\end{align*}
Since $h_1=g_1=g$ holds, we inductively obtain $[h_n]=[g]$  for all $1\leq n \in \cn$. Hence, for $1\leq n\in \cn$, \eqref{sdaodfsiofdsj} is clear from the second line in  \eqref{dsndsnmdsnmdsjdskjdskjdsoidsoidsoisdds98ds9898ds98dsdsdd}  as well as Lemma \ref{fdfdfdfd}. Since $[g_{-1}]=[g_1^{-1}]=[g^{-1}]$ holds by Lemma \ref{asklddd} and Lemma \ref{fdfdfdfd}, we can argue analogously to show that \eqref{sdaodfsiofdsj} also holds for all $\cn\ni n\leq -1$.
\end{proof}
\begin{convention}
For the rest of this subsection, we consider the situation in Proposition \ref{sadfpoifdsjk}, i.e., 
$S$ is positive with fixed compact $\MFree(S)\ni \Sigma\subset\innt[S]$, and fixed  $\Sigma$-decomposition $\SD$ of $S$:
\begingroup
\setlength{\leftmargini}{11pt}
\begin{itemize}
\item 
$\Sigma_0,\dots,\Sigma_n$\:\: and\:\: $[g_0],\dots,[g_n]$\:\: if\:\: $S\cong \UE$\:\: holds. 
\item
\hspace{21pt}$(\{\Sigma_n\}_{n\in \cn},\{[g_n]\}_{n\in \cn})$\:\hspace{21.5pt}
 if\:\: $S\cong \ID$\hspace{16.8pt} holds.
\end{itemize}
\endgroup
\noindent
In the first case, we set $\Sigma_{n+1}:=\Sigma_0$.
\end{convention}
\begin{remark}
\label{fdkjlfdjlkfdlksdsdsdsjfdjkfd}
\noindent

\vspace{-5pt}
\begingroup
\setlength{\leftmargini}{15pt}
{
\renewcommand{\theenumi}{{\sf\arabic{enumi})}} 
\renewcommand{\labelenumi}{\theenumi}
\begin{enumerate}
\item
\label{fdkjlfdjlkfdlksdsdsdsjfdjkfd2}
Assume that $S\cong\UE$ holds; and let $\ovl{\Sigma}_0,\dots,\ovl{\Sigma}_n$, $[\ovl{g}_0],\dots,[\ovl{g}_n]$  with $n\geq 2$, be the (``inversely orientated'') decomposition from Lemma \ref{asdhlkdsajkd}. Then, it is immediate from Proposition  \ref{sadfpoifdsjk}.\ref{sadfpoifdsjk2} that
\begin{align*}
\:[\ovl{g}_k]=[g^{-k}]\qquad\hspace{2pt}\forall\: k=0,\dots,n
\qquad\quad\text{holds for each}\qquad\quad g\in [g_1].
\end{align*}
\vspace{-15pt}
\item
\label{fdkjlfdjlkfdlksdsdsdsjfdjkfd1}
Assume that $S\cong\ID$ holds; and let $(\{\ovl{\Sigma}_n\}_{n\in \ocn},\{[\ovl{g}_n]\}_{n\in \ocn})$ be the (``inversely orientated'') decomposition from Remark \ref{dfxhbghg}.\ref{dfxhbghg2}.  Then, it is immediate from Proposition  \ref{sadfpoifdsjk}.\ref{sadfpoifdsjk1} that 
\begin{align*}
\:[\ovl{g}_n]=[g^{-n}]\qquad \forall\: n\in \ocn\hspace{29.6pt}\qquad\quad\text{holds for each}\qquad\quad g\in [g_1].
\end{align*}
\vspace{-15pt} 
\end{enumerate}}
\endgroup 
\end{remark} 
\noindent
We next want to figure out, which classes can occur on the right side of \eqref{sdaodfsiofdsj1} and \eqref{sdaodfsiofdsj} for any further positive segment $\MFree(S)\ni \Sigma'\subset \innt[S]$. For this, we observe the following:
\begingroup
\setlength{\leftmargini}{17pt}
{
\renewcommand{\theenumi}{{\sf\Alph{enumi})}} 
\renewcommand{\labelenumi}{\theenumi}
\begin{enumerate}
\item
Assume that $\Sigma'$ equals some $\Sigma_p$ that occurs in $\SD$. The  
$\Sigma'$-decomposition
\begingroup
\setlength{\leftmarginii}{11pt}
\begin{itemize}
\item
from Remark\hspace{5pt} \ref{sdffdsfsd}.\ref{sdffdsfsd2}\: if\: $S\cong \UE$\: holds
\item
from Remark \ref{dfxhbghg}.\ref{dfxhbghg3}\: if\: $S\cong \ID$\hspace{14.1pt} holds
\end{itemize}
\endgroup
\noindent
is called {\bf $\boldsymbol{\SD}$-oriented} in the following.  Proposition \ref{sadfpoifdsjk} and Lemma \ref{fdfdfdfd} imply $[g_1']=[g]$; hence, we have 
\begingroup
\setlength{\leftmarginii}{11pt}
\begin{itemize}
\item
$[g'_k]\hspace{0.8pt}=[g^k]$\hspace{0.3pt}\: for $k=0,\dots,n$\: if\: $S\cong \UE$\: holds,\: by Proposition \ref{sadfpoifdsjk}.\ref{sadfpoifdsjk2}.
\item
$[g_n']=[g^n]$\: for all $n\in \cn'$\hspace{12pt}\: if\: $S\cong \ID$\hspace{12.0pt}\: holds,\: by Proposition  \ref{sadfpoifdsjk}.\ref{sadfpoifdsjk1}. 
\end{itemize}
\endgroup
\item
Assume that $\Sigma'\neq \Sigma_n$ holds for each $\Sigma_n$ that occurs in $\SD$. Then, there exists 
\begingroup
\setlength{\leftmarginii}{11pt}
\begin{itemize}
\item
\hspace{7pt}$0\leq p\leq n$\:\hspace{5.8pt} if\: $S\cong \UE$
\item
$\cn_-\leq p< \cn_+$\: if\: $S\cong \ID$
\end{itemize}
\endgroup
\noindent
unique 
as well as 
$z\in \partial_{\Sigma_{p}}\cap \partial_{\Sigma_{p+1}}$ unique,    
such that \hspace*{\fill}
(recall $\Sigma_{n+1}=\Sigma_0$ if $S\cong \UE$)
\begin{align}
\label{sdsdsds}
	\big|\partial_{\Sigma'}\cap \innt[\Sigma_p]\big|=1=\big|\partial_{\Sigma'}\cap \innt[\Sigma_{p+1}]\big|,\qquad\Sigma'\subset \innt[\Sigma_{p}\cup \Sigma_{p+1}],\qquad  z\in\innt[\Sigma'].
\end{align} 
{\it Proof of the Claim.}
By maximality (see Remark \ref{dfxhbghg}.\ref{dfxhbghgFree} for $S\cong \ID$), we can neither have $\Sigma_\ell\subset \Sigma'$ nor $\Sigma'\subset \Sigma_\ell$ 
\begingroup
\setlength{\leftmarginii}{11pt}
\begin{itemize}
\item
for\: $\ell=0,\dots,n$\:\hspace{4pt} if\: $S\cong\UE$\: holds,
\item
for\: $\cn_-<\ell<\cn_+$\: if\: $S\cong\ID$\hspace{14pt} holds.
\end{itemize}
\endgroup
\noindent
Moreover, since $\MFree(S)\ni\Sigma'\subset \innt[S]$ is compact, in the case $S\cong \ID$, we cannot have
\begin{align*}
	\Sigma'\subseteq \Sigma_{\cn_-}\quad \text{if}\quad \cn_-\neq -\infty\qquad\quad \text{or}\qquad\quad \Sigma'\subseteq \Sigma_{\cn_+}\quad \text{if}\quad \cn_+\neq\infty.
\end{align*}
 From this, the claim is clear.\hspace*{\fill}$\ddagger$
\vspace{2pt}
 
\noindent 
Now, let us write $\partial_{\Sigma'}=\{z_-,z_+\}$, whereby   
$z_-$ and $z_+$ denote the boundary points of $\Sigma'$ that are contained in $\innt[\Sigma_{p}]$ and $\innt[\Sigma_{p+1}]$, respectively. The unique $\Sigma'$-decomposition $\SD'$ of $S$ with $z_+\in\partial_{\Sigma'_{1}}$ is called {\bf $\boldsymbol{\SD}$-oriented} in the following.\footnote{Recall  that if $S\cong\UE$ and $\lent{S}=n=1$ holds (see Proposition \ref{ghdgddgf}.\ref{ghdgddgf3}), then Lemma \ref{asdhlkdsajkd} shows that there exists  
only one $\Sigma'$-decomposition of $S$. In the other cases, the condition $z_+\in\Sigma'_{1}$ obviously fixes a unique $\Sigma'$-decomposition of $S$.} 
Let $\SD'$ be given by   
\begingroup
\setlength{\leftmarginii}{11pt}
\begin{itemize}
\item 
$\Sigma'_0,\dots,\Sigma'_n$\:\: and\:\: $[g'_0],\dots,[g'_n]$\:\: if\:\: $S\cong \UE$\:\: holds, \hspace*{\fill}(recall  Proposition \ref{ghdgddgf}.\ref{ghdgddgf3})
\item
\hspace{18.3pt}$(\{\Sigma'_n\}_{n\in \cn'},\{[g'_n]\}_{n\in \cn'})$\:\hspace{18.5pt}
 if\:\: $S\cong \ID$\hspace{16.3pt} holds,
\end{itemize}
\endgroup
\noindent 
and observe the following:  
\begingroup
\setlength{\leftmarginii}{14pt}
{
\renewcommand{\theenumii}{.{\small\sf\arabic{enumii})}} 
\renewcommand{\labelenumii}{{\small\sf\arabic{enumii})}}
\begin{enumerate}
\item
\label{poepopwepowepwoepowpwoeweopop1}
	Since $\Sigma'$ is positive, $g_1'\cdot \iota(\Sigma'_{z_-})=\iota(\Sigma'_{z_+})$ holds for\hspace*{\fill}(\he$g_1'\cdot \iota(z_-)=\iota(z_+)$\he) 
	\begingroup
\setlength{\leftmarginiii}{11pt}
\begin{itemize}
\item
	 $\Sigma'_{z_-}$ a compact boundary segment of $\Sigma'_0$ with $\partial_{\Sigma'_{z_-}}\ni z_-$,
\item
	$\Sigma'_{z_+}$ a compact boundary segment  of $\Sigma'_1$ with $\partial_{\Sigma'_{z_+}}\ni z_+$.
\end{itemize}
\endgroup
\item
\label{poepopwepowepwoepowpwoeweopop2}
Since $z_-\in \innt[\Sigma_{p}]$ and $z_+\in \innt[\Sigma_{p+1}]$ holds, Point {\small\sf 1)} implies $g_1'\cdot \iota|_{\Sigma_{p}}\cpsim \iota|_{\Sigma_{p+1}}$:
	\begingroup
\setlength{\leftmarginiii}{11pt}
\begin{itemize}
\item
	Assume  $S\cong\ID$, or $S\cong\UE$ with $0\leq p< n$. Then,  
	Proposition \ref{sadfpoifdsjk} together with Lemma \ref{fdfdfdfd} (first step) yields:
\begin{align*}
	g_1'\cdot \iota|_{\Sigma_{p}}\cpsim \iota|_{\Sigma_{p+1}} \quad\:\:\:\Longrightarrow\quad\:\:\: g_1'\cdot \iota|_{\Sigma_{p}}\cpsim g\cdot \iota|_{\Sigma_{p}}\quad\:\:\:\Longrightarrow\quad\:\:\: [g_1']=[g]\hspace{101.8pt}
\end{align*}
\item
	Assume $S\cong\UE$ with $p=n$ ($\Sigma_{n+1}=\Sigma_0$). Then,  
	Proposition \ref{sadfpoifdsjk}.\ref{sadfpoifdsjk2} (second and  third step) yields:
\begin{align*}
	g_1'\cdot \iota|_{\Sigma_{p}}\cpsim \iota|_{\Sigma_{p+1}} \quad\:\:\:&\Longrightarrow\quad\:\:\: (g'_1)^{-1}\cdot \iota|_{\Sigma_{0}}\cpsim \iota|_{\Sigma_{n}} \qquad\quad\:\:\:\Longrightarrow\qquad\:\:\: (g'_1)^{-1}\cdot \iota|_{\Sigma_{0}}\cpsim g^n\cdot \iota|_{\Sigma_{0}}  \\
&\Longrightarrow\quad\:\:\: [(g'_1)^{-1}]=[g^n]=[g^{-1}]\quad\:\:\: \stackrel{\text{Lemma }\ref{fdfdfdfd}}{\Longrightarrow}\quad\:\:\: [g'_1]=[g]
\end{align*}	
\end{itemize}
\endgroup
\end{enumerate}}
\endgroup
\end{enumerate}}
\endgroup
\noindent
We have shown the following statement: 
\begin{lemma}
\label{fgfghghg}
Assume that $S$ is positive; and let $\MFree(S)\ni \Sigma\subset\innt[S]$ be compact.  
Then, the class that occurs on the right sides of \eqref{sdaodfsiofdsj1} and \eqref{sdaodfsiofdsj} is the same for each $\SD$-oriented $\Sigma'$-decomposition of $S$ ($\MFree(S)\ni \Sigma'\subset \innt[S]$).
\end{lemma}
\noindent
We finally want to show that in the positive case, $S$ indeed admits much more positive segments than only those  
that 
belong to $\SD$:
\begin{proposition}
\label{lkfdlkfdlkfdlkfdlk}
Let $S$ be positive; and let $\MFree(S)\ni \Sigma\subset\innt[S]$ be compact with $\Sigma$-decomposition $\SD$. We write $\partial_{\Sigma_0}=\{z_0,z_1\}$ with $z_1\in \partial_{\Sigma_1}$ (\he hence\hspace{1pt} $g_1\cdot \iota(z_0)=\iota(z_1)$), and let $\Seg(S)\ni \KK\subseteq \Sigma_1$ be compact with $\partial_\KK=\{z_1,z_1'\}$. Then, there exists $\MFree(S)\ni\Sigma'\subseteq \Sigma_0\cup\Sigma_1$ compact (necessarily positive) with 
$$
\KK\subseteq \Sigma'\qquad \text{and}\qquad  \partial_{\Sigma'}=\{z_0',z_1'\}\qquad\text{for some}\qquad z_0'\in \Sigma_0.
$$ 
Additionally, the following assertions hold:
\begingroup
\setlength{\leftmargini}{15pt}{
\renewcommand{\theenumi}{{\alph{enumi}})} 
\renewcommand{\labelenumi}{\theenumi}
\begin{enumerate}  
\item
\label{lkfdlkfdlkfdlkfdlks1}
We have the implication:
\begin{align*}
 \hspace{70.5pt}
  z_1'\in \innt[\Sigma_1]
\qquad\:\:\Longrightarrow\qquad\:\: z_0'\in \innt[\Sigma_0]\quad\wedge\quad z_1\in \innt[\Sigma']\quad\wedge\quad \Sigma'\subset \innt[\Sigma_0\cup\Sigma_1]
\end{align*}
\vspace{-17pt}    
\item
\label{lkfdlkfdlkfdlkfdlks2}
If $S\cong \ID$ holds with $\cn_+=1$,  then we have the  implication:
\begin{align*}
\KK=\Sigma_1 
\quad\wedge\quad \Sigma_+\subset \Sigma_0\qquad\:\:\Longrightarrow\qquad\:\: 
z_0'\in \innt[\Sigma_0]\quad\wedge\quad z_1\in \innt[\Sigma']\quad\wedge\quad \Sigma'\subseteq \Sigma_0\cup\Sigma_1
\end{align*}
\end{enumerate}}
\endgroup 
\end{proposition}
\begin{proof} 
We first observe that \hspace*{\fill}(use that $S\cong \UE$ or $S\cong \ID$ holds)
\vspace{-16pt}
$$
\Seg(S)\ni S':= \overbrace{(\Sigma_0 \setminus \{z_0\})}^{\displaystyle =: \tilde{\Sigma}_0\in \Seg(S)}\hspace{-1.5pt}\cup\:\hspace{1.5pt} \KK\qquad \text{holds, with}\qquad S'\subset S.
$$ 
Then, Remark \ref{cnmnmcviufeiureiure}.\ref{cnmnmcviufeiureiure2}   
provides a homeomorphism 
$\homeo\colon S'\rightarrow \ID'\equiv (x_0,x_1']$ with
\begin{align*}
	\homeo^{-1}(\tilde{\Sigma}_0)=(x_0,x_1]\quad\text{and}\quad  \homeo^{-1}(\KK)=[x_1,x_1']\qquad\:\: \text{for}\qquad\:\:
	x_1:=\homeo^{-1}(z_1)\quad\text{and}\quad  x_1':=\homeo^{-1}(z_1').
\end{align*}
\begingroup
\setlength{\leftmargini}{11pt}
\begin{itemize}
\item
We have $\KK\in \Free(S)$ by Lemma \ref{jdkjfdkjfd}, as $\KK\subseteq \Sigma_1\in \Free(S)$ holds. Hence, $\KK\in \Free(S')$ holds by \eqref{jdkjfdkjfddf} as $\KK\subseteq S'$. Lemma \ref{fhdhh} thus provides some $\Sigma'\in \MFree(S')$ with $\KK\subseteq \Sigma'$. 
\item
We set $D:=\homeo^{-1}(\Sigma')$. Then,    
$\ID'\neq D$  holds, because
\vspace{-15pt}
$$
\ID'=D\qquad\Longrightarrow\qquad S'=\Sigma'\in \Free(S')\qquad\stackrel{\eqref{jdkjfdkjfddf}}{\Longrightarrow}\qquad S'\in \Free(S)
\quad\:\:\stackrel{\text{Lemma }\ref{kjkjdskjkjdsaassaa}}{\Longrightarrow}\quad\:\:  \Sigma_0\subset \overbrace{\Sigma_0\cup \KK}^{ =\:\clos[S']}\in \Free(S),
$$
which contradicts $\Sigma_0\in \MFree(S)$. 
 Since 
 $\Sigma'$ is closed in $S'$ by Lemma \ref{kjkjdskjkjdsaassaa}, we thus have
\begin{align*}
 D=[x_0',x_1']\qquad\text{for some}\qquad x_0<x_0' \leq  x_1<x_1'.
\end{align*} 
\vspace{-20pt}

 In particular, $\Sigma'$ is compact with $\Sigma'\subset S'$.
\item
$\Sigma'$ is positive w.r.t.\ $(S',\iota|_{S'})$; hence, $\Sigma'\in \MFree(S)$ holds by Lemma \ref{qasggpapd}. 
\vspace{1pt}

{\it Proof.} Assume that $\Sigma'$ is negative, and set $z_0':=\homeo^{-1}(x_0')$. Let $[g'],\Sigma_{z_0'},\Sigma_{b'}$ be  as in \ref{starlkdslkdslkdslksdlkdslklkdsdsdsds} in Proposition \ref{prop:shifttrans} applied to $\Sigma'$ and $z_0'$, with $g'\in G_{z'_0}$ and $b'=z_0'$. The following two cases can occur:
\vspace{-3pt} 
\begingroup
\setlength{\leftmarginii}{12pt}
\begin{itemize}
\item
$x_0'=x_1$ holds, hence $z_0'=z_1$:
\vspace{2pt}

Then, the same elementary arguments as used in Lemma \ref{dsoioidslkdskldskldskldsds98ds09ds9ds9898ds} show that $\Sigma_0$ is negative; specifically, that $[(g')^{-1}]$ is the unique class from Proposition \ref{prop:shifttrans} applied to $\Sigma_0$ and $z_1$, so that $[g_1]=[(g')^{-1}]$ holds. This, however, implies $g_1\in G_{z_1}$, which contradicts that $\Sigma_0$ is positive.
\vspace{2pt}
\item
$x_0'<x_1$ holds, hence $z_0'\in \innt[\Sigma_0]$:
\vspace{2pt}

Then, $g'\cdot \iota(\Sigma_{b'})=\iota(\Sigma_{z_0'})$ implies $g'\cdot \iota|_{\Sigma_0}\cpsim \iota|_{\Sigma_0}$, hence $g'\in G_S$ as $\Sigma_0$ is free. This, however, contradicts that $[g']\neq [e]$ holds. 
\hspace*{\fill}$\ddagger$
\end{itemize}
\endgroup
\end{itemize}
\endgroup
\noindent
It remains to prove the implications in \ref{lkfdlkfdlkfdlkfdlks1} and \ref{lkfdlkfdlkfdlkfdlks2}. In both cases, it suffices to show that the  condition on the corresponding left side implies  $\KK\subset \Sigma'$, hence $x_0'<x_1$:
\vspace{-2pt}
\begingroup
\setlength{\leftmargini}{15pt}{
\renewcommand{\theenumi}{{\alph{enumi}})} 
\renewcommand{\labelenumi}{\theenumi}
\begin{enumerate}
\item
Assume $z_1'\in \innt[\Sigma_1]$, hence $\KK\subset \Sigma_1$. 
Then,  
$\KK=\Sigma'$ yields the contradiction
$ 
\MFree(S)\ni \Sigma'=\KK\subset \Sigma_1\in \Free(S)$.  
\item
Assume $S\cong \ID$ with $\cn_+=1$, as well as $\KK=\Sigma_1$ and $\Sigma_+\subset \Sigma_0$: 
\begingroup
\setlength{\leftmarginii}{11pt}
\begin{itemize}
\item
There exists a chart $(U,\psi)$ around $z_1$ with $U\cap \Sigma_+=\emptyset$. 
\vspace{2pt}

(\he Since $\Sigma_0$ is positive, we have $z_0\in \Sigma_+$, hence $z_1\notin \Sigma_+$ as $\Sigma_+$ is a proper boundary segment of $\Sigma_0$.\he)
\end{itemize}
\endgroup
Assume now that $\KK=\Sigma'$ holds:
\begingroup
\setlength{\leftmarginii}{11pt}
\begin{itemize}
\item
We have $\Sigma_1=\KK=\Sigma'\in \MFree(S)$, and $\Sigma_1$ is positive (as $S$ is positive).  
\item
Let $[g'],\Sigma_{z_1},\Sigma_{b'}$ be  as in \ref{starlkdslkdslkdslksdlkdslklkdsdsdsds} in Proposition \ref{prop:shifttrans} applied to $\Sigma_1,z_1$, hence $g'\cdot\iota(\Sigma_{b'})=\iota(\Sigma_{z_1})$.  
\item
Shrinking $\Sigma_{z_1}$ around $z_1$ and $\Sigma_{b'}$ around $b'$, we can assume by the first point that $\Sigma_{z_1}\subseteq \Sigma_0\setminus \Sigma_+$ holds. 
\item
Then, we have $g'^{-1}\cdot \iota|_{\Sigma_0}\cpsim \iota|_{\Sigma_1}$, hence $[g'^{-1}]=[g_1]$ by Lemma \ref{dfggfgf}. \item
Lemma \ref{fdfdfdfd} shows $[g']=[g_1^{-1}]$, so that the third point yields 
$$
\iota(\Sigma_{z_1})=g_1^{-1}\cdot \iota(\Sigma_{b'})\subseteq g_1^{-1}\cdot \iota(\Sigma_1)=\iota(\Sigma_+). 
$$
\end{itemize}
\endgroup
We obtain $\iota(\Sigma_0\setminus \Sigma_+)\supseteq \iota(\Sigma_{z_1})\subseteq \iota(\Sigma_+)$, which contradicts that $\iota$ is injective.  \qedhere
\end{enumerate}}
\endgroup
\end{proof}

\begin{corollary}
\label{fdshjfdf}
Assume that $S$ is positive. Then, to each $z\in \innt[S]$,  there exists some (necessarily positive) $\tilde{\Sigma}\in \MFree(S)$  with $z\in \innt[\tilde{\Sigma}]$. 
\end{corollary}
\begin{proof}
Let $\MFree(S)\ni \Sigma\subset \innt[S]$ be compact (positive), with $\Sigma$-decomposition $\SD$:
\begingroup
\setlength{\leftmargini}{11pt}
\begin{itemize}
\item
If $S\cong \UE$ holds, then we denote $\SD$ by $\Sigma_0,\dots,\Sigma_n$, $[g_0],\dots,[g_n]$ for $n\geq 1$, with boundary points $z_0,\dots,z_{n-1}\in S$ as in Remark \ref{sdffdsfsd}.\ref{sdffdsfsddsiuiudsiusd}:\hspace*{\fill}($z_{n+1}:=z_0$\: and\: $\Sigma_{n+1}:=\Sigma_0$)
\vspace{-3pt}
\begingroup
\setlength{\leftmarginii}{12pt}
\begin{itemize}
\item
The claim is clear if $z\in \innt[\Sigma_p]$ holds for some $0\leq p\leq n$.
\item
In the other case, $z=z_{p+1}$ holds for some $0\leq p\leq n$, and we set $\Sigma':=\Sigma_{p}$. The claim now follows from Proposition \ref{lkfdlkfdlkfdlkfdlk}.\ref{lkfdlkfdlkfdlkfdlks1}, when applied to the $\Sigma'$-decomposition $\SD'$ from  Remark \ref{sdffdsfsd}.\ref{sdffdsfsd2}. 
\end{itemize}
\endgroup
\item
If $S\cong \ID$ holds, then we denote $\SD$ by $(\{\Sigma_n\}_{n\in \cn},\{[g_n]\}_{n\in \cn})$ (with $\cn\in \CN$), with boundary points $\{z_n\}_{n\in \cnN}\subseteq S$ as in Remark \ref{dfxhbghg}.\ref{sd9898dsoiewioewuizewzuuz34387kjredsmx}:
\vspace{-3pt}
\begingroup
\setlength{\leftmarginii}{12pt}
\begin{itemize}
\item
Assume that $z\in \innt[\Sigma_n]$ holds for some $\cn_-<n<\cn_+$. Then, the claim is clear from Remark \ref{dfxhbghg}.\ref{dfxhbghgFree}. 
\item
Assume that $z\in \innt[\Sigma_n]$ holds for $-\infty<\cn_-=n$,    or for $n=\cn_+<\infty$: 
\begingroup
\setlength{\leftmarginiii}{11pt}
\begin{itemize}
\item
We can assume that  $n=\cn_+<\infty$ holds, just by 
replacing $\SD$ by the  (``inversely orientated'') $\Sigma$-decomposition from Remark \ref{dfxhbghg}.\ref{dfxhbghg2} in the case $-\infty<\cn_-=n$. 
\item
We set $p:=\cn_+-1$ and $\Sigma':=\Sigma_{p}$. Then,  
we can assume that  $n=\cn_+=1$ holds, just by 
replacing $\SD$ by the $\Sigma'$-decomposition  from  Remark \ref{dfxhbghg}.\ref{dfxhbghg3}.
\item
The claim now follows from Proposition \ref{lkfdlkfdlkfdlkfdlk}.\ref{lkfdlkfdlkfdlkfdlks1}, just by fixing a compact segment $\KK\subseteq \Sigma_{\cn_+}$ with $z_{\cn_+}\in \partial_\KK$ and $z\in \innt[\KK]$.
\end{itemize}
\endgroup
\item 
Assume that $z=z_{n}$ holds for some $n\in \cnN$.  
\begingroup
\setlength{\leftmarginiii}{11pt}
\begin{itemize}
\item
We can assume that $n\geq 1$ holds, just by replacing $\SD$ by the (``inversely orientated'') $\Sigma$-decomposition from Remark \ref{dfxhbghg}.\ref{dfxhbghg2} in the case $n\leq -1$. 
\item
We set $p:=n-1$ and $\Sigma':=\Sigma_{p}$. Then,  
we can assume that  $n=1$ holds, just by 
replacing $\SD$ by the $\Sigma'$-decomposition  from  Remark \ref{dfxhbghg}.\ref{dfxhbghg3}. 
\item 
The claim now follows from Proposition \ref{lkfdlkfdlkfdlkfdlk}.\ref{lkfdlkfdlkfdlkfdlks1}.\qedhere
\end{itemize}
\endgroup
\end{itemize}
\endgroup
\end{itemize}
\endgroup
\end{proof}

\subsubsection{Negative Decompositions}
In this final subsection, we investigate  the case where $S$ is negative in more detail. Specifically, we provide explicit formulas for the classes that occur in a fixed $\Sigma$-decomposition of $S$ with $\MFree(S)\ni\Sigma\subset\innt[S]$.   
For this, we define the map  
  $\sigma\colon \ZZ_{\neq 0}\rightarrow \{-1,1\}$  by 
\begin{align*}
\sigma(n):=
\begin{cases} 
	(-1)^{n-1} &\mbox{if}\quad n > 0 \\ 
	(-1)^n & \mbox{if}\quad n < 0.
\end{cases} 
\end{align*}
We have the following analogue to Proposition  \ref{sadfpoifdsjk}:
\begin{proposition}
\label{sdsdsdsdffghhh}
Assume that $S$ is negative; and let $\MFree(S)\ni \Sigma\subset\innt[S]$ be compact. 
\begingroup
\setlength{\leftmargini}{15pt}
\begin{enumerate}
\item
\label{sdsdsdsdffghhh2}
Let $S\cong\UE$ with  
$\Sigma$-decomposition $\Sigma_0,\dots,\Sigma_n$ and $[g_0],\dots,[g_n]$, and set $g_{-1}:=g_n$. Then,   
\begin{align}
\label{sdsdffghhh2}
 \:[g_k]=[g_{\sigma(1)}\cdot {\dots}\cdot g_{\sigma(k)}]
 \hspace{21.5pt}\qquad\quad\forall\: 1\leq k\leq n.
\end{align}
\item
\label{sdsdsdsdffghhh1}
Let $S\cong \ID$ 
with $\Sigma$-decomposition $(\{\Sigma_n\}_{n\in \cn},\{[g_n]\}_{n\in \cn})$. 
Then,
\begin{align}
\label{sdsdffghhh1}
 \:[g_n]=[g_{\sigma(\sign(n))}\cdot {\dots}\cdot g_{\sigma(n)}]\qquad\quad\forall\: n\in \cnN.\footnotemark \hspace{6.5pt}
\end{align}
\footnotetext{Recall $\sign\colon \ZZ\setm\{0\}\ni \pm\he |n|\mapsto \pm 1\in \{-1,1\}$.}
\end{enumerate}
\endgroup
\end{proposition}
\noindent
The proof of Proposition \ref{sdsdsdsdffghhh} will now be established step by step: 
\begin{claim}
\label{sdoidoidsoisoidoidsoidsiodsdsds}
 Lemma \ref{sdsdsdsdffghhh}.\ref{sdsdsdsdffghhh2} follows from Lemma \ref{sdsdsdsdffghhh}.\ref{sdsdsdsdffghhh1}.
\begin{proof}
Assume $S\cong\UE$, and let $\SD$ denote the corresponding $\Sigma$-decomposition in Lemma \ref{sdsdsdsdffghhh}.\ref{sdsdsdsdffghhh2}.  
Obviously, Lemma \ref{sdsdsdsdffghhh}.\ref{sdsdsdsdffghhh2} holds for $n=1$. Let thus $n\geq 2$: 
\begingroup
\setlength{\leftmargini}{11pt}
\begin{itemize}
\item
We fix $\Seg(S)\ni S'\subset S$ open with $\Sigma_0,\dots,\Sigma_{n-1}\subseteq S'$.
\item
We let $\Sigma'_{-1},\Sigma'_n\in \Seg(S)$ denote the components of $S'\cap \Sigma_n$ (boundary segments of $\Sigma_n$) that share a boundary point with $\Sigma_0,\Sigma_{n-1}$, respectively. 
\item
We define a $\Sigma$-decomposition $\SD'\equiv (\{\Sigma'_n\}_{n\in \cn},\{[g'_n]\}_{n\in \cn})$ of $S'$ by
\vspace{-2pt} 
\begingroup
\setlength{\leftmarginii}{13pt}
\begin{itemize}
\item
\hspace{6.5pt}$\cn:=\{-1,0,1,\dots,n\}$,
\vspace{3pt}
\item
$\Sigma'_p:=\Sigma_p$\: for\: $p=0,\dots,n-1$\: as well as\: $\Sigma'_{-1},\Sigma'_n$\: as above,
\vspace{2pt}
\item
\hspace{2.4pt}$g'_p:=g_p$\:\hspace{2.4pt} for\: $p=0,\dots,n$\:\qquad\hspace{0.5pt}as well as\: $g_{-1}:=g_n$.
\end{itemize}
\endgroup
\end{itemize}
\endgroup
\noindent
Then, \eqref{sdsdffghhh2} holds for $\SD$ if \eqref{sdsdffghhh1} holds for $\SD'$.\qedhere
\end{proof}
\end{claim}
\noindent
Now, to investigate the situation in Lemma \ref{sdsdsdsdffghhh}.\ref{sdsdsdsdffghhh1}, we first observe the following:
\begin{remark}
\label{dspodslkdewkjewiuewiukjdskjds9898ds98ds98dsds}
Assume that $S\cong \ID$ holds; and let furthermore
\begingroup
\setlength{\leftmargini}{11pt}
\begin{itemize}
\item[$*$]
$\MFree(S)\ni\Sigma_\pm\subset \innt[S]$ be negative, with 
\begin{align*}
\partial_{\Sigma_\pm}=\{z_\pm,z\}\qquad\quad&\text{and}\qquad\quad\Sigma_-\cap\Sigma_+=\partial_{\Sigma_-}\cap\partial_{\Sigma_+}=\{z\}\\[2pt]
\text{as }&\text{well as}\\[-4pt]
	h\cdot \iota(\Sigma_\pm)=\iota(\Sigma_\mp)\qquad\text{for}\qquad h\in G_z\setm G_S\qquad\text{h}\!&\:\text{ence}\qquad h\cdot \iota(z_\pm)=\iota(z_\mp)\qquad\text{and}\qquad [h]\stackrel{\eqref{ajaahhdsadfdgdgfs}}{=}[h^{-1}].\hspace{15pt} 
\end{align*} 
\vspace{-17pt}
\item[$*$]
$\Seg(S)\ni \Sigma_{\pm\pm}\subset S$ be closed in $S$, with $\Sigma_{\pm\pm}\cap \Sigma_\pm=\partial_{\Sigma_{\pm\pm}}\cap \partial_{\Sigma_{\pm}}=\{z_\pm\}$ such that the following holds:
\vspace{-2pt}
\begingroup
\setlength{\leftmarginii}{11pt}
\begin{itemize}
\item
$\iota|_{\Sigma_{\pm\pm}}$ is an embedding.
\vspace{3pt}
\item
There exist $h_\pm\in G_{\pm z}\setm G_S$, as well as  boundary segments $\Sigma_{z_\pm}\subseteq \Sigma_\pm$ of $\Sigma_\pm$   with 
\begin{align}
\label{fdfdsoelkfdsoelfds}
\hspace{22.5pt}z_{\pm}\in \partial_{\Sigma_{z_\pm}}\qquad\quad
\text{and}\qquad\quad 
h_\pm\cdot \iota(\Sigma_{z_\pm})= \iota(\Sigma_{\pm\pm}).\:\:\: 
\end{align}
\end{itemize}
\endgroup 
\end{itemize}
\endgroup
\noindent
Then, the following two identities hold:
\begin{align}
\label{ldslkdlkdsoidsoidsoids09ds09ds09ds09sd09dsdsds}
	[h_-]=[h\cdot h_+\cdot h]\qquad\quad \text{and}\qquad\quad [h_+]=[h\cdot h_-\cdot h]\:\:\:
\end{align}
\begin{proof}
We only prove the left side of \eqref{ldslkdlkdsoidsoidsoids09ds09ds09ds09sd09dsdsds}, because   the right side follows analogously:
\begingroup
\setlength{\leftmargini}{11pt}
\begin{itemize}
\item  
Since $h\cdot \iota(\Sigma_-) =\iota(\Sigma_+)$ holds, with $h\cdot \iota(z_-)=\iota(z_+)$ and $z_\pm\in \innt[S]$, Lemma \ref{lemma:BasicAnalytt1} implies that 
\begin{align}
\label{dshjdsaudsaidsa}
h\cdot \iota(\Sigma'_{--}) =\iota(\Sigma'_{++})
\end{align}
holds for certain compact boundary segments $\Sigma'_{\pm\pm}\subseteq \Sigma_{\pm\pm}$ of $\Sigma_{\pm\pm}$ with $z_\pm\in \partial_{\Sigma'_{\pm\pm}}$. 
\item 
We obtain from \eqref{dshjdsaudsaidsa} that
\begin{align*}
h \cdot \iota|_{\Sigma_{--}}\cpsim \iota|_{\Sigma_{++}}\qquad\quad&\stackrel{\eqref{fdfdsoelkfdsoelfds}}{\Longrightarrow}\qquad\quad\hspace{34.7pt}
	h \cdot \iota|_{\Sigma_{--}}\cpsim h_+\cdot \iota|_{\Sigma_{+}}\\
	\qquad\quad&\stackrel{\eqref{fdfdsoelkfdsoelfds}}{\Longrightarrow}\qquad\quad \hspace{12.6pt}(h \cdot h_-)\cdot\iota|_{\Sigma_{-}}\cpsim h_+\cdot \iota|_{\Sigma_{+}}\\
	\qquad\quad&\stackrel{\phantom{\eqref{fdfdsoelkfdsoelfds}}}{\Longrightarrow}\qquad\quad (h \cdot h_-\cdot h)\cdot\iota|_{\Sigma_{+}}\cpsim h_+\cdot \iota|_{\Sigma_{+}}\\
&\hspace{-10.2pt}\stackrel{\Sigma_+\he\in\:  \Free(S)}{\Longrightarrow}\qquad\hspace{55pt} [h_+]=[h\cdot h_-\cdot h]	
\end{align*}
\vspace{-22pt}

holds, whereby we have applied $h\cdot \iota(\Sigma_+)=\iota(\Sigma_-)$ in the third step. 
\qedhere
\end{itemize}
\endgroup
\end{proof}
\end{remark}
\noindent
Assume now that we are in the situation of Lemma \ref{sdsdsdsdffghhh}.\ref{sdsdsdsdffghhh1}: 
\begingroup
\setlength{\leftmargini}{11pt}
\begin{itemize}
\item
We choose   
$\{h_n\}_{n\in \cnN}\subseteq G$ as in Remark \ref{dfxhbghg}.\ref{dfxhbghg998ds98ds}. Then, we have
\begin{align*}
h_n=g_n\cdot g_{n+1}^{-1}\quad\text{for}\quad \cn_-\leq n\leq -1\qquad\quad\text{as well as}\qquad\quad h_n=g_n\cdot g_{n-1}^{-1}\quad\text{for}\quad 1\leq n\leq \cn_+,\hspace{25pt}
\end{align*} 
whereby Lemma  \ref{dspodsopodssdklkdss98ds98ds98ds} yields
\begin{align}
\label{safgrfgtr}
[h_n]\stackrel{\eqref{ajaahhdsadfdgdgfs}}{=}[h_n^{-1}]\qquad\forall\: n\in \cn\qquad\quad \text{hence}\qquad\quad [g_{\pm1}]=[h_{\pm1}]=[h^{-1}_{\pm1}]=[g^{-1}_{\pm1}].
\end{align}
\item  
Remark \ref{dspodslkdewkjewiuewiukjdskjds9898ds98ds98dsds} provides the following identities:
\begin{align}
\begin{split}
\label{twerfdff}
	&\hspace{-30pt}[h_{n-1}]=[h_{n}\cdot h_{n+1}\cdot h_{n}]\qquad\forall\: \cn_-< n\leq -2\qquad\quad\text{and}\qquad\quad [h_{-2}]=[g_{-1}\cdot g_{1}\cdot g_{-1}]\\
	&\hspace{-30pt}[h_{n+1}]=[h_{n}\cdot h_{n-1}\cdot h_{n}]\qquad\forall\: \hspace{7.3pt}2\leq n< \cn_+\qquad\quad\hspace{0.6pt}\text{and}\qquad\quad\hspace{6.3pt} [h_2]=[g_1\cdot g_{-1}\cdot g_1]
\end{split}
\end{align}
\item
We have the following identities:
\begin{align}
\label{awee}
\begin{split}
	[h_n]&=[g_{-1}\cdot (g_{1}\cdot g_{-1})^{|n|-1}]\qquad\quad \forall\: \cn_-\leq n\leq -1\\
	[h_n]&=[g_1\cdot (g_{-1}\cdot g_1)^{n-1}]\qquad\quad\hspace{11pt} \forall\:\hspace{7pt} 1\leq n\leq \cn_+
\end{split}
\end{align}
\begin{proof}
We first observe the following:
\begingroup
\setlength{\leftmarginii}{12pt}
\begin{itemize}
\item[$*$]   
Since $g_{\pm 1}\in \overlap(S)$ holds, Lemma \ref{fdfdfdfd} yields
\begin{align}
\label{qforrme2}
	\hspace{2pt}\:[g_{\pm 1}\cdot q\cdot g_{\pm 1}]\stackrel{\eqref{safgrfgtr}}{=}[g_{\pm1}\cdot q\cdot  g_{\pm 1}^{-1}]\stackrel{\eqref{stabiconji}}{=}[e]
	\qquad\quad\forall\: q\in G_S\subseteq \overlap(S).
\end{align}
It follows inductively from Lemma \ref{fdfdfdfd} and \eqref{qforrme2} that
\begin{align}
\label{qforrmel}
	q^\pm_{n}:=(g_{\mp 1}\cdot g_{\pm 1})^{n}\cdot (g_{\pm 1}\cdot g_{\mp1})^{n}\in G_S\qquad\quad\forall\: n\in \NN.
\end{align}
\item[$*$]
We have the implications:
\vspace{-21pt}
\begin{align}
\label{kcfdjksdfkjds}
\begin{split}
\qquad\qquad\quad \cn_-\leq -2\qquad\quad&\Longrightarrow\qquad\quad g_{1}\cdot g_{-1} \in \overlap(S)\\
\cn_+\geq \phantom{-}2\qquad\quad&\Longrightarrow\qquad\quad g_{-1}\cdot g_{1}\in \overlap(S).
\end{split}
\end{align}
{\it Proof of \eqref{kcfdjksdfkjds}.}
We only prove the second implication, because the first one follows analogously: 
\begingroup
\setlength{\leftmarginiii}{14pt}
\begin{itemize}
\item[$-$] 
According to the definition of a $\Sigma$-decomposition, 
there exists a boundary segment $\Sigma'$ of $\Sigma_0$ with   
\begin{align}
\label{fdfdfdnmnmcnmvc}
g_{-1}\cdot \iota(\Sigma')=\iota(\Sigma_{-1})\qquad\quad\text{hence}\qquad\quad	\iota(\Sigma')\stackrel{\eqref{safgrfgtr}}{=}g_{-1}\cdot\iota(\Sigma_{-1}).
\end{align}
\item[$-$] 
Since $\cn_+\geq 2$ holds, we have $g_1\cdot \iota(\Sigma_0)=\iota(\Sigma_1)$ with $g_1\in G_{z_1}$, hence $g_1\cdot \iota(z_{-1})=\iota(z_2)$. Thus,  
\vspace{-3pt}
\begin{align*}
	g_1\cdot \iota(\Sigma_0)= \iota(\Sigma_1)\quad&\stackrel{\text{Lemma } \ref{lemma:BasicAnalytt1}}{\Longrightarrow}\quad \hspace{8pt} g_1\cdot \iota|_{\Sigma_{-1}}\cpsim \iota|_{\Sigma_2} \hspace{45.5pt}\qquad\stackrel{\eqref{safgrfgtr}}{\Longrightarrow}\qquad  \iota|_{\Sigma_{-1}}\cpsim g_1\cdot\iota|_{\Sigma_2}\\
	&\hspace{8pt}\Longrightarrow\qquad  g_{-1}\cdot \iota|_{\Sigma_{-1}} \cpsim (g_{-1}\cdot g_1)\cdot \iota|_{\Sigma_2}\qquad\hspace{-1.2pt}\stackrel{\eqref{fdfdfdnmnmcnmvc}}{\Longrightarrow}\qquad\hspace{18pt} \iota\cpsim (g_{-1}\cdot g_1)\cdot \iota,
\end{align*}
which proves the claim.
\hspace*{\fill}$\ddagger$
\end{itemize}
\endgroup
\end{itemize}
\endgroup
Now, \eqref{awee} is clear for 
\begingroup
\setlength{\leftmarginii}{11pt}
\begin{itemize}
\item 
$n=\pm 1$; by $h_{\pm 1}=g_{\pm 1}$.  
\vspace{3pt}
\item   
$n=2$\: if\: $\cn_-\leq -2$\: holds, as well as for\: $n=2$\: if\: $2\leq \cn_+$ holds; by the right  side of \eqref{twerfdff}. 
\end{itemize}
\endgroup
In particular, if $\cn_+\geq 3$ holds (the case $\cn_-\leq -3$ is treated analogously), there exists some $2\leq m< n_+$ such that \eqref{awee} holds for all $1\leq n\leq m$; and we can argue by induction: 

\vspace{6pt}

\noindent
Let  $\alpha:= g_1\cdot (g_{-1}\cdot g_1)^{m-2}$  and $\beta:= g_1\cdot (g_{-1}\cdot g_1)^{m-1}$.  
By induction hypothesis, we have $h_{m-1}=\alpha\cdot q$ and  $h_m=\beta\cdot q'$ for certain $q,q'\in G_S$.  We obtain
\begin{align}
\label{dzuzufdssdufds}
\begin{split}
\:[h_{m-1}\cdot h_{m}]&\stackrel{\eqref{qwepokfdjkhfd}}{=}
[\alpha\cdot \beta]
=[(g_1\cdot q_{m-2}^+)\cdot (g_1\cdot g_{-1}\cdot g_1)]
\stackrel{\eqref{twerfdff}}{=} [(g_1\cdot q_{m-2}^+)\cdot h_2]\\
&\stackrel{\eqref{qwepokfdjkhfd}}{=}[g_1\cdot h_2]
\stackrel{\eqref{twerfdff}}{=}[g_1^2\cdot (g_{-1}\cdot g_1)]
\stackrel{\eqref{qwepokfdjkhfd},\he \eqref{qforrme2},\he \eqref{kcfdjksdfkjds}}{=}[g_{-1}\cdot g_1].
\end{split}
\end{align}
Specifically, in the last step, we have used that $g_1^2\in G_S$ holds by \eqref{qforrme2}, that $e\in \overlap(S)$ holds, and that $(g_{-1}\cdot g_1)\in \overlap(S)$ holds by \eqref{kcfdjksdfkjds}.  
We obtain
\vspace{-3pt}
\[
	[h_{m+1}]\stackrel{\eqref{twerfdff}}{=}[h_m\cdot (h_{m-1}\cdot h_m)]\stackrel{\eqref{dzuzufdssdufds}}{=}[h_m\cdot (g_{-1}\cdot g_1)]\stackrel{\eqref{qwepokfdjkhfd},\he\eqref{awee},\he \eqref{kcfdjksdfkjds}}{=}[g_1\cdot (g_{-1}\cdot g_1)^m].\qedhere
\]
\end{proof}
\end{itemize}
\endgroup
\noindent
We are ready for the proof of Lemma \ref{sdsdsdsdffghhh}.\ref{sdsdsdsdffghhh1}:
\begin{proof}[Proof of Lemma \ref{sdsdsdsdffghhh}.\ref{sdsdsdsdffghhh1}]
Clearly, \eqref{sdsdffghhh1} holds for $n=\pm 1$; so that we can argue by induction: 
\vspace{4pt}

\noindent
Assume  that \eqref{sdsdffghhh1} holds  
for $n=1,\dots,m$, for some $1\leq m < \cn_+$. We set   $\alpha:= g_1\cdot (g_{-1}\cdot g_1)^{m}$, and observe that 
$h_{m+1}= \alpha\cdot q$ holds for some $q\in G_S$  by \eqref{awee}, hence
\begin{align}
\label{lkdskldslkdslklkdslkdslkdslkdsids98d09ds09ds0909ds09dsds}
[g_{m+1}]\stackrel{\eqref{dsndsnmdsnmdsjdskjdskjdsoidsoidsoisdds98ds9898ds98dsdsdd}}{=}[h_{m+1}\cdot g_m]\stackrel{\eqref{qwepokfdjkhfd}}{=}
[\alpha\cdot g_m].
\end{align} 
\begingroup
\setlength{\leftmargini}{11pt}
\begin{itemize}
\item
If $m=2\cdot k$ is even, then $g_m\stackrel{\eqref{sdsdffghhh1}}{=}(g_1\cdot g_{-1})^k$ holds by the induction hypothesis; hence, 
\begin{align*}
	[g_{m+1}]&\stackrel{\eqref{lkdskldslkdslklkdslkdslkdslkdsids98d09ds09ds0909ds09dsds}}{=}[\alpha\cdot g_m]=
	[g_1\cdot (g_{-1}\cdot g_1)^{2k}\cdot (g_1\cdot g_{-1})^k]
	=[g_1\cdot (g_{-1}\cdot g_1)^{k}\cdot q_k^+]\stackrel{\eqref{qforrmel}}{=}[g_{\sigma(1)}\cdot {\dots}\cdot g_{\sigma(m+1)}].
\end{align*} 
\vspace{-13pt}
\item
If $m=2k+1$ is odd, then $g_m\stackrel{\eqref{sdsdffghhh1}}{=}(g_1\cdot g_{-1})^k\cdot g_1$ holds by the induction hypothesis; hence,
\begin{align*}
[g_{m+1}]&	\stackrel{\eqref{lkdskldslkdslklkdslkdslkdslkdsids98d09ds09ds0909ds09dsds}}{=}[\alpha\cdot g_m]
=
[g_1\cdot(g_{-1}\cdot g_{1})^{2k+1}\cdot (g_{1}\cdot g_{-1})^k \cdot g_1]\\
&	\stackrel{\phantom{\eqref{lkdskldslkdslklkdslkdslkdslkdsids98d09ds09ds0909ds09dsds}}}{=}[(g_1\cdot g_{-1})^{k+1}\cdot g_1\cdot q_k^+\cdot g_1]
	\stackrel{\eqref{qforrmel},\he \eqref{qforrme2}}{=}[(g_{1}\cdot g_{-1})^{k+1}]=[g_{\sigma(1)}\cdot {\dots}\cdot g_{\sigma(m+1)}].
\end{align*}
\end{itemize}
\endgroup
\noindent
It thus follows by induction that \eqref{sdsdffghhh1} holds for all $1\leq n\leq \cn_+$; and, an analogous argumentation shows that \eqref{sdsdffghhh1} also holds for all   
$\cn_-\leq n\leq -1$.\qedhere 
\end{proof}

\section*{Acknowledgements}
This work was supported in part by the Alexander von Humboldt foundation of Germany, and NSF Grants PHY-1205968 and PHY-1505490.

% \bibliographystyle{alpha}  
% \bibliography{dqKlops} 
\end{document}